\documentclass[10pt]{article}%
\usepackage{placeins}
\usepackage{tikz}
\usepackage{pgf}
\usetikzlibrary{arrows}
\usetikzlibrary{calc}
\usetikzlibrary{decorations.pathmorphing}
\usepackage{xfrac}

\usepackage{faktor}
\usepackage{soul}
\setstcolor{red}
\setulcolor{red}

\usepackage{rotating} % needed for \begin{sidewaysfigure}\begin{sidewaystable}

\usepackage{mathrsfs}
\DeclareMathAlphabet{\mathpzc}{OT1}{pzc}{m}{it}

\usepackage{color}
\usepackage{amsmath}
\usepackage{graphicx}
\usepackage{graphics}
\usepackage{float}
\usepackage{amsfonts}
\usepackage{amssymb}%
\usepackage{latexsym}
\usepackage{psfrag}
\usepackage{subfigure}
\usepackage{accents}

\newtheorem{theorem}{Theorem}[section]
\newtheorem{informaltheorem}[theorem]{Informal Theorem}
\newtheorem{corollary}[theorem]{Corollary}
\newtheorem{definition}[theorem]{Definition}
\newenvironment{proof}[1][Proof]{\noindent \emph{#1.} }
{\hfill \ \rule{0.5em}{0.5em}}
\newtheorem{lemma}[theorem]{Lemma}
\newtheorem{proposition}[theorem]{Proposition}

\numberwithin{equation}{section}
\numberwithin{table}{section}
\numberwithin{figure}{section}

\usepackage[a4paper]{geometry}
\newtheorem{remark}[theorem]{Remark}
\newtheorem{example}[theorem]{Example}

\newcommand{\R}{\mathbb{R}}

\newcommand{\cI}{{\cal I}}

\newcommand{\cP}{{\cal P}}

\newcommand{\cR}{{\cal R}}

\newcommand{\cD}{{\cal D}}
\newcommand{\cV}{{\cal V}}
\newcommand{\cW}{{\cal W}}

\newcommand{\cY}{{\cal Y}}

\newcommand{\cZ}{{\cal Z}}

    \newcommand\quotient[2]{
        \mathchoice
            {% \displaystyle
                \text{\raise1ex\hbox{$#1$}\Big/\lower1ex\hbox{$#2$}}%
            }
            {% \textstyle
                #1\,/\,#2
            }
            {% \scriptstyle
                #1\,/\,#2
            }
            {% \scriptscriptstyle  
                #1\,/\,#2
            }
    }

    \makeatletter
\newcommand{\nameditem}[1]{%
  \item[#1]%
  \def\@currentlabel{#1}%
}
\makeatother

\newcommand{\ri}{{\rm i}}

\usepackage{enumitem}
\newcommand{\beq}{\begin{equation}}
\newcommand{\eeq}{\end{equation}}
\newcommand{\beqs}{\begin{equation*}}
\newcommand{\eeqs}{\end{equation*}}
\newcommand{\bit}{\begin{itemize}}
\newcommand{\eit}{\end{itemize}}
\newcommand{\ben}{\begin{enumerate}}
\newcommand{\een}{\end{enumerate}}
\newcommand{\bal}{\begin{align}}
\newcommand{\eal}{\end{align}}
\newcommand{\bals}{\begin{align*}}
\newcommand{\eals}{\end{align*}}
\newcommand{\bse}{\begin{subequations}}
\newcommand{\ese}{\end{subequations}}
\newcommand{\bpr}{\begin{proposition}}
\newcommand{\epr}{\end{proposition}}
\newcommand{\bre}{\begin{remark}}
\newcommand{\ere}{\end{remark}}
\newcommand{\bpf}{\begin{proof}}
\newcommand{\epf}{\end{proof}}
\newcommand{\ble}{\begin{lemma}}
\newcommand{\ele}{\end{lemma}}
\newcommand{\bco}{\begin{corollary}}
\newcommand{\eco}{\end{corollary}}
\newcommand{\bex}{\begin{example}}
\newcommand{\eex}{\end{example}}
\newcommand{\bth}{\begin{theorem}}
\newcommand{\enth}{\end{theorem}}

\def\XXint#1#2#3{{\setbox0=\hbox{$#1{#2#3}{\int}$}
     \vcenter{\hbox{$#2#3$}}\kern-.5\wd0}}

\usepackage{hyperref}
\counterwithin{figure}{section}
\definecolor{myblue}{rgb}{0,0,0.6}
\hypersetup{colorlinks=true,linkcolor=myblue,citecolor=myblue,filecolor=myblue,urlcolor=myblue}
\newcommand*{\N}[1]{\left\|#1\right\|}

\allowdisplaybreaks[4]
\newcommand{\tfa}{\text{ for all }}
\newcommand{\tfor}{\text{ for }}

\newcommand{\tin}{\text{ in }}

\newcommand{\tand}{\text{ and }}

\newcommand{\vertiii}[1]{{\left\vert\kern-0.25ex\left\vert\kern-0.25ex\left\vert #1
    \right\vert\kern-0.25ex\right\vert\kern-0.25ex\right\vert}}

\definecolor{jwcol}{RGB}{27, 137, 18}  %{rgb}{1,0.88,0.21} changed color for visibility (david)

\definecolor{dalcol}{rgb}{0.8,0,0}

\definecolor{escol}{rgb}{0,0,0.8}
\definecolor{estcol}{rgb}{0,0.5,0}
\definecolor{esnewcol}{rgb}{0,0.5,0}

\newcommand{\supp}{{\rm supp } }

\newcommand{\mythmname}[1]{\textbf{\emph{(#1)}}}

\newcommand{\tr}{{\rm tr}}

\newcommand{\e}{\epsilon}

\newcommand{\RPMLo}{R_{\rm PML, -}}

\newcommand{\bV}{\mathbf{V}}

\newcommand{\matrixI}{{\mathsf{I}}}

\newcommand{\fine}{\cV_h}
\newcommand{\coarse}{\cV_{\Hc}}

\newcommand{\subdomain}{\cV_{\ell,h}}

\newcommand{\matrixD}{\mathsf{D}}
\newcommand{\matrixA}{\mathsf{A}}

\newcommand{\matrixAzero}{\matrixA_{0}}

\newcommand{\matrixR}{\mathsf{ R}}
\newcommand{\matrixM}{\mathsf{ M}}
\newcommand{\matrixB}{\mathsf{ B}}

\newcommand{\chiellg}{\chi_\ell^{(1)}}
\newcommand{\chiellgg}{\chi_\ell^{(2)}}
\newcommand{\chiellggg}{\chi_\ell^{(3)}}
\newcommand{\chiellgggg}{\chi_\ell^{(4)}}
\newcommand{\chiell}{\chi_\ell}
\newcommand{\Hilbert}{\cV}

\newcommand{\Hc}{{H_{\rm c}}}
\newcommand{\Qcoarse}{Q^0_{\Hc}}
\newcommand{\Qsub}{Q^\ell_h}

\newcommand{\Zspace}[1]{\mathcal{Z}^{#1}}
\newcommand{\Hspace}[1]{\mathcal{H}^{#1}}

\newcommand{\ZcupHspace}[1]{\mathcal{Y}^{#1}}
\newcommand{\ZcapHspace}[1]{\mathcal{W}^{#1}}
\makeatletter
\newcommand{\settheoremtag}[1]{% \settheoremtag{<tag>}
  \let\oldthetheorem\thetheorem% Store \thetheorem
  \renewcommand{\thetheorem}{#1}% Redefine it to a fixed value
  \g@addto@macro\endtheorem{% At \end{theorem}, ...
    \addtocounter{theorem}{-1}% ...restore theorem counter value and...
    \global\let\thetheorem\oldthetheorem}% ...restore \thetheorem
  }
\makeatother

\usetikzlibrary{positioning}
\usepackage{lscape}

\usepackage{pgfplots}
\usepackage{pgfplotstable}
\usepackage{etoolbox} % For conditional logic
\pgfplotsset{compat=newest}
\usepgfplotslibrary{groupplots}
\usepgfplotslibrary{polar}
\usepgfplotslibrary{smithchart}
\usepgfplotslibrary{statistics}
\usepgfplotslibrary{dateplot}
\usepgfplotslibrary{ternary}
\usepackage{ifthen}
\usepgfplotslibrary{fillbetween}

\pgfplotscreateplotcyclelist{plt-styles-7}{
    green!70!black,  solid, line width=1.1pt\\
    cyan!70!black,   solid, line width=1.1pt\\
    blue!80!black,   solid, line width=1.1pt\\
    violet!75!black, solid, line width=1.1pt\\
    orange!90!black, solid, line width=1.1pt\\
    red!75!black,    solid, line width=1.1pt\\
}

\newcommand{\plotmethodtestv}[4]{
  \addplot+[
    line width={0.5}, mark size={1.25pt},
  ] table[x expr=\coordindex, y=#1, col sep=comma]{#4};
  \addlegendentry{#3}
}

\newcommand{\generatekv}[4]{
\plotmethodtestv{#2}{01}{$k$=#3}{#4_k#3.csv}
}

\newcommand{\generateplot}[4]{
\ifthenelse{\equal{#1}{00}} {
    \generatekv{CAP}{#4}{100} {csvs/#2}
    \generatekv{CAP}{#4}{200} {csvs/#2}
    \generatekv{CAP}{#4}{400} {csvs/#2}
    \generatekv{CAP}{#4}{600} {csvs/#2}
    \generatekv{CAP}{#4}{800} {csvs/#2}
    \generatekv{CAP}{#4}{1000} {csvs/#2}
}
{
    \ifthenelse{\equal{#1}{0X}}
    {
        \generatekv{CAP}{#4}{100.077} {csvs/#2}
        \generatekv{CAP}{#4}{200.154} {csvs/#2}
        \generatekv{CAP}{#4}{399.924} {csvs/#2}
        \generatekv{CAP}{#4}{600.078} {csvs/#2}
        \generatekv{CAP}{#4}{799.849} {csvs/#2}
        \generatekv{CAP}{#4}{1000.005} {csvs/#2}
    }
    {
        \ifthenelse{\equal{#1}{0Y}}
        {
            \generatekv{CAP}{#4}{103.092} {csvs/#2}
            \generatekv{CAP}{#4}{194.729} {csvs/#2}
            \generatekv{CAP}{#4}{400.912} {csvs/#2}
            \generatekv{CAP}{#4}{595.641} {csvs/#2}
            \generatekv{CAP}{#4}{801.824} {csvs/#2}
            \generatekv{CAP}{#4}{996.553} {csvs/#2}
        }
        {
            \generatekv{CAP}{#4}{97.196} {csvs/#2}
            \generatekv{CAP}{#4}{194.391} {csvs/#2}
            \generatekv{CAP}{#4}{404.982} {csvs/#2}
            \generatekv{CAP}{#4}{599.374} {csvs/#2}
            \generatekv{CAP}{#4}{793.765} {csvs/#2}
            \generatekv{CAP}{#4}{1004.36} {csvs/#2}
        }
    }
}
}

\newcommand{\generatepic}[4]{
\centering
\makebox[\textwidth][c]{%
\begin{tikzpicture}[scale=.8]
\begin{groupplot}[
    group style={
    group name=mygroup,
        group size=2 by 1,
        horizontal sep=1.5cm
    },
    width=0.48\textwidth,
    height=6cm,
    cycle list name={plt-styles-7},
    xmajorgrids,
    ymajorgrids,
    legend cell align={left},
    legend pos={outer north east},
    legend style={{nodes={scale=0.8, transform shape}}},
    xlabel={GMRES iteration},
    ylabel={residual},
    ymode={log},
    log basis y={10},
    y filter/.expression={\pgfmathresult < -10 ? nan : \pgfmathresult},
    xmin={0},
    xmax={40},
    ymin={1e-11},
    ymax={1e2},
]

\nextgroupplot[
    xlabel={Fixed point iteration},
    legend style={draw=none, opacity=0}
]

\generateplot{#1}{#3}{#4}{CAP#2}

\nextgroupplot[
    xlabel={GMRES iteration},
    ylabel={},
    title={}
]

\generateplot{#1}{#3}{#4}{CAPGMRES#2}

\end{groupplot}
\node[anchor=south,yshift=2mm]
    at ($(mygroup c1r1.north)!0.5!(mygroup c2r1.north)$)
    {{\bf{#4}}};

\end{tikzpicture}}
}

\newcommand{\PMLs}{g}

\definecolor{jeffColor}{RGB}{102, 0, 204}

\title{Two-level hybrid Schwarz preconditioners with preasymptotic piecewise-polynomial coarse spaces for the high-frequency Helmholtz equation
}

\author{
J.~Galkowski\thanks{Department of Mathematics, University College London, 25 Gordon Street, London, WC1H 0AY, UK,   \tt J.Galkowski@ucl.ac.uk},
\and
E.~A.~Spence\thanks{Department of Mathematical Sciences, University of Bath, Bath, BA2 7AY, UK, \tt E.A.Spence@bath.ac.uk },
\and
P.~H.~Tournier\thanks{
Sorbonne Universit\'e, Universit\'e Paris Cit\'e, CNRS, INRIA, Laboratoire Jacques-Louis Lions, LJLL, EPC ALPINES, 4 place Jussieu, Paris, F-75005, France
\tt 
pierre-henri.tournier@sorbonne-universite.fr}
}

\date{\today}

\begin{document}
\maketitle

\begin{abstract}
We consider two-level hybrid Schwarz domain-decomposition preconditioners for finite element discretizations of the high-frequency Helmholtz equation (with wavenumber $k$). To date, all existing $k$-explicit theory for these preconditioners considers coarse spaces in the \emph{asymptotic regime}, i.e., where the Galerkin solution is $k$-uniformly quasi-optimal. 

We consider the situation where the fine and coarse spaces consist, respectively, of piecewise polynomials of degree $p$ on meshes of width $h$ and $H_c$. We prove results for the coarse space in the \emph{preasymptotic regime}, under the familiar mesh threshold that $(\Hc k)^{2p} \rho(k)$ is sufficiently small, where $\rho(k)$ is the norm of the solution operator (such that $\rho(k)\sim k$ for the  Helmholtz equation in free space). We prove that, in this regime, the fixed-point iteration with this preconditioner converges in a number of iterations that grows at most like $\log k$.

We give numerical experiments demonstrating this theory.
Furthermore, these experiments indicate that our theoretical results are sharp, in that
the number of iterations required for these preconditioners to converge with 
$(\Hc k)^{2p} \rho(k) \gg 1$ grows rapidly with $k$. 
\end{abstract}

\section{Context and informal statement of the main result}

Coarse grids are 
the key to parallel scalability of domain-decomposition (DD) methods for self-adjoint positive-definite problems (such as Laplace's equation); see, e.g., \cite{ToWi:05, Wi:09}, \cite[Chapter 4]{DoJoNa:15}. 
% % However, in the context of high-frequency wave problems, such as the high-frequency Helmholtz equation, the development and theory of two-level DD methods  (see, e.g., the recent computational studies \cite{BoDoJoTo:21,DoFrLaPaTo:26} and the references therein). 
However, 
both the development of and the frequency-explicit theory of two-level DD methods for high-frequency wave problems, such as the high-frequency Helmholtz equation, are longstanding open problems 
(see, e.g., the recent computational studies \cite{BoDoJoTo:21,DoFrLaPaTo:26} and the references therein). 

The six recent papers   \cite{HuLi:24, LuXuZhZo:24, MaAlSc:24, FuGoLiWa:24, GS4, GrSp:25} 
considered preconditioning finite-element discretisations of the Helmholtz equation with two-level Schwarz DD, and 
gave sufficient conditions for either the fixed-point method or GMRES (depending on the paper) 
to converge in a $k$-independent number of iterations (working in an appropriate $k$-weighted norm). All these papers consider coarse spaces in the \emph{asymptotic} regime, i.e., where the Galerkin solution in the coarse space is $k$-uniformly quasi-optimal (with this terminology famously introduced in \cite{IhBa:95a, IhBa:97})

The present paper considers fine and coarse spaces consisting of piecewise polynomials of degree $p$ on meshes of width $h$ and $H_c$, respectively. Piecewise-polynomial coarse spaces for high-frequency wave problems have previously been studied empirically in 
\cite{GSV1, GSV2, BoDoGrSpTo:17c, DDHelmholtz, DDMaxwell, BoDoJoTo:21, SiHeDwVu:26}.
In contrast to the existing two-level analyses mentioned above, we consider the situation where the coarse 
 space is in the \emph{preasymptotic} regime; i.e., the Galerkin solution in the coarse space has $k$-uniformly bounded relative error for oscillatory data 
\cite{IhBa:97,DuWu:15, GS3} but is not necessarily $k$-uniformly quasi-optimal.
Specifically, we consider coarse spaces satisfying $(H_c k)^{2p}\rho(k)\ll 1$,
where $\rho(k)$ is the norm of the Helmholtz solution operator (and such that $\rho(k) \sim k$ for the Helmholtz equation in  free space); here and in the rest of the paper 
we use the notation that $A\ll B$ if $A\leq c B$ for a small constant $c$.

%
% with this the familiar mesh threshold needed to guarantee controllably-small relative error of the finite-element solution as $k\to \infty$ \cite{IhBa:97,DuWu:15, GS3}. % and defines the \emph{preasymptotic regime}.

In its simplest-possible form, the main result of the paper is the following.

\begin{informaltheorem}[Informal statement of the main result (Theorem \ref{t:main})]\label{t:informal} 
Consider the Helmholtz scattering problem in the exterior of a smooth, bounded Dirichlet or Neumann obstacle, with piecewise smooth coefficients (modelling a penetrable obstacle), and with an absorbing layer approximating the radiation condition. Suppose that $\rho(k)$ is polynomially bounded in $k$
 (as occurs for ``most" $k$ by \cite{LSW1}).

Consider finite-element discretisations of this problem in a space of mapped piecewise polynomials with meshwidth $h$
%, where the finite-element mesh resolves the boundaries of the obstacles and interfaces across which the coefficients jump.
and a hybrid Schwarz preconditioner with the following properties:
\begin{itemize}
    \item the subdomains are of size $\sim k^{-1}$ with overlaps $\sim k^{-1}$,
    \item the subdomain problems have absorption on the subdomain boundaries that are not part of the obstacle boundary, and
    \item the piecewise-polynomial coarse space satisfies $(\Hc k)^{2p} \rho(k)\ll 1$, and $h\ll \Hc$. 
\end{itemize}
Then the preconditioned fixed point method (and hence also GMRES) converges in a number of iterations 
%$\lesssim\log [1+\rho (k)(H_ck)^p] + 
$\lesssim \log k$.
\end{informaltheorem}

% sthe condition $(H_c k)^{2p}\rho(k)\ll 1$ required in Informal Theorem~\ref{t:informal}/Theorem \ref{t:main} is the familiar mesh threshold needed to guarantee controllably-small relative error of the finite-element solution as $k\to \infty$ \cite{IhBa:97,DuWu:15, GS3}. % and defines the \emph{preasymptotic regime}.
%(Recall that $\rho(k)$ is polynomially bounded in $k$ 
%for ``most" $k$ by \cite{LSW1}.)

% Furthermore, with $(\Hc k)^{2p} \rho(k)\sim 1$ and $\rho(k)$ polynomially bounded in $k$ (as occurs for ``most" $k$ by \cite{LSW1}), the number of iterations required in Informal Theorem~\ref{t:informal} grows at most like $\log k$.

In such a two-level preconditioner applied to the Helmholtz equation one would ideally like the coarse space to be small. However, 
the numerical experiments 
 in \S\ref{s:numerics} below 
indicate that the sufficient conditions on $\Hc$ in Informal Theorem~\ref{t:informal}/Theorem \ref{t:main} are \emph{optimal} in the sense that for piecewise polynomial coarse spaces with $(H_ck)^{2p}\rho(k)\gg 1$
the number of GMRES iterations grows quickly with $k$,
% or the number of GMRES iterations to be $k$-uniformly bounded.
% That is, we give numerical experiments where 
%  $(H_ck)^{2p}\rho(k)\gg 1$ 
%   and the number of GMRES iterations required for convergence of the preconditioned system grows significantly with $k$, 
  even when $h$ is chosen very small. 
  
The underlying reason for the requirement $(H_ck)^{2p}\rho(k)\ll 1$ is that a one-level DD method with subdomains of size $k^{-1}$  requires at least $\sim k$ iterations to resolve 
the propagation of the Helmholtz solution at length scales independent of $k$. This propagation must therefore be resolved by the coarse space, requiring 
the familiar condition
$(H_ck)^{2p}\rho(k)\ll 1$ for accuracy of the $h$-version FEM solution in the coarse space.

\section{Statement of the main result}

\subsection{Helmholtz problems}

The general Helmholtz-type PDE to which our theory applies is given by \eqref{e:PML1} below. We arrive at this PDE motivated by scattering problems with the radiation condition approximated by an absorbing layer, such as a perfectly matched layer (PML) \cite{Be:94, CoMo:98} or a complex absorbing potential (CAP) \cite{RiMe:93}.

To define the scattering problem, let $\Omega_-\subset \mathbb{R}^d$ be a bounded open obstacle with connected complement, $\Omega_+:=\mathbb{R}^d\setminus \overline{\Omega_-}$.
Furthermore, let $A$ be a symmetric, positive-definite matrix with real coefficients that equals the identity outside a compact set and let
$n$ be a strictly positive, real-valued function that equals one outside a compact set. 
We consider the Helmholtz scattering problem:~given $f\in L^2_{\rm comp}(\overline{\Omega_+})$, find $u\in H^1_{\rm loc}(\overline{\Omega_+})$ such that
\begin{equation}
\label{e:edp}
-k^{-2}\operatorname{div}( A\nabla u)-n u=f\text{ in }\Omega_+,\qquad (\mathcal{B}u)|_{\partial\Omega_+}=0,\qquad (k^{-1}\partial_r-i)u=o_{r\to \infty}(r^{\frac{1-d}{2}})
\end{equation}
where $\Gamma_s := \partial\Omega_+=\Gamma_{\textup{hard}}\sqcup \Gamma_{\textup{soft}} $ with $\Gamma_{\textup{hard}}$ and $\Gamma_{\textup{soft}}$ smooth, closed hypersurfaces and
$(\mathcal{B}u)|_{\Gamma_{\textup{soft}}}= u$ in the sound-soft case and $(\mathcal{B}u)|_{\Gamma_{\textup{hard}}}= \nu\cdot(A\nabla u)$, with $\nu$ the normal to $\Gamma_s$ %(relative to the matrix $A$) 
in the sound-hard case.

We then approximate the Sommerfeld radiation condition in \eqref{e:edp} using an absorbing layer. Let $\Omega_{\tr} \subset \mathbb{R}^d$ be open, bounded, connected and contain  
$\overline{\Omega_-}\cup \supp (A-I)\cup \supp (n-1)$. We truncate the problem~\eqref{e:edp} to the computational domain $\Omega:=\Omega_+\cap \Omega_{\tr}$ and apply the finite-element method to the problem:~given $f\in L^2(\Omega)$, find $u\in H^1(\Omega)$ such that
\begin{equation}
\label{e:PML1}
Pu:=-k^{-2}\operatorname{div}(\tilde{A}\nabla u)-ik^{-1} \tfrac{1}{2}(\tilde{b}\cdot\nabla u+\operatorname{div}(\tilde{b}u))-\tilde{n} u=f\text{ in }\Omega,\quad (\tilde{\mathcal{B}}u)|_{\partial\Omega_+}=0,\quad u|_{\partial\Omega_{\tr}}=0,
\end{equation} 
where $\tilde{B}u|_{\Gamma_{\rm soft}}=u$, $\tilde{B}u|_{\Gamma_{\rm hard}}=\nu \cdot 
(\tilde{A}k^{-1}\nabla u+\frac{i}{2} \tilde{b}u)$ and we assume that $\tilde{A}$ is symmetric.
(Note that if $\tilde{A}$, $\tilde{b}$, and $\tilde{n}$ are real then \eqref{e:PML1} is self adjoint -- this is the reason we have written the Helmholtz PDE in this form). When approximating solutions to~\eqref{e:edp} using a perfectly matched layer (PML), $\tilde{A}$, $\tilde{b}$ and $\tilde{n}$ are respectively equal to $A$, $0$, and $n$ near $\supp (A-I)\cup (n-1)\cup \Omega_-$, differing only because of the addition of the PML.

We denote the sesquilinear form associated to~\eqref{e:PML1} by $a=a(k):\cZ\times \cZ\to \mathbb{C}$; i.e., 
\begin{equation} \label{e:quadratic}
a(u,v):=\langle \tilde{A}k^{-1}\nabla u,k^{-1}\nabla v\rangle_{\Omega}-\tfrac{i}{2}\langle \tilde{b}\cdot k^{-1}\nabla u,v\rangle_{\Omega}+\tfrac{i}{2}\langle  \tilde{b}u, k^{-1}\nabla v\rangle_{\Omega} -\langle \tilde{n}u,v\rangle_{\Omega},
\eeq
and 
\beq\label{e:cZ}
\cZ:=\big\{ u\in  H^1(\Omega)\,:\, u|_{\partial\Omega_{\tr}}=0,\, u|_{\Gamma_{\rm soft}}=0\big\}
\eeq
with norm
\beqs
\|u\|^2_{\cZ}:=\|u\|^2_{H^1_k(\Omega)}:= \|k^{-1}\nabla u\|^2_{L^2(\Omega
)}+\|u\|^2_{L^2(\Omega)}.
\eeqs
For a sesquilinear form $a:\cZ\times \cZ\to \mathbb{C}$, we denote by
$$
\rho(a)
:= \bigg(
\inf_{0\neq u \in \cZ} \sup_{0\neq v\in \cZ} \frac{|a(u,v)|}{\|u\|_{\cZ} \|v\|_{\cZ}}
\bigg)^{-1}
$$
the norm of the solution operator (a.k.a.~the continuous inf-sup constant).

\begin{definition}[Helmholtz problem with absorption and truncation]\label{d:Helmholtz}
    We say that the sesquilinear form $a$ in~\eqref{e:quadratic} is a piecewise $C^{p+1}$ Helmholtz problem with absorption and $C^r$ truncation if the following holds.
    %~\eqref{e:absorbing} holds, 

(i) There is $\theta_0\in\mathbb{R}$ such that  \beq\label{e:strongElliptic}
\inf_{k>k_0}\Re \big(e^{i\theta_0}\tilde{A}^{ij}(x)\xi_i\xi_j\big)\geq c|\xi|^2,\qquad x\in \Omega,\,\xi\in\mathbb{R}^d.
\eeq    

    (ii) $\partial\Omega_-\in C^{p,1}$, $\partial\Omega_{\tr}\in C^{r-1,1}$, and there is  $\Gamma_p\Subset \Omega$ a $C^{p,1}$, smooth hypersurface such that $\tilde{A}$, $\tilde{b}$, and $\tilde{n}$ are all $\overline{C^{p-1,1}}(\Omega\setminus\Gamma_p)$.

(iii) There is a neighborhood, $U$ of $\partial\Omega_{\tr}$ such that for all $k_0>0$ there are  $c>0$ and $\theta_1,\theta_2\in\mathbb{R}$ such that for $k\geq k_0$,
\begin{equation}
%\begin{gathered}
\label{e:absorbing}
\inf_{k>k_0}\Re \big(e^{i\theta_1}\big(\tilde{A}^{ij}(x)\xi_i\xi_j +\tilde{b}^i(x)\xi_i-\tilde{n}(x)\big)\big)\geq c(|\xi|^2+1),\qquad x\in U,\,\xi\in\mathbb{R}^d
\eeq
and
\beq\label{e:imagParts}
\Big(\supp \Im \big(e^{i\theta_2}\widetilde{A}\big)\cup \supp \Im \big(e^{i\theta_2}\widetilde{b}\big)\cup \supp \Im \big(e^{i\theta_2}\widetilde{n}\big)\Big)\cap \Big(\overline{(\Gamma_p\cup \partial\Omega_-)\setminus \overline{U}}\Big)=\emptyset.
\end{equation}
    \end{definition}
We immediately highlight that, by multiplying $P$ by a constant, one of $\theta_0,\theta_1$, and $\theta_2$ can be set to zero.        Point (i) of Definition \ref{d:Helmholtz} is that $P$ is strongly elliptic in $\Omega$ (in the sense of, e.g., \cite[Page 122]{Mc:00}).

    In Point (ii) of Definition \ref{d:Helmholtz}, 
    we use the notation that a function $u\in \overline{C^{r,1}}(V)$ if $u\in C^{r,1}(V)$ and 
    $$
    \|\partial^\alpha u\|_{L^\infty(V)}<\infty ,\qquad |\alpha|\leq r+1
    $$ (i.e., $u$ is $C^{r,1}$ ``up to the boundary" of $V$).

 In Point (iii) of Definition \ref{d:Helmholtz},
the condition \eqref{e:absorbing} corresponds to $P$ being ``absorbing" in $U$, and, with $\theta_2=0$, the condition \eqref{e:imagParts} says that any absorption in the coefficients away from $U$ is also away from the boundary of the scatterer $\partial\Omega_-$  and the interfaces $\Gamma_p$ (i.e., the Helmholtz problem is self adjoint near the parts of $\partial\Omega_-$ and $\Gamma_p$ that are away from $U$).

Note that the domain of the unique, closed, unbounded operator on $L^2$ generated by a piecewise $C^{1}$ Helmholtz problem with $C^1$ truncation is given by
$$
\Bigg\{u\in H^1(\Omega)\cap H^2(\Omega\setminus \Gamma_p)\,:\, \begin{gathered}\nu \cdot \big(\tilde{A}k^{-1}\nabla u+\tfrac{i}{2}\tilde{b}u\big)|_{\Gamma_{\rm hard}}=0,\, u|_{\Gamma_{\rm soft}\cup \Gamma_{\tr}=0},\\
\nu \cdot \big(\tilde{A}_+k^{-1}\nabla u_+-\tilde{A}_-k^{-1}\nabla u_-+\tfrac{i}{2}(\tilde{b}_+-\tilde{b}_-)u\big)|_{\Gamma_p}=0\end{gathered}\Bigg\},
$$
where the notation $\tilde{A}_{\pm}$, $\tilde{b}_{\pm}$, and $u_{\pm}$ denotes restriction from either side of $\Gamma_p$. 

\subsection{Components of the two-level hybrid Schwarz preconditioner}

The definition of the two-level hybrid Schwarz preconditioner involves the following components, which are defined precisely in the following subsections.
\begin{enumerate}
    \item A Helmholtz problem with absorption and truncation, $a$ (in the sense of Definition \ref{d:Helmholtz} above).
    %    a piecewise $C^{p+1}$ 
    %Helmholtz problem with %absorption and $C^r$ %truncation, $a$,
    \item Two open covers of $\Omega$, $\{\Omega_{\ell}\}_{\ell=1}^N$ and $\{W_\ell\}_{\ell=1}^N$ with 
    \begin{equation} \label{e:subdomainThing1}
    W_\ell \subset \Omega_\ell \quad\tand\quad\partial W_\ell\cap \partial\Omega_\ell\subset \partial \Omega.
    \end{equation}
(Here, $W_\ell$ will be where the Helmholtz problem on the $\ell$th subdomain equals the underlying Helmholtz problem; i.e., for $\Omega_\ell$ not touching $\partial \Omega$, $W_\ell$ is inside $\Omega_\ell$).
        \item Open sets $\{U_\ell\}_{\ell=1}^N$ with $U_\ell\subset \Omega_\ell$ and
    \begin{equation}
\label{e:subdomainThing2}
         \partial\Omega_{\ell}\setminus \partial\Omega \subset \overline{U}_{\ell}.
    \end{equation}
(Here, $U_\ell$ will be the absorbing part of the Helmholtz problem on the $\ell$th subdomain.)
    \item 
   Subdomain problems  
   on $\Omega_{\ell}$ 
   %$C^1$ 
   for $(a,W_\ell,U_\ell)$, $\{a_\ell\}_{\ell=1}^N$.
    \item A partition of unity $\{\chi_\ell\}_{\ell=1}^N$
    subordinate to $W_\ell$.
    \item Coarse and fine spaces, $\coarse$ and $\fine$ with $\coarse\subset \fine$.
    \item Maps $\mathcal{I}_\ell:H_{\partial\Omega}^1(\Omega_\ell)\cap \overline{C^0}(\Omega_\ell)\to \subdomain$ where
    \beq\label{e:H1partial}
    H^1_{\partial\Omega}(\Omega_\ell):=\big\{u\in H^1(\Omega_\ell)\,:\, u|_{\partial\Omega_\ell\setminus \partial\Omega}=0,\, u|_{\Gamma_{\rm soft}}=0,\, u|_{\Gamma_{\tr}}=0\big\},\quad \subdomain:=\fine\cap H_{\partial\Omega}^1(\Omega_\ell)
    \eeq
(the standard choice of $\cI_\ell$ is nodal interpolation on Lagrange points).
\end{enumerate}
The geometric set up described here is illustrated in Figure~\ref{f:geometry}.

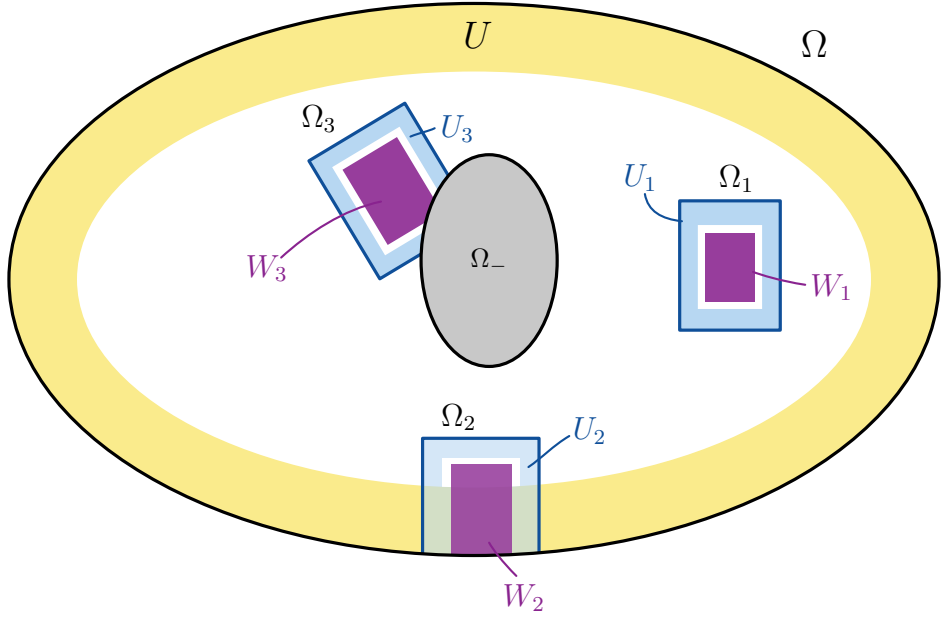
\begin{figure}
    \centering
    \definecolor{globalyellow}{RGB}{250,235,140}
\definecolor{localblue}{RGB}{15,79,153}
\definecolor{localbluefill}{RGB}{183,215,242}
\definecolor{wviolet}{RGB}{137,35,143}
\definecolor{light-gray}{RGB}{200,200,200}

\begin{tikzpicture}[line cap=round,line join=round]

% Smooth outer boundary of Omega and smooth inner boundary of the global layer U.
% \def\outerOmega{%
%   (-6.05,-.05)
%   .. controls (-6.15,1.72) and (-4.62,3.12) .. (-2.48,3.55)
%   .. controls (-.35,4.02) and (2.36,3.58) .. (4.10,2.58)
%   .. controls (5.57,1.74) and (6.32,.63) .. (6.12,-.55)
%   .. controls (5.90,-1.87) and (4.30,-2.70) .. (2.70,-3.13)
%   .. controls (1.85,-3.40) and (1.25,-3.50) .. (.86,-3.55)
%   .. controls (.72,-3.58) and (.67,-3.59) .. (.60,-3.59)
%   .. controls (.30,-3.63) and (-.12,-3.62) .. (-.42,-3.60)
%   .. controls (-.50,-3.59) and (-.60,-3.58) .. (-.68,-3.56)
%   .. controls (-1.15,-3.54) and (-1.80,-3.45) .. (-2.12,-3.38)
%   .. controls (-2.75,-3.25) and (-3.50,-2.92) .. (-4.05,-2.62)
%   .. controls (-5.45,-1.92) and (-6.18,-1.02) .. cycle}

% \def\innerOmega{%
%   (-5.25,-.08)
%   .. controls (-5.30,1.25) and (-4.12,2.35) .. (-2.48,2.72)
%   .. controls (-.58,3.14) and (.67,2.89) .. (1.92,2.30)
%   .. controls (3.25,1.71) and (4.62,1.67) .. (5.06,.57)
%   .. controls (5.57,-.58) and (4.38,-1.75) .. (2.91,-2.19)
%   .. controls (2.15,-2.45) and (1.25,-2.62) .. (.60,-2.70)
%   .. controls (.30,-2.74) and (-.12,-2.77) .. (-.42,-2.76)
%   .. controls (-1.45,-2.74) and (-2.55,-2.55) .. (-3.31,-2.18)
%   .. controls (-4.57,-1.66) and (-5.30,-.90) .. cycle}

\def\outerOmega{(0,.05) ellipse (6.15 and 3.65)}
\def\innerOmega{(0,.05) ellipse (5.25 and 2.75)}

% Smooth outer arc of a two-dimensional Helmholtz resonator.
\def\resonatorOuter{%
(.2,.3)ellipse(.9 and 1.4);
  % (-.24,-.84)
  % .. controls (-.95,-.62) and (-1.08,.40) .. (-.63,1.03)
  % .. controls (-.10,1.62) and (.72,1.64) .. (1.16,1.06)
  % .. controls (1.58,.45) and (1.36,-.58) .. (.64,-1)
  }

% Invisible mask for the entire resonator interior.  It prevents a subdomain
% drawn behind the resonator from reappearing inside its cavity.
\def\resonatorMask{%
  \resonatorOuter
  .. controls (.44,-1.03) and (-.04,-1.03) .. cycle}

% Smooth C-shaped wall.  The two neck tips and both circular arcs are joined
% only by Bezier curves, so partial Omega_- has no angular corners.
\def\resonator{%
  \resonatorOuter
  .. controls (.46,-.99) and (.32,-.70) .. (.48,-.54)
  .. controls (.84,-.27) and (.88,.35) .. (.48,.70)
  .. controls (.15,.98) and (-.25,.91) .. (-.4,.58)
  .. controls (-.70,.08) and (-.3,-.38) .. (-.08,-.54)
  .. controls (.08,-.66) and (-.06,-.93) .. (-.24,-.84)
  -- cycle}

% Complete (solid) shading of the global absorbing layer.
\path[fill=globalyellow,even odd rule] \outerOmega\space \innerOmega;

% Omega_3: a rotated right-angled box clipped by the resonator.  The covered
% lower-right portion is replaced by the coincident boundary of Omega_-.
\begin{scope}[shift={(-1.08,1.22)},rotate=31]
  \path[fill=localbluefill,even odd rule]
    (-.74,-.91) rectangle (.74,.91)
    (-.52,-.68) rectangle (.52,.68);
  \draw[localblue,very thick] (-.74,-.91) rectangle (.74,.91);
  %\draw[localblue,thick] (-.52,-.68) rectangle (.52,.68);
  % W_3 lies wholly inside the unshaded core, with a visible gap from U_3.
  % It still extends slightly into Omega_-; drawing the resonator later clips
  % it exactly, so the visible W_3 touches the shared physical boundary.
  \fill[wviolet!88]
    (-.42,-.58) rectangle (.42,.58);
\end{scope}

% Omega_1: upright rectangular subdomain on the right.
\path[fill=localbluefill,even odd rule]
  (2.72,-.62) rectangle (4.05,1.09)
  (2.96,-.34) rectangle (3.81,.77);
\draw[localblue,very thick] (2.72,-.62) rectangle (4.05,1.09);
%\draw[localblue,thick] (2.96,-.34) rectangle (3.81,.77);
\fill[wviolet!88]
  (3.05,-.24) rectangle (3.71,.67);

% Omega_2: its three artificial sides are straight and right-angled, while
% its fourth side is exactly the corresponding curved segment of dOmega.
% A translucent blue fill allows the yellow global layer to remain visible
% beneath the local layer in their overlap.
\begin{scope}
  \clip \outerOmega;
  \path[fill=localbluefill,fill opacity=.58]
    (-.68,-4.0) rectangle (.86,-2.05);
  \draw[localblue,very thick]
    (-.68,-4.0) -- (-.68,-2.05) -- (.86,-2.05) -- (.86,-4.0);
\end{scope}
% The white core is clipped to the inner ellipse. Thus yellow U and blue U_2
% extend precisely to the inner solid ellipse, with no hand-matched curve.
\begin{scope}
  \clip \innerOmega;
  \path[fill=white] (-.42,-4.0) rectangle (.60,-2.31);
  %\draw[localblue,thick]
   % (-.42,-4.0) -- (-.42,-2.31) -- (.60,-2.31) -- (.60,-4.0);
\end{scope}
% W_2 is clipped by Omega and therefore shares a nontrivial boundary segment.
\begin{scope}
  \clip \outerOmega;
  \fill[wviolet!88,opacity=.88] (-.3,-3.7) rectangle (.5, -2.4);
\end{scope}

% Redraw the two smooth boundaries so their geometry remains visually clear.
\draw[black,very thick] \outerOmega;
%\draw[black,thick] \innerOmega;

% Mask both the wall and cavity before drawing the resonator.  Consequently
% U_3 stops at the outside of the resonator and cannot enter its interior.
\path[fill=white] \resonatorMask;
% Drawing the smooth wall after Omega_3 makes their coincident boundary exact.
\path[fill=light-gray,draw=black,very thick] \resonator;

% Global labels.
\node[font=\Large] at (4.5,3.17) {$\Omega$};
\node[font=\Large] at (.05,3.28) {$U$};
\node[font=\normalsize] at (.2,.3) {$\Omega_-$};

% Labels and leader lines for Omega_1, U_1, W_1.
\node[font=\large] at (3.47,1.38) {$\Omega_1$};
\node[localblue,font=\large] at (2.2,1.4) {$U_1$};
\draw[localblue,thick]
  (2.26,1.18) .. controls (2.30,.95) and (2.48,.85) .. (2.75,.82);
\node[wviolet,font=\large] at (4.73,-.05) {$W_1$};
\draw[wviolet,thick]
  (4.38,-.02) .. controls (4.05,.03) and (3.80,.11) .. (3.69,.15);

% Labels and leader lines for Omega_2, U_2, W_2.
\node[font=\large] at (-.20,-1.77) {$\Omega_2$};
\node[localblue,font=\large] at (1.53,-1.95) {$U_2$};
\draw[localblue,thick]
  (1.26,-2.02) .. controls (1.04,-2.15) and (.90,-2.31) .. (.72,-2.43);
\node[wviolet,font=\large] at (.66,-4.16) {$W_2$};
\draw[wviolet,thick]
  (.56,-3.87) .. controls (.37,-3.68) and (.25,-3.52) .. (.20,-3.43);

% Labels and leader lines for the rotated Omega_3.
\node[font=\large] at (-2.05,2.18) {$\Omega_3$};
\node[localblue,font=\large] at (-.22,2.05) {$U_3$};
\draw[localblue,thick]
  (-.52,2.05) .. controls (-.6,2.01) and (-.72,1.99) .. (-.84,1.94);
\node[wviolet,font=\large] at (-2.74,.20) {$W_3$};
\draw[wviolet,thick]
  (-2.55,.34) .. controls (-2.20,.66) and (-1.72,.98) .. (-1.23,1.08);

\end{tikzpicture}
    \caption{\label{f:geometry}The diagram shows an example of a domain $\Omega $ with an obstacle $\Omega_-$ and its absorbing layer, $U$. Also shown are three subdomains, $\Omega_j$, $j=1,2,3,$ with their absorbing layers, $U_j$ and regions of agreement with the global problem, $W_j$.}
    \label{fig:placeholder}
\end{figure}

\subsection{The subdomain problems}

%Given a piecewise $C^{p+1}$ Helmholtz problem with absorption and $C^r$ truncation, $a$, we now define subdomain problems. 

We now define the subdomain problems. This definition requires somewhat complicated assumptions about the absorption at the boundaries of the subdomains (similar to those in the definition of the underlying Helmholtz problem; i.e., Definition \ref{d:Helmholtz}). However, we highlight that 
these conditions are all met 
if $\tilde{A}, \tilde{b}$, and $\tilde{n}$ are all real away from the absorption approximating the radiation condition (i.e., one starts from a problem with no absorption) and the absorption in the subdomain problems is a complex absorbing potential.

\begin{definition}[Subdomain problems]
Given a piecewise $C^{p+1}$ Helmholtz problem with absorption and $C^r$ truncation, $a$ \eqref{e:quadratic}, 
let $c_\ell>0$, $W_\ell\subset \Omega_\ell\subset \Omega$ and $U_\ell \subset \Omega_\ell$ open such that 
$$
\overline{W}_\ell \cap (\partial\Omega_{\ell}\setminus \partial\Omega)=\emptyset, \quad \partial\Omega_{\ell}\setminus \partial\Omega \subset \overline{U}_{\ell}.
$$ 
Let
$$
\cZ_\ell:=\big\{ u\in H^1(\Omega)\,:\, u|_{\Gamma_{\rm soft}}=0, \,\supp \,u\subset \overline{\Omega_{\ell}}\big\}.
$$
We say that $a_{\ell}:\cZ_\ell\times \cZ_{\ell}\to \mathbb{C}$ is a piecewise $C^{p+1}$ Helmholtz subdomain problem on $\Omega_\ell$ with absorption and Lipschitz truncation for $(a, W_\ell, U_\ell)$ on $\Omega_\ell$ with constant $c_\ell$ 
%and $\mu$-sized advection 
if $\partial\Omega_\ell\in C^{0,1}$
$$
a_\ell( u, v)=\langle A_\ell k^{-1}\nabla u,k^{-1}\nabla v\rangle_{\Omega_\ell}-\tfrac{i}{2}\langle b_\ell\cdot k^{-1}\nabla u,v\rangle_{\Omega_\ell}+\tfrac{i}{2}\langle  b_\ell u, k^{-1}\nabla v\rangle -\langle n_\ell u,v\rangle_{\Omega_ell},
$$
with $A_\ell,b_\ell,n_\ell\in \overline{C^{p-1,1}}(\Omega_\ell\setminus \Gamma_p)$, 
\begin{gather}\label{e:subdomainCoeff1}
\big(\supp (A_\ell-\tilde{A})\cup \supp (b_\ell-\tilde{b})\cup \supp (n_\ell-\tilde{n})\big)\cap W_\ell=\emptyset,\\ \text{\eqref{e:strongElliptic}  and   \eqref{e:absorbing}  hold with $(\tilde{A}, \tilde{b},\tilde{n},c,U,\Omega)$ replaced by $(A_\ell, b_\ell,n_\ell,c_\ell,U_\ell,\Omega_\ell)$},\label{e:absorbingUl}\\
\big(\supp \Im (e^{i\theta_2}A_\ell)\cup \supp \Im (e^{i\theta_2}b_\ell)\big)\cap (\overline{(\partial\Omega_\ell\cup \Gamma_p)\setminus \overline{U}_\ell})=\emptyset,\label{e:subdomainCoeff2}
\end{gather}
where $\theta_2$ is as in 
\eqref{e:imagParts}.

We sometimes write that $a_\ell$ is a subdomain problem on $\Omega_\ell$ for $a$ if there exist $U_\ell,W_\ell,c_\ell$ such that $a_\ell$ is a subdomain problem on $\Omega_\ell$ for $(a,W_\ell,U_\ell)$ with constant $c_\ell$. 
\end{definition}

Regarding the three conditions \eqref{e:subdomainCoeff1}, 
\eqref{e:absorbingUl},
\eqref{e:subdomainCoeff2}:~the conditions \eqref{e:subdomainCoeff1} and \eqref{e:absorbingUl} imply that the region $W_\ell$ is where the subdomain problem agrees with the global Helmholtz problem and $U_\ell$ is the part of $\Omega_\ell$ in which there is absorption. 
The condition~\eqref{e:subdomainCoeff2} with $\theta_2=0$ requires that 
%, when the original problem is self adjoint near $\Gamma_p$ and $\partial\Omega_-$ and 
when 
$\Omega_\ell$ touches $\Gamma_p$ or $\partial\Omega_-$, the subdomain absorption at the boundary of the absorbing layer %absorption in the subdomain problem 
must be implemented with a CAP (involving complex $\tilde{n}$) rather than a PML. 

%Finally,~\eqref{e:advection1} is used to guarantee that $a_\ell$ is, in fact, coercive on functions supported in $U_\ell$. For both PML or CAP applied to~\eqref{e:edp}, \eqref{e:advection1} is satisfied for any $\mu>0$ since, in these cases, $|b|\leq Ck^{-1}$ or $b=0$ respectively.

To state the main result below, we need a quantitative notion of the scale of functions, sesquilinear forms, and Lipschitz domains.

\begin{definition}[Scale of a function]
Let $\{D_\alpha\}_{\alpha\in\mathbb{N}^d}\subset (0,\infty)$. We say that the scale of $\chi\in \overline{C^{m-1,1}}(\Omega\setminus \Gamma_p)$ has with respect to $\{D_{\alpha}\}$ is 
$$
\delta(\chi,\{D_\alpha\},m):=\begin{cases} \Big(\sup_{1\leq|\alpha|\leq m}\big(D_{\alpha}^{-1}\|\partial^\alpha \chi\|_{L^\infty(\Omega\setminus\Gamma_p)}\big)^{\frac{1}{|\alpha|}}\Big)^{-1}&\text{ if }\|\chi\|_{L^\infty}\leq D_0\\
0&\text{else}.
\end{cases}
$$
\end{definition}

Observe that if $\chi
\in \overline{C^{m-1,1}}(\Omega\setminus \Gamma_p)$ has scale $\delta$ with respect to $\{D_\alpha\}$ then
$$
\|\partial^\alpha \chi\|_{L^\infty(\Omega\setminus\Gamma_p)}\leq D_\alpha \delta^{-|\alpha|}.
$$

Given the notion of the scale of a function, we define the scale of the sesquilinear form via the scale of its coefficients. 

\begin{definition}[Scale of a sesquilinear form]
Let $m\geq0$ and $\{D_\alpha\}_{\alpha\in\mathbb{N}^d}\subset (0,\infty)$. We say that the \emph{scale of a sesquilinear form $a$} given by
$$
a(u,v)=\langle Ak^{-1}\nabla u,k^{-1}\nabla v\rangle 
-\tfrac{i}{2}\langle b\cdot k^{-1}\nabla u,v\rangle+\tfrac{i}{2}\langle  b u, k^{-1}\nabla v\rangle
-\langle nu,v\rangle
$$
is 
$$
\delta(a,\{D_\alpha\},m+1):=\min\Big\{\min_{i,j}\delta(A^{ij},\{D_\alpha\},m),\,\min_{i}\delta(b^i,\{D_\alpha\},m),\,\delta(n,\{D_\alpha\},m)\Big\}.
$$
\end{definition}

We now define precisely the notion of the scale of a domain. Roughly speaking, a Lipschitz domain has scale $\delta$ if after zooming in to scale $\delta$, one obtains a Lipschitz domain with both the Lipschitz constant and the number of domains in the covering controlled independently of $\delta$.

\begin{definition}[Scale of a Lipschitz domain]
Let $D>0$ and define for $r>0$ 
$$
V_r^1:=\{x\in\mathbb{R}^d\,:\, |x|<r\},\quad V_r^2:=\{(x_1,x')\in (0,r)\times\mathbb{R}^{d-1}\,:\, |x'|<r\}.
$$
We say that $\Omega\in C^{0,1}$ has scale $\delta$ with respect to $D$ if there are $
\{U^j_i\}_{i=1}^{N_j}\subset \delta^{-1}\Omega$, $j=1,2$, with  $N_j\leq D$
$$
\overline{U^1_i}\cap (\delta^{-1}\partial\Omega)=\emptyset,
$$
and $\{F^j_i\}_{i=1}^{N_j}$, $F^j_i:\overline{V_2^j}\to \overline{U^j_i}$ bi-Lipschitz maps such that 
\begin{gather*}
\max_{i,j}\big\{\|D F^j_i\|_{L^\infty},\|(DF^j_i)^{-1}\|_{L^\infty}\big\}\leq D ,\\
F^2_{i}(\{x_1>0\}\cap \overline{V^2})\subset \delta^{-1}\Omega,\quad F^2_{i}(\{x_1=0\}\cap \overline{V_2^2})=\partial(\delta^{-1}\Omega)\cap \overline{U_i^2},\\
\delta^{-1}\Omega\subset \bigcup_{j=1,2}\bigcup_{i=1}^{N_j}F^j_i(V^j_{1}).
\end{gather*}
We then write 
$$
\delta(\Omega,D):=\sup\big\{ \delta>0\,:\, \Omega\text{ has scale }\delta\text{ with respect to } D\big\}.
$$
\end{definition}

\subsection{The finite-element space}

We consider finite element spaces consisting of mapped piecewise polynomials on curved elements. Specifically, for $\mathcal{T}$ a simplical, affine conforming triangulation of $\Omega$ with element maps $\{\mathcal{F}_T\}_{T\in\mathcal{T}}$ and reference element $\widehat{T}$ (in the sense of, e.g., \cite[\S A.1]{GSAN}), we denote the space of mapped polynomials of degree $p$ on the triangulation $\mathcal{T}$ by
$$
V^p(\mathcal{T},\{\mathcal{F}_{T}\}_{T\in\mathcal{T}}):=\big\{ u\in H^1(\Omega)\,:\, u|_{T}\circ \mathcal{F}_T\in\mathbb{P}^{p}(\widehat{T})\big\}.
$$
We then let 
$$
V_{\mathcal{T}}^p:=V(\mathcal{T},\{\mathcal{F}_T\}_{T\in\mathcal{T}})\cap \cZ.
$$
For a triangulation $\mathcal{T}$ and $U\subset \Omega$, we denote the maximal width of an element by 
$$
h(\mathcal{T},U):=\sup_{T\in\mathcal{T},T\cap U\neq \emptyset}\operatorname{diam}(T),
$$
  and the \emph{local uniformity constant of $\mathcal{T}$ at scale $\e$} by
$$
\mathcal{U}(\mathcal{T},\e):=\sup_{x\in\Omega}\sup_{\substack{T_1,T_2\in \mathcal{T}\\T_1\cap B(x,\e)\neq \emptyset\\ T_2\cap B(x,\e)\neq \emptyset}}\frac{\operatorname{diam}(T_1)}{\operatorname{diam}(T_2)}.
$$

\subsection{
%Partitions and 
Definition of the two-level hybrid Schwarz preconditioner}

Denote the respective inclusion maps by 
$\iota :\fine\to \cZ $, $\iota_0:\coarse\to \cZ$,
$\iota_\ell:\subdomain \to H^1_{\partial\Omega}(\Omega_\ell)$,
%(where both $\subdomain$ and 
% $H^1_{\partial\Omega}(\Omega_\ell)$ are defined by 
% \eqref{e:H1partial}),  
$\iota_{\ell,h}:\subdomain\to \fine$, and $\iota_{0,h}:\coarse\to \fine$.  %the inclusion maps. 
At the level of functions, the hybrid Schwarz left-preconditioner for $P$ is $B_L^{-1}(P):(\fine)^*\to \fine$
\beq\label{e:BLinverse}
B_L^{-1}=B_L^{-1}(P):= \iota_{0,h}A_0^{-1}\iota_{0,h}^* +\Big(\sum_{\ell=1}^N \iota_{\ell,h} \mathcal{I}_\ell \chi_\ell\iota_\ell A_\ell ^{-1}\iota_{\ell,h}^*\Big) (I-A\iota_{0,h}A_0^{-1}\iota_{0,h}^*),
\eeq
where $\mathcal{I}_\ell:H_{\partial\Omega}^1(\Omega_\ell)\cap \overline{C^0}(\Omega_\ell)\to \subdomain$ and 
$$
A_0:=\iota_0^*P\iota_0,\qquad A:=\iota^*P\iota,\qquad A_\ell:=\iota^*_\ell P_\ell \iota_\ell,
$$
where $P$ is the operator defined by $a$ and $P_\ell$ that defined by $a_\ell$. 
Note that $\iota_{\bullet}^*$ is the 
adjoint of $\iota_{\bullet}$; recall that this can be identified with the corresponding $L^2$ orthogonal projector. 

Fix a basis $\{\phi_i\}_{i=1}^L$ of $\fine$ with mass matrix $\matrixM_{ij}=\langle \phi_j,\phi_i\rangle$, as well as bases $\{\phi_i^c\}_{i=1}^{L_c}$ of $\coarse$, and $\{\phi_{i,\ell}\}_{i=1}^{L_\ell}$ of $\subdomain$ respectively.
Then set
\begin{gather*}
(\matrixR_\ell)^*_{ij}:=\sum_{m}\matrixM^{-1}_{im} \langle \phi_{j,\ell} ,\phi_{m}\rangle,\qquad (\matrixR_0)^*_{ij}:=\sum_{m}\matrixM^{-1}_{im}\langle \phi_j^c,\phi_m\rangle,\\
(\matrixR_{\ell}^{\chi_\ell})^*_{ij}:= \sum_m\matrixM^{-1}_{im} \langle \mathcal{I}_\ell \chi_\ell \phi_{j,\ell},\phi_m\rangle,\\
(\matrixA_0)_{ij}:=a(\phi^c_j,\phi^c_i),\qquad (\matrixA_\ell)_{ij}:=a_\ell(\phi_{j,\ell},\phi_{i,\ell}),\qquad \matrixA_{ij}:=a(\phi_j,\phi_i).
\end{gather*}
At the level of matrices, the two-level hybrid-Schwarz left preconditioner is then given by
\beq\label{e:matrixB}
\matrixB_L^{-1}= \matrixB_L^{-1}(\matrixA):= 
\matrixR_0^*\matrixAzero^{-1} \matrixR_0 
+
\bigg(\sum_{\ell=1}^N
\big(\matrixR_\ell^{\chiell}\big)^*\matrixA_\ell^{-1} \matrixR_\ell\bigg)
\big(\matrixI -\matrixA\matrixR_0^*\matrixAzero^{-1} \matrixR_0 \big)
\eeq
and the right preconditioner given by
\beqs
\matrixB_R^{-1}=\matrixB_R^{-1}(\matrixA)
:=\matrixR_0^*\matrixAzero^{-1} \matrixR_0 
+
\big(\matrixI -\matrixR_0^*\matrixAzero^{-1} \matrixR_0\matrixA\big)
\bigg(\sum_{\ell=1}^N
\matrixR_\ell^*\matrixA_\ell^{-1} \matrixR_\ell^{\chiell} \bigg).
\eeqs
To state our theorem, we need a few more definitions.
\begin{definition}
    Let $\{\Omega_\ell\}_{\ell=1}^N$ be subsets of $\Omega$. We define the \emph{$k^{-1}$-overlap growth} of $\{\Omega_{\ell}\}_{\ell=1}^N$, by  
$$
\Lambda(\{\Omega_{\ell}\},k):=\max_{\ell=1,\ldots,N}\sup_{R\geq 0}\tfrac{1}{\log( 2+R)}\log\Big(\#\big\{ j\,:\, kd(\Omega_j,\Omega_\ell)\leq R\}\Big).
$$
\end{definition}
That is, $\Lambda(\{\Omega_{\ell}\},k)$ is the maximal growth of the number of domains that intersect the ball of radius $R/k$ around $\Omega_\ell$.
Define
the \emph{overlap number} by
\beq\label{e:overlapDef}
\tilde{\Lambda}(\{\Omega_\ell\})  := 
\max\bigl\{ \#\mathcal{N} (\ell ):\ell =1,...,N\bigr\}<\infty,\quad\text{where}\quad \mathcal{N} (\ell ):=\bigl\{ \ell^\prime  :\Omega_\ell \cap \Omega_{\ell^\prime}  \not =\emptyset \bigr\}<\infty.
\eeq
Then $\Lambda(\{\Omega_{\ell}\},k)\geq \log \tilde{\Lambda}(\{\Omega_\ell\})/\log 2$ (by taking $R=0$ in the definition); i.e., $\Lambda(\{\Omega_\ell\},k)$ controls $\tilde{\Lambda}(\{\Omega_\ell\})$.
Conversely, 
for any fixed $D$, $c>0$ if 
$\tilde{\Lambda}(\{\Omega_\ell\})<\infty$ and 
%$\{\Omega_\ell\}$ with 
$
\min_\ell\delta(\Omega_\ell,D)\geq ck^{-1},
$
then $\Lambda(\{\Omega_\ell\},k)<\infty$ (i.e., $\tilde{\Lambda}(\{\Omega_\ell\})$ controls $\Lambda(\{\Omega_\ell\},k)$ under a scale assumption).

Finally, given a basis $\{\phi_i\}_{i=1}^L$, let the real symmetric positive-definite matrix $\matrixD_{k,\{\phi\}}$ %\in \Rea^{n\times n}$ 
be such that, 
\begin{equation}\label{e:ip2}
\Big\|\sum_{i=1}^L\bV_i\phi_i\Big\|_{H^1_k(\Omega)}^2  = \big\langle \matrixD_{k,\{\phi\}} \bV,  \matrixD_{k,\{\phi\}}                    \bV\big\rangle_2,
\eeq
where $\langle\cdot,\cdot\rangle_2$ denotes the Euclidean inner product.
With these definitions, one can check that, for any $K\geq 0$,
$$
\| (I - B^{-1}_L A)^K \|_{\cZ \to \cZ}= 
\big\|\matrixD_{k,\{\phi\}}(\matrixI-\matrixB_L^{-1}\matrixA)^K\matrixD_{k,\{\phi\}}^{-1}\big\|_{\ell^2\to \ell^2}
$$
and
$$
\| (I -  AB^{-1}_R)^K \|_{\cZ^* \to \cZ^*}= 
\big\|\matrixD_{k,\{\phi\}}^{-1}(\matrixI-\matrixA\matrixB_R^{-1})^K\matrixD_{k,\{\phi\}}\big\|_{\ell^2\to \ell^2}.
$$

\subsection{The main result}

We are now in a position to state our main result. The structure of this result is, roughly speaking, 
\begin{enumerate}
    \item fix some constants and a $C^\infty$ Helmholtz problem, 
\item state that there exist two other constants $c$ and $C$,
\item list the components of the two-level hybrid Schwarz preconditioner, stating how the geometry of the triangulations and subdomains, the regularity of the cutoffs, etc.~depend on the constants in Points 1 and 2 (see \eqref{e:triangulations}, \eqref{e:geometry0}, \eqref{e:deltaLower}, and then \eqref{e:geometry} below), and
\item give a bound on the norm of powers of the difference between the preconditioned matrix  and the identity (\eqref{e:thebound} below).
\end{enumerate}
\begin{theorem}
\label{t:main}
    Let $M>0$, $0<s<\frac{1}{2}$, $\Upsilon>0$, $k_0>0$, $p\geq 1$, $c_0>0$, $\{D_{\alpha}\}_{\alpha\in \mathbb{N}^d}\subset (0,\infty)$, and $a=a(k)$ be a piecewise  $C^{\infty}$ Helmholtz problem with absorption and $C^\infty$ truncation. 
    There are $c,C>0$ such that for all $k>k_0$ with 
    \beq\label{e:polyGrowth}
    \rho:=\rho(a(k))\leq \Upsilon k^M, 
    \eeq
    all simplicial, affine conforming $C^{p+1}$ triangulations $\mathcal{T}_c$ and $\mathcal{T}_f$ of $\Omega$ with constant $\Upsilon$ satisfying $V_{\mathcal{T}_c}^p\subset V_{\mathcal{T}_f}^p$, 
    \beq\label{e:triangulations}(H_ck)^{2p}\rho+H_ck\leq c, \qquad \qquad H_c:=h(\mathcal{T}_c,\Omega),\qquad \qquad 
    \mathcal{U}(\mathcal{T}_c,k^{-1})+\mathcal{U}(\mathcal{T}_f,k^{-1})\leq \Upsilon,
    \eeq
    all coverings $\{\widetilde{W}_\ell\}_{\ell=1}^N$, $\{W_{\ell}\}_{\ell=1}^N$ and $\{\Omega_{\ell}\}_{\ell=1}^N$ satisfying~\eqref{e:subdomainThing1}, $\widetilde{W}_\ell\subset W_\ell$,
    \beq\label{e:geometry0}
    \Lambda(\{\Omega_\ell\},k)\leq \Upsilon, \qquad \partial\Omega_\ell\subset \cup_{T\in\mathcal{T}_f}\partial T,\qquad \partial\Omega\cup \Gamma_p\subset \cup_{T\in\mathcal{T}_c}\partial T,
    \eeq
    sets $\{U_\ell\}_{\ell=1}^N$ satisfying~\eqref{e:subdomainThing2}, piecewise $C^{\infty}$ subdomain problems on $\Omega_\ell$ with Lipschitz truncation for $(a(k),W_\ell,U_\ell)$ with constant $c_0$, 
    %and $c$-sized advection, 
    $\{a_\ell(k)\}_{\ell=1}^N$, and partitions of unity $\{\chi_\ell\}_{\ell=1}^N\subset \overline{C^{p,1}}(\Omega\setminus\Gamma_p)$ 
    subordinate to $\{\widetilde{W}_\ell\}_{\ell=1}^N$ and satisfying
\beq\label{e:domain2domain}
    \nu \cdot \tilde{A}\nabla \chi_\ell=0\text{ on }\Gamma_{\rm hard}\cap \partial\Omega_\ell,\qquad \nu \cdot (\tilde{A}_+\nabla (\chi_\ell)_+-\tilde{A}_-\nabla (\chi_\ell)_-)=0\text{ on }\Gamma_p\cap \Omega_\ell,
    \eeq
and\beq\label{e:deltaLower}
\Upsilon\min\Big\{\delta(a_\ell,\{D_\alpha\},p+1),\delta(\chi_\ell,\{D_\alpha\},p+1)\Big\}\geq \delta_\ell\geq c_0k^{-1}, \qquad \delta_\ell:=\delta(\Omega_\ell,D_0),
\eeq
\beq\label{e:geometry}
\big\{x\in\Omega_\ell\,:\, d(x,\partial\Omega_\ell\setminus \partial\Omega)<\Upsilon^{-1}\delta_\ell\big\}\subset U_\ell, \qquad d(W_\ell^c,\widetilde{W}_\ell)\geq \Upsilon^{-1}\delta_\ell,
\eeq
\beq \label{e:rhoellBound}\rho_\ell:=\rho(a_\ell(k))\leq \Upsilon (\delta_\ell k)^{M},
\eeq
with
\beq\label{e:subdomainThreshold}
\sup_{\ell}(h_\ell k)^{p}\rho_\ell +(h k)^s\leq c,\qquad h_\ell:=h(\mathcal{T}_f,\Omega_\ell),\qquad h:=\sup_\ell h_\ell
\eeq
all $\mathcal{I}_\ell:\overline{C^0}(\Omega_\ell)\to \subdomain$  and $\{\mu_\ell\}_{\ell=1}^N\subset (0,\infty)$ satisfying 
    \begin{equation}
\begin{gathered}\label{e:assIell}
    \|\mathcal{I}_\ell v_{\ell,h}\|_{H_k^1(\Omega_\ell)}\leq \Upsilon\|v_{\ell,h}\|_{H_k^1(\Omega_\ell)},\quad
    \|(I-\mathcal{I}_\ell)\chi_\ell v_{\ell,h}\|_{H_k^1(\Omega_\ell)}\leq \Upsilon \mu_\ell\|v_{\ell,h}\|_{H_k^1(\Omega_\ell)}\quad \tfor v_{\ell,h}\in \subdomain,\\
    \|(I-\mathcal{I}_\ell)u\|^2_{H_k^{1}(\Omega_\ell)}\leq \Upsilon(h_\ell k)^{2p}\sum_{\substack{T\in \mathcal{T}_f\\T\cap \Omega_\ell\neq \emptyset}}\|u\|^2_{H_k^{p+1}(T)}, \quad u\in H^{p+1}_k(\Omega_\ell)\cap \overline{C^0}(\Omega_\ell),
\end{gathered}
    \end{equation}
    and all $K\geq 1$, the preconditioner $\matrixB_L^{-1}(\matrixA)$ defined with $\coarse=\mathcal{V}_{\mathcal{T}_c}^p$ and $\fine=\mathcal{V}_{\mathcal{T}_f}^p$, using the basis $\{\phi\}_{i=1}^L$ of $\fine$, satisfies
    \begin{equation}
    \label{e:thebound}
\begin{aligned}
\big\|\matrixD_{k,\{\phi\}}(\matrixI-\matrixB_L^{-1}\matrixA)^K\matrixD_{k,\{\phi\}}^{-1}&\big\|_{\ell^2\to \ell^2}+\big\|\matrixD_{k,\{\phi\}}^{-1}(\matrixI-\matrixA\matrixB_R^{-1})^K\matrixD_{k,\{\phi\}}\big\|_{\ell^2\to \ell^2}\\
&\leq C^K\e^{K-1}\Big(\e+\max_\ell\Big\{ \rho_\ell (k\delta_\ell)^{-1}\rho (H_ck)^{p}+ \rho_\ell (H_ck)^{p}+\langle k \delta_\ell\rangle^{-M}\Big\}\Big),
\end{aligned}
\end{equation}
where
\begin{align}\nonumber
\e&:=\max_\ell\bigg\{ \rho_\ell (k\delta_\ell)^{-1}\rho (H_ck)^{2p}+ \rho_\ell (H_ck)^{2p}
 +\langle k\delta_\ell\rangle^{-M}(H_ck)^p\\
 &\qquad\qquad\qquad\qquad + (h /H_c)^{p}  +   \mu_\ell  + \big((h_\ell k)^p\rho_{\ell}+(h_\ell k)^{s}\big)\Big(\rho_\ell(k\delta_\ell)^{-1}+1\Big)\bigg\}.
 \label{e:epsilon}
\end{align}
$$
%\rho_\ell:=\rho(a_\ell(k)),\qquad h_\ell:=h(\mathcal{T}_f,\Omega_\ell),\qquad h:=h(\mathcal{T}_f,\Omega).
$$
\end{theorem}

To the best of our knowledge, Theorem \ref{t:main} is the first result on the convergence of two-level DD methods that does not require the coarse space to be in the asymptotic regime (i.e., the first result for which the Galerkin solution in the coarse space is not necessarily $k$-uniformly quasi-optimal). 

\subsection{Discussion of Theorem \ref{t:main}}

% We're going to discuss situations where 
% $\rho_\ell \leq C k\delta_\ell$ and $\epsilon$ is small.

% In this case 

\paragraph{Subdomain geometry.}
The definition of $\delta_\ell$ in condition \eqref{e:deltaLower} implies that the domains $\Omega_\ell$ must have diameters that are $\sim \delta_\ell$ and, in addition, that the absorbing layer have width at least $c\delta_\ell$. In particular, the minimal domain diameter is $\sim k^{-1}$ and we require generous overlap.

\paragraph{Making $\epsilon$ \eqref{e:epsilon} small.}

To understand when $\epsilon$ is small, we need information about $\rho, \mu_\ell$, and  $\rho_\ell$.

In this discussion, for simplicity we assume that $\rho \geq ck$; 
this occurs if the Helmholtz problem is not absorbing everywhere.
We also consider the standard choice of $\cI_\ell$; i.e., nodal interpolation on Lagrange points. In this case (and in our setting of mapped piecewise polynomials on curved elements) the superapproximation property in the first line of \eqref{e:assIell} holds with 
\beq\label{e:muell}
\mu_\ell=h_\ell \delta_\ell^{-1}
\eeq
by \cite[Lemma 4.6]{GS4}, and the approximation property in the second line of \eqref{e:assIell} holds by 
%,  e.g., \cite[Corollary 4.4.24]{BrSc:08} and 
\cite[\S D.2]{AGS2}. Equation \eqref{e:muell} together with the facts that $\delta_\ell\geq ck^{-1}$ (by \eqref{e:deltaLower})
and  $h_\ell k\ll 1$
(by \eqref{e:subdomainThreshold}) 
imply that the contribution of $\mu_\ell$ to $\e$ \eqref{e:epsilon} is small.

We highlight that 
the most practically-interesting choice of subdomain size is arguably $\delta_\ell \sim k^{-1}$ (since one wants to minimise the number of degrees of freedom per subdomain for parallel scalability). However, it is important that the behaviour of the hybrid Schwarz preconditioner not be very sensitive to the number of wavelengths in a subdomain; i.e., we need to check that the result of Theorem \ref{t:main} does not degenerate if $k\delta_\ell $ becomes large.

We now split the discussion into three cases: (i) $k\delta_\ell \leq c$ for $c$ small, (ii) $k\delta_\ell\geq C$ with $C$ large, and
and (iii) $c\leq k\delta_\ell \leq C$.

\medskip
\noindent \emph{Case 1:~domains that are small multiples of $k^{-1}$: $k\delta_\ell \leq c$.}
By the Poincar\'e--Friedrichs inequality, 
there is $c>0$ such that if $\operatorname{diam}(\Omega_\ell)\leq ck^{-1}$, then $\rho_\ell \leq 1$ (thus, in particular, the bound ~\eqref{e:rhoellBound} holds).

\medskip
\noindent \emph{Case 2:~non-trapping domains that are much larger than $k^{-1}$: $k\delta_\ell\geq C$.}
    We say that the subdomain problem $a_\ell$ is non-trapping if $\rho_\ell \leq C(k\delta_\ell)$.  There are many situations where one can prove that $a_\ell$ is non-trapping in this sense. The simplest (although somewhat trivial) case is when $\Omega_\ell \subset U$. In this case, it is easy to see that $\rho_\ell \leq C$ for $k\delta_\ell\gg 1$. 
        A more non-trivial situation occurs when $\Omega_\ell \cap \Gamma_p=\emptyset$. Set
    $$
    p_\ell(x,\xi):= A_{\ell}^{ij}(x)\xi_i\xi_j+b_\ell^i(x)\xi_i-n_\ell(x),
    $$
    and suppose there are $c,C,T>0$ such that $\pm \Im p_\ell \geq -C(k\delta_\ell)^{-1}$ on $|p_\ell|\leq c$, and for all solutions to 
    $$
    \dot x= \partial_{\xi}\Re p_\ell(x,\xi),\qquad \dot\xi=-\partial_x \Re p_\ell(x,\xi), \quad (x(0),\xi(0))=(x_0,\xi_0),\quad \text{with} \quad x_0\in\Omega_\ell,\, |p_\ell(x_0,\xi_0)|\leq c
    $$
    (extended by reflection at $\partial\Omega_-$, or, more precisely, as generalised broken bicharacteristics; see e.g~\cite[Appendix C]{AGS2},~\cite[\S 24.3]{Ho:85}) there is $0\leq t\leq T$ such that $x(\pm t)\in U_\ell$. Then, for $k\delta_\ell \gg1 $,  $\rho_\ell \leq Ck\delta_\ell$ (see e.g.~\cite{Bu:02,GaSpWu:18}).  This is, in fact, the origin of the name non-trapping since every trajectory escapes into the absorbing layer. This example can be extended to subdomains intersecting $\Gamma_p$ if one replaces generalised broken bicharacteristics by an analogue that may split into a reflected and transmitted part at $\Gamma_p$ (see e.g.~\cite{JaSaSt:15, SaVa:97} for the relevant dynamics). 

We point out that, when $\delta_\ell \ll 1$, the coefficients of the Helmholtz problem are effectively constant on each subdomain and all scatterer boundaries $\partial\Omega_-$ and interfaces $\Gamma_p$ are effectively hyperplanes. In this case, if the only change made to the (global) coefficients to create the subdomain problems is to add absorption, then $a_\ell$ will be non-trapping in the dynamical sense described above.
    
\medskip
\noindent\emph{Case 3:~middle-sized domains: $c\leq k\delta_\ell\leq C$.}
There are $c,C>0$ depending on $a$, $\Upsilon$, and $\{D_\alpha\}$, such that, except for $c\leq k\delta_\ell\leq C$, either the analysis in case 1 or the analysis in case 2 applies. To handle this intermediate regime, observe that the Fredholm alternative shows that for a fixed subdomain problem, $\rho_\ell$ may only fail to be uniformly bounded in $k$ due to the presence of finitely many eigenvalues. 
If, for example, the absorption on the subdomains is a CAP (and $k$ is large enough if there is a PML on the global problem), then there are no such eigenvalues and $\rho_\ell \leq c$. Indeed, in this case, the imaginary part of the PDE on each subdomain has a sign, and uniqueness follows from the unique continuation principle.

\medskip

In all of the above cases $\rho_\ell \leq Ck\delta_\ell$; hence, recalling that $\rho \geq ck$ and $\delta_\ell\leq c$,
we see that $\e$ \eqref{e:epsilon} is controlled by
    $$
\e\lesssim \rho (H_ck)^{2p}
 +
 (H_ck)^p + (h /H_c)^{p}+(h k)^s  +\max_\ell (h_\ell k)^pk\delta_\ell. 
$$
Therefore, to make $\e\ll 1$, we require $\rho(H_ck)^{2p}\ll 1$, $h\ll H_c$, and $(h_\ell k)^p(k\delta_\ell)\ll 1$. When $k\delta_\ell \sim 1$, this last requirement is automatic and this is the result summarised in Informal Theorem \ref{t:informal}.

% For the most typical choices of subdomains, one can, in fact, check that  $\delta_\ell \sim k^{-1}$ implies $\rho_\ell\leq C$
The condition $\rho(H_ck)^{2p}\ll 1$ is 
observed empirically to be necessary for the Galerkin solution in the coarse space to have controlled relative error and, in fact, the matrix $\matrixA_0$ may fail to be invertible above this threshold
 \cite{BeSaToVe:23}. In Section~\ref{s:numerics}, we give numerical evidence in the setup of Theorem~\ref{t:main} that indicates that the threshold $\rho(H_ck)^{2p}\ll 1$ is required to have the number of iterations -- of both the fixed-point method and GMRES -- bounded uniformly in $k$.

    \paragraph{Consequences of having $\e$ small for iteration counts.}
    In all situations where $\e\ll 1$, one can achieve a residual of size $\gamma>0$ after 
    \beq\label{e:logs}
    \sim \frac{\log \gamma-\log (1+(H_ck)^p\rho)}{\log \e}
    \eeq
    iterations, for both the fixed-point method and GMRES in the $\matrixD_{k, \{\phi\}}$ weighted inner product \eqref{e:ip2}.
Recall that unweighted GMRES takes an extra $C\log (hk)^{-1}$ iterations than weighted GMRES \cite[Corollary 5.8]{GoGrSp:20}. That is, the fixed-point method, weighted GMRES, and unweighted GMRES all convergence 
in no more than $\sim \log k$ iterations provided that $(H_ck)^{2p}\rho\ll 1$ and 
$h$ is not superalgebraically small in $k$.
From \eqref{e:logs}, we also see that 
Theorem \ref{t:main} recovers the result from \cite{GS4} that if 
$(H_ck)^p\rho\leq C$ then convergence of the fixed-point method/weighted GMRES occurs in a $k$-independent number of iterations.
    % This shows that with coarse spaces in the preasymptotic regime one still has convergence in no more than $\sim \log k$ iterations and recovers the fact, proved in~\cite{GS4}, that if 

\paragraph{The assumptions on the underlying Helmholtz problem.}

To simplify the analysis of the coarse problem, Theorem \ref{t:main} assumes that the global truncation boundary is smooth and that absorption of the form~\eqref{e:absorbing} holds. These assumptions, in particular, rule out truncation of the global Helmholtz problem with Cartesian PMLs. To cover Cartesian PMLs, one would need to repeat the preasymptotic $h$-FEM error analysis in \cite{GS3} but explicitly take into account the better behaviour of the Helmholtz solution operator in the PML region (using, e.g., ideas from \cite{GGGLS2}).

In addition, Theorem \ref{t:main} assumes that all coefficients and scatterer boundaries are smooth so that various localisation procedures produce errors that are super-algebraically small in $k$ (see the proof of Lemma \ref{l:coarse}). This assumption could be relaxed to $C^{p}$-type regularity assumptions, but we would then need to strengthen the polynomial growth assumption on $\rho(k)$ \eqref{e:polyGrowth}; specifically, for each $p$ one would allow a fixed $p$-dependent amount of polynomial growth in $\rho(k)$.

\paragraph{The assumptions on the cutoffs.}
Two features of the cutoffs $\chi_\ell$ are important for the proof of Theorem \ref{t:main}. First, they need to have sufficiently high ($C^{p,1}$) regularity (encoded in Theorem \ref{t:main} via \eqref{e:deltaLower}) so that they are bounded multipliers on $H_{k}^{p+1}$ and hence approximation and super-approximation estimates hold in piecewise polynomial spaces of degree $p$. Second
for the sum of local solves to plausibly be close to the true Helmholtz solution (which is in the domain of the operator $P$), it is crucial that $\chi_\ell$ map the domain of $P$ to itself. This is guaranteed by the conditions~\eqref{e:domain2domain}.

% $\rho_\ell$

\subsection{Discussion of the ideas of the proof}

To understand the proof of Theorem~\ref{t:main}, observe that the error in the preconditioned matrix is
\begin{align}
I-B_L^{-1}A&=\Big( I-\sum_{\ell=1}^N\iota_{\ell,h}\cI_\ell \chi_\ell \iota_\ell A_\ell^{-1}\iota_{\ell,h}^* A\Big)(I-\iota_{0,h}A_0^{-1}\iota_{0,h}^* A)\label{e:errorDecompose}\\
&\approx\Big( \sum_{\ell=1}^N\chi_\ell (I -\iota_{\ell,h} A_\ell^{-1}\iota_{\ell,h}^* A )\Big)(I-\iota_{0,h}A_0^{-1}\iota_{0,h}^* A)\nonumber,
\end{align}
where the error in  passing to the second line is controlled using super-approximation results in $\subdomain$. 
As advocated in~\cite[\S1.2.3]{GS4}, one should think of the 
error in the coarse solve, $(I-\iota_{0,h}A_0^{-1}\iota_{0,h}^* A) $, as removing the low frequencies and the error in the subdomain solve $\sum_\ell(\chi_\ell(I-\iota_{\ell,h}A_\ell^{-1}\iota_{\ell,h}^* A)$ as removing the high-frequencies
(since high frequencies do not propagate under the action of the Helmholtz solution operator, and thus can be well-approximated by a sum of local solves), so that their combination contracts the entire norm. 
The analysis in \cite{GS4} rigorously implemented this idea in the case when the Galerkin solution in the coarse space is $k$-uniformly quasi-optimal. In contrast, here we consider piecewise-polynomial subspaces in the preasymptotic regime, i.e., for which the Galerkin solution in the coarse space is not $k$-uniformly quasi-optimal; see \eqref{e:overkill1} below.

For the purposes of this discussion, we denote by $\Psi:H_k^1\to H_k^{p+1}$ an operator that returns the ``frequency $k$''-part of a function $u$-- this operator is constructed using the functional calculus of an associated self-adjoint operator, and is the key component of preasymptotic error analyses~\cite{GS3} (see Section~\ref{s:preliminaries} for the full definition of $\Psi$). 

\bigskip

\noindent{\bf Coarse solve as frequency $\lesssim k$ removal.}

Denote by $Q^0_{H_c}:\fine\to \coarse$ the solution to 
\begin{equation} 
\label{e:galerkinCoarse}
a(Q^0_{H_c}v_h,v_{H_c})=a(v_h,v_{H_c})\qquad \tfa v_{H_c}\in\coarse;
\end{equation}
i.e. 
\beq\label{e:overkill2}
Q^0_{H_c}=  A_0^{-1}\iota_{0,h}^*A.
\eeq
Then, by~\cite{GS3} there is $c>0$ such that if $(H_ck)^{2p}\rho\leq c$, then $Q^0_{H_c}$ is well-defined; more precisely $A_0$ is invertible 
and 
\beq\label{e:overkill1}
\|(I-Q^0_{H_c})u\|_{H_k^1}\leq C (1 + (\Hc k)^p \rho)\|(I-\Pi_{H_c})u\|_{H_k^1}.
\eeq
%Moreover \cite{GS3} proves bounds on $(I-Q^0_{H_c})$.
Furthermore, we show in Lemma~\ref{l:coarse} (see also~\cite[Theorem 6.27]{GSAN}) that, with $\Pi_{H_c}:H_k^1\to \coarse$ the $H_k^1$ orthogonal projector,
\begin{equation} 
\label{e:coarseEstimateIntro}
\begin{aligned}
\|(I-\Psi)(I-Q^0_{H_c})u\|_{H_k^1}&\leq C(1+(H_ck)^{2p}\rho)\|(I-\Pi_{H_c})u\|_{H_k^1},\\
\|\Psi(I-Q^0_{H_c})u\|_{H_k^{p+1}}&\leq C(H_ck)^{p}\rho \|(I-\Pi_{H_c})u\|_{H_k^1},\\
\|(I-Q^0_{H_c})u\|_{H_k^{1-p}}&\leq C(H_ck)^{p}\rho \|(I-\Pi_{H_c})u\|_{H_k^1}.
\end{aligned}
\end{equation}
Since 
\beq\label{e:approx}
\|(I-\Pi_{H_c})\Psi u\|_{H_k^1}\leq (H_ck)^{p}\|\Psi u\|_{H_k^1},
\eeq
the bounds \eqref{e:coarseEstimateIntro} imply that, when $(H_ck)^{2p}\rho \ll 1$, the coarse solve effectively removes the low frequencies of $u$; to see this, observe that if $u=\Psi u$, then the full norm contracts under the coarse solve. However, the high frequencies are not well approximated and so, when $1\ll(H_ck)^p\rho$ and $(H_ck)^{2p}\rho \ll 1$, they produce a possibly large component in $(I-Q^0_{H_c})$. 

To quantitatively capture the difference between the coarse solve acting on high and low frequencies, we use 
the norm 
%More quantitatively, if we introduce the following norm on $H_k^1$ where we weight the high frequencies of $u$: 
\begin{equation}
    \label{e:xPrelim}
    \|u\|_{X}^2:=\|\Psi u\|_{H_k^{p+1}}^2+(H_ck)^{-2p}\|(I-\Psi)u\|_{H_k^1}^2,
\end{equation}
i.e., where high-frequencies are weighted by $(\Hc k)^{-2p}$. 
% Observe that 
% \beq\label{e:normEquiv1}
% \| u\|_X \leq C (\Hc k)^{-p} \|u\|_{H^1_k}.
% \eeq
Then
\beq\label{e:overkill3}
\|(I-Q^0_{H_c})u\|_{X}\leq C\big(1+(H_ck)^{2p}\rho\big)(H_ck)^{-p}\|(I-\Pi_{H_c})u\|_{H_k^1}
\eeq
and
$$
\|(I-\Pi_{H_c})u\|_{H_k^1}\leq C(H_ck)^p\|u\|_X,
$$
so that 
$$
\|(I-Q^0_{H_c})u\|_{X}\leq C\big(1+(H_ck)^{2p}\rho\big)\|u\|_X;
$$
compare to \eqref{e:overkill1}.
That is, the coarse solve has bounded relative error in the norm $\|\cdot\|_X$ when $(\Hc k)^{2p}\rho \leq c$.

\bigskip
\noindent {\bf Orthogonal decomposition on $\coarse$.}

A key step in the proof is the simple observation that, 
by \eqref{e:overkill1}, $(I-Q^0_{H_c})\Pi_{H_c}=0$; i.e., 
the coarse solve is  the identity in the coarse space. 
By~\eqref{e:errorDecompose} and \eqref{e:overkill2}, $I-B_L^{-1}A = E(I-Q^0_{H_c})$ for some operator $E:H_k^1\to H_k^1$. Therefore
$I-B_L^{-1}A=(I-B_L^{-1}A)(I-\Pi_{\Hc})$, so that 
$$
(I-B_L^{-1}A)^K=(I-B_L^{-1}A)\big((I-\Pi_{H_c})(I-B_L^{-1}A)\big)^{K-1},
$$
and thus    it is enough to show that $(I-\Pi_{H_c})(I-B_{L}^{-1}A):H_k^1\to H_k^1$ is small and $(I-B_{L}^{-1}A):H_k^1\to H_k^1$ is controlled.

By \eqref{e:coarseEstimateIntro} and \eqref{e:approx},
\begin{align}
\|(I-\Pi_{H_c})(I-Q^0_{H_c})\Psi u\|_{H_k^1}
&\leq \|(I-\Pi_{H_c})\Psi(I-Q^0_{H_c})\Psi u\|_{H_k^1}
+
\|(I-\Pi_{H_c})(I-\Psi)(I-Q^0_{H_c})\Psi u\|_{H_k^1}\nonumber
\\
&\leq C \big(1+(H_ck)^{2p}\rho\big)\|(I-\Pi_{H_c})\Psi u\|_{H_k^1}\nonumber\\
&\leq C (H_ck)^p\big(1+(H_ck)^{2p}\rho\big)\| u\|_{H_k^1}
\label{e:explain1}
\end{align}
%\end{equation}
and 
\beq\label{e:explain2}
\|(I-\Pi_{H_c})\Psi (I-Q^0_{H_c})(I-\Psi) u\|_{H_k^1}\leq C(H_ck)^{2p}\rho\|u\|_{H_k^1}.
\eeq
The two bounds \eqref{e:explain1} and \eqref{e:explain2} show that using the coarse solve as a preconditioner in $\fine$ results in an error that is acceptable (i.e., small after approximation in $\coarse$) except at frequencies $\gg k$. Because these frequencies do not propagate under the Helmholtz solution operator, it is then reasonable to expect that the addition of a subdomain solve as in \eqref{e:errorDecompose} can correct these errors.
The analysis in the next step shows that this is indeed the case.

\bigskip

\noindent {\bf Analysis of the subdomain problems.}

To finish the sketch proof of Theorem~\ref{t:main} it is therefore sufficient to prove that  
\begin{align}
\Big\|(I-\Pi_{H_c})\sum_\ell \chi_\ell (I-\iota_{\ell,h}A_\ell^{-1}\iota_{\ell,h}^* A) (I-Q^0_{H_c})u\Big\|_{H_k^1}\ll (H_ck)^p\|(I-Q^0_{H_c})u\|_{X}
\label{e:subdomainIntro}
\end{align}
since, by~\eqref{e:overkill3},
\begin{align*}
   & (H_ck)^p\|(I-Q^0_{H_c})u\|_{X}\\
&\qquad \leq C(H_ck)^p (1+(H_ck)^{2p}\rho )(H_ck)^{-p}\|(I-\Pi_{H_c})u\|_{H_k^1}\leq C(1+(H_ck)^{2p}\rho)\|u\|_{H_k^1}.
\end{align*}
The proof of \eqref{e:subdomainIntro} is somewhat more subtle than the analysis of the coarse space, and is divided into two parts. We denote by $Q^\ell_h:\fine\to \subdomain$ the solution of 
\beq\label{e:QellIntro}
a_\ell(Q^\ell_hu_h,v_{\ell,h})=a(u_h,E_{\ell}v_{\ell,h})\qquad \tfa v_{\ell,h}\in\subdomain,
\eeq
where $E_{\ell}:H^1_{\partial\Omega}(\Omega_\ell)\to \cZ$ denotes extension by $0$; 
that is, $Q^\ell_h= A_\ell^{-1}\iota_{\ell,h}^*A$, so that \eqref{e:subdomainIntro} involves $(I-Q^\ell_h)(I- \Qcoarse)$.

We then write 
$$
(I-Q^\ell_h)=(I-Q^\ell)+(Q^\ell-Q^\ell_h), 
$$
where
\beq\label{e:Qell}
Q^\ell:= P_\ell^{-1}E_{\ell}^*P.
\eeq
That is, $Q^\ell$ is the continuous solution for which $Q^{\ell}_h$ is the Galerkin solution.

We start by showing that $\chi_\ell(I-Q^\ell)(I-Q^0_{H_c})$ is in $H^{p+1}_k$, and hence well approximated in $\coarse$. This relies on the following features of the setup
\begin{enumerate}
    \item The quadratic forms $a_\ell$ are equal to $a$ on $\supp \chi_\ell$ and hence $\chi_\ell (I-P_\ell^{-1}E_\ell^*P)$ is smoothing. More precisely, away from $\partial\Omega_\ell\setminus\partial\Omega$, this allows us to write $\chi_\ell (I-Q^\ell) =\chi_\ell P_{\ell}^{-1}[\chiellg,P]$ for some $\chiellg$ with $\supp (1-\chiellg)\cap \supp \chiell=\emptyset$. 
    \item There is absorption in a neighborhood of the boundaries $\partial\Omega_\ell\setminus \partial\Omega$.  
    \item $P(I-Q^0_{H_c})$ is orthogonal to $\coarse$ and hence, for terms near $\partial\Omega_\ell\setminus\partial\Omega$, the forcing term is small. 
\end{enumerate}

It then remains to understand $\chi_\ell (Q^\ell-Q_h^\ell)$. For this, we use the local-quasi-optimality results of~\cite{AGS1} together with our analysis of $Q^\ell$ to see that it is enough to estimate $Q^\ell-Q_h^\ell$ in some very low regularity Sobolev space. This is then done using the preasymptotic Galerkin estimates from~\cite{GS3} applied to the subdomain problem.

\section{Numerical experiments}
\label{s:numerics}

In this section, we present numerical experiments for the Helmholtz equation \emph{either} in free space \emph{or} in the exterior of a trapping obstacle. We consider polynomials of degree 2 and 4, choose $\delta_\ell\sim k^{-1}$ (for which Theorem \ref{t:main} gives the fastest convergence of the two-level method), and present results demonstrating that \begin{enumerate}[label=(\Alph*),ref=\Alph*]
    \item \label{i:A} under the assumptions of Theorem~\ref{t:main}, one obtains exponential convergence of both GMRES and the iterative solver uniformly in $k$ (i.e., consistent with Theorem \ref{t:main}), and 
    \item \label{i:B}
    if $(H_c k)^{2p}\rho(k)\gg 1$, then the iterative method diverges and the number of iterations required for GMRES to converge grows quickly with $k$. This failure of convergence happens regardless of how small one chooses $h$. 
\end{enumerate}
More precisely, under Point 2 we present numerical experiments 
with 
\begin{enumerate}[label=(B.\arabic*),ref=B.\arabic*]
\item \label{i:B1} 
$(\frac{hk}{C_f})^{2p}\rho(k) =1$ and $1\ll(\frac{\Hc k}{C_c})^{2p} \rho(k) \ll \rho(k)$,
\item \label{i:B2} $hk=C_f$ and $\Hc k=C_c$, 
and
\item  \label{i:B3}  
$(\frac{hk}{C_f})^{2p}\rho(k) =1$ and $\Hc k=C_c$.
\end{enumerate}
In all these three cases the iterative method diverges and the number of iterations required
for GMRES to converge grows quickly with $k$.

\subsection{Numerical setup}

All the numerical experiments were implemented in FreeFEM, with the code available at \url{https://github.com/FreeFem/FreeFem-sources/tree/develop/examples/ffddm/Helmholtz-2d-2lvl-nested.edp}.

\paragraph{Truncation and the sesquilinear form.}

We truncate to the ball of radius $R_{\tr}=0.9$, i.e., $\Omega_{\tr}:= B(0,R_{\tr})$ with an absorbing layer starting at $R_{\RPMLo}=0.7$. Since FreeFEM does not implement curved meshes, we commit a variational crime by approximating $\partial \Omega_{\tr}$ using linear piecewise polynomials (see \cite{ChSp:25} for an $k$-explicit analysis of the FEM error in this situation).

The Sommerfeld radiation condition is approximated using a radial PML, leading to the sesquilinear from $a:H_0^1(\Omega_{\tr})\times H_0^1(\Omega_{\tr})\to \mathbb{C}$ given by \eqref{e:quadratic} 
with $\tilde{A}, \tilde{b},$ and $\tilde{n}$ defined as follows. 

We use the PML scaling function 
$$
\PMLs=\begin{cases}10\exp\Big(-\frac{1}{3(r-R_{\RPMLo})}\Big)&r>R_{\RPMLo}\\ 0&r\leq R_{\RPMLo}\end{cases}.
$$
Then, defining
\beqs
\alpha(r) := 1 + \ri \PMLs'(r) \quad \tand\quad \beta(r) := 1 + \ri \PMLs(r)/r.
\eeqs
and the matrices $D$ and $H$ using polar coordinates $(r,\varphi)$ by
\beqs
D :=
\left(
\begin{array}{cc}
\beta(r)\alpha(r)^{-1} &0 \\
0 & \alpha(r) \beta(r)^{-1}
\end{array}
\right) 
\quad\tand\quad
H :=
\left(
\begin{array}{cc}
\cos \varphi & - \sin\varphi \\
\sin \varphi & \cos\varphi
\end{array}
\right),
\eeqs
 we set
$\tilde{b} =0$,
\beq\label{e:firstPML}
\tilde{A} := 
\begin{cases}
A
\hspace{-1ex}
& \tin \Omega,\\
HDH^T 
\hspace{-1ex}
&\tin (B_{\RPMLo})^c
\end{cases}
\quad\tand\quad
\tilde{n} := 
\begin{cases}
 n &\tin \Omega \cap B_{\RPMLo},\\
\alpha(r) \beta(r)^{d-1} 
\hspace{-1ex}
&\tin (B_{\RPMLo})^c.
\end{cases}
\eeq

We use the source 
\begin{equation} \label{e:rhsExperiments}
f=\chi(|x|)J_0(k|x|),\quad\text{ with} \quad\chi=\begin{cases} \frac{1}{1+\exp\left(\frac{1}{0.1-x}-\frac{1}{x-0.05}\right)}&0.05<x<0.1\\
1&x<0.05\\
0&x>0.1
\end{cases}.
\end{equation}

\paragraph{Subdomain problems.}
The subdomains $\Omega_\ell$ are obtained using the METIS algorithm. They are chosen to be of size $\sim k^{-1}$ and have generous overlap. We deviate slightly from the setup of Theorem~\ref{t:main} by using piecewise linear partition-of-unity functions $\chiell$ and an absorbing profile that is not smooth. The subdomain problems are then given by $a_\ell: \cZ_\ell\times \cZ_\ell \to \mathbb{C}$,
$$
a_{\ell}(u,v)= a(u,v)-i \langle q_\ell  u,v\rangle, 
$$
where 
$$
0\leq q_{\ell}=\begin{cases}10\exp\Big(-\frac{L_\ell}{0.6(d(x,\partial\Omega_{\ell}\setminus \partial\Omega)-L_\ell)}\Big)&d(x,\partial\Omega_{\ell}\setminus \partial\Omega)<L_\ell\\
0&\text{else},
\end{cases}
$$
and $L_\ell\sim k^{-1}$ is half of the subdomain overlap width.
% \in \overline{C^{0,1}}(\Omega_\ell )$, $\supp q_{\ell}\cap \supp \chiell=\emptyset$, and $q_\ell>c>0$ in a neighborhood of width $\sim k^{-1}$ around $\partial\Omega_{\ell}\setminus \partial\Omega $.

\paragraph{Piecewise polynomial spaces}

The mesh is generated using a standard Delaunay algorithm as implemented in FreeFEM and the basis functions $\{\phi\}_{i=1}^K$ from~\eqref{e:ip2} are the standard representation of piecewise polynomials using Lagrange nodes expressed in barycentric coordinates.

\subsection{Experiments in free space.}

In this case $\rho\sim Ck$; %\cite{GaSpWu:20} 
in the experiments below we take $\rho=k$.
%(In fact,~\cite{GaSpWu:20} one can show that $C=2L/\pi$, where $L$ is the length of the longest line in the region without the PML.) 

% $p=2$
% \begin{itemize}
%     \item $C_f=4$ and $C_c=8$ (1 above)
%     \item $C_f=4$ and $C_c=8$  (2.1)
%     \item $C_f=1$ and $C_c=2$ (2.2) fixed ppw
%     \item $C_f=4$ and $C_c=2$ (2.3) fixed ppw in coarse
% \end{itemize}

% $p=4$

% \begin{itemize}
%     \item $C_f=5.5$ and $C_c=11$ (1 above)
%     \item $C_f=5.5$ and $C_c=11$  (2.1)
%     \item $C_f=2.5$ and $C_c=5$ (2.2)
%     \item $C_f=6$ and $C_c=5$ (2.3)
% \end{itemize}

% \paragraph{Experiments in the exterior of a trapping domain}

% In this section, we present experiments in the domain 
% $$
% \Omega=\Omega_{\tr}\setminus \Omega_-,\qquad \Omega_-:=\cup_{x_0\in\{\pm 1\},y_0\in\{\pm 1\}} .3(x_0,y_0)+\begin{pmatrix} 1&1\\-1&1\end{pmatrix} \tfrac{1}{\sqrt{2}}(-.15,.15)^2
% $$
% with $P$ as in~\eqref{e:experiments} and $f$ as in~\eqref{e:rhsExperiments}

% In this case, one can show that for $k=\frac{2\pi}{.6\sqrt{2}-.3}\sqrt{n^2+m^2}$ with $m=0,n\in\mathbb{Z}$ of $n/m\in \{0,\pm 1,\pm 2\}$ we have $\rho(k)\gtrsim k^2$. 

% Source is between midpoints of lower two squares

%\FloatBarrier
\paragraph{P2 elements.} 
%\ 
%\FloatBarrier 
Figures \ref{f:P2A}, \ref{f:P2B1}, \ref{f:P2B2}, and \ref{f:P2B3} show numerical results for $p=2$ in the scenarios in Points \ref{i:A}, \ref{i:B1}, \ref{i:B2}, and \ref{i:B3}, respectively. 
In Figure \ref{f:P2A} we see that both the fixed-point iteration and GMRES converge in a $k$-uniform number of iterations (in accordance with Theorem \ref{t:main}), but in the other three experiments we see the number of iterations for both methods increasing as $k$ increases.

\begin{figure}[h!]
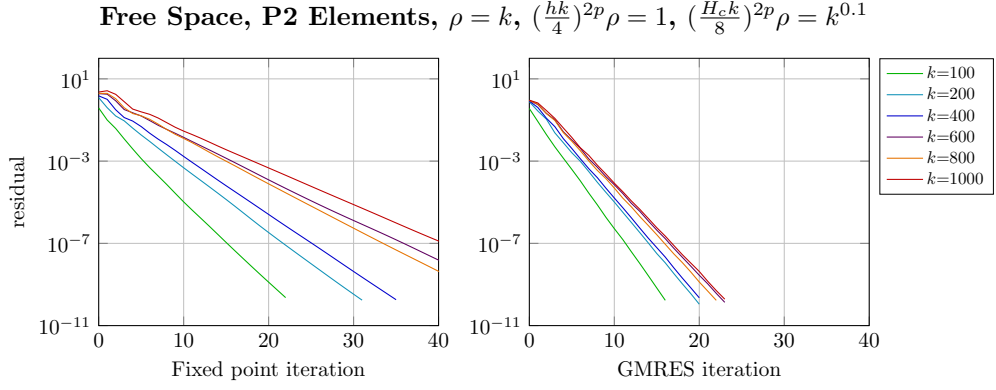

\generatepic{00}{}{2lvl_radialmmg_hom_c4_C8_keps0_k100p9_P2}{Free Space, P2 Elements, $\rho = k$, $(\frac{hk}{4})^{2p} \rho = 1$, $(\frac{H_ck}{8})^{2p} \rho = 1$}
\caption{Experiments for free space in the set up of Point \ref{i:A} above (i.e., covered by Theorem \ref{t:main}).}
\label{f:P2A}
\end{figure}

\begin{figure}
\generatepic{00}{}{2lvl_radialmmg_hom_c4_C8_keps0.1_k100p9_P2}{Free Space, P2 Elements, $\rho = k$, $(\frac{hk}{4})^{2p} \rho = 1$, $(\frac{H_ck}{8})^{2p} \rho = k^{0.1}$}
\caption{
Experiments for free space 
with $p=2$ in the set up of Point \ref{i:B1} above (i.e., a scenario not covered by Theorem \ref{t:main}).
}
\label{f:P2B1}
\end{figure}

\begin{figure}
\generatepic{00}{}{2lvl_radialmmg_hom_ccst1_Ccst2_k100p9_P2}{Free Space, P2 Elements, $\rho = k$, $hk = 1, H_ck=2$}
\caption{
Experiments for free space 
with $p=2$ in the set up of Point \ref{i:B2} above (i.e., a scenario not covered by Theorem \ref{t:main}).
}
\label{f:P2B2}
\end{figure}

%\enlargethispage{2\baselineskip}
\begin{figure}
\generatepic{00}{}{2lvl_radialmmg_hom_c4_Ccst2_k100p9_P2}{Free Space, P2 Elements, $\rho = k$, $(\frac{hk}{4})^{2p} \rho = 1$, $H_ck=2$}
\caption{Experiments for free space with $p=2$ in the set up of Point \ref{i:B3} above (i.e., a scenario not covered by Theorem \ref{t:main}).
}
\label{f:P2B3}
\end{figure}

% \begin{figure}[h!]
% \generatepic{00}{}{2lvl_final_hom_c4_C8_keps0.1_k100p9_P4}{P4, $\rho = k$, $(\frac{hk}{4})^{2p} \rho = 1$, $(\frac{H_ck}{8})^{2p} \rho = k^{0.1}$}
% \caption{
% Experiments for free space with $p=4$ in the set up of Point \ref{i:B1} above (i.e., a scenario not covered by Theorem \ref{t:main}).
% }
% \end{figure}
% \begin{figure}[h!]
% \generatepic{00}{}{2lvl_final_hom_c5_C10_keps0_k100p9_P4}{P4, $\rho = k$, $(\frac{hk}{5})^{2p} \rho = 1$, $(\frac{H_ck}{10})^{2p} \rho = 1$}
% \caption{
% Experiments for free space with $p=4$ in the set up of Point \ref{i:A} above (i.e., a scenario covered by Theorem \ref{t:main}).
% }
% \label{f:P4A}
% \end{figure}

% \begin{figure}[h!]
% \generatepic{00}{}{2lvl_final_hom_c5_C10_keps0.1_k100p9_P4}{P4, $\rho = k$, $(\frac{hk}{5})^{2p} \rho = 1$, $(\frac{H_ck}{10})^{2p} \rho = k^{0.1}$}
% \caption{
% Experiments for free space with $p=4$ in the set up of Point \ref{i:B1} above (i.e., a scenario not covered by Theorem \ref{t:main}).
% }
% \label{f:P4B1}
% \end{figure}

%\FloatBarrier

\paragraph{P4 elements.} Figures \ref{f:P4A}, \ref{f:P4B1}, \ref{f:P4B2}, and \ref{f:P4B3}  
show numerical results for $p=4$ in the scenarios in Points \ref{i:A}, \ref{i:B1}, \ref{i:B2}, and \ref{i:B3} respectively. 
In Figure \ref{f:P2A} we see that both the fixed-point iteration and GMRES converge in a $k$-uniform number of iterations (in accordance with Theorem \ref{t:main}), but in the other two experiments we see the number of iterations for both iterations increasing as $k$ increases.

%\newpage

\begin{figure}[h!]
\generatepic{00}{}{2lvl_radialfix_hom_c5_C10_keps0_k100p9_P4}{Free Space, P4 Elements, $\rho = k$, $(\frac{hk}{5})^{2p} \rho = 1$, $(\frac{H_ck}{10})^{2p} \rho = 1$}
\caption{
Experiments for free space with $p=4$ in the set up of Point \ref{i:A} above (i.e., a scenario covered by Theorem \ref{t:main}).
}
\label{f:P4A}
\end{figure}

\begin{figure}[h!]
\generatepic{00}{}{2lvl_radialfix_hom_c5_C10_keps0.1_k100p9_P4}{Free Space, P4 Elements, $\rho = k$, $(\frac{hk}{5})^{2p} \rho = 1$, $(\frac{H_ck}{10})^{2p} \rho = k^{0.1}$}
\caption{
Experiments for free space with $p=4$ in the set up of Point \ref{i:B1} above (i.e., a scenario not covered by Theorem \ref{t:main}).
}
\label{f:P4B1}
\end{figure}

% \begin{figure}[h!]

% \generatepic{00}{}{2lvl_final_hom_c6_C12_keps0_k100p9_P4}{P4, $\rho = k$, $(\frac{hk}{6})^{2p} \rho = 1$, $(\frac{H_ck}{12})^{2p} \rho = 1$}
% \caption{
% Experiments for free space with $p=4$ in the set up of Point \ref{i:A} above (i.e., a scenario covered by Theorem \ref{t:main}).
% }
% \end{figure}

% \begin{figure}[h!]
% \generatepic{00}{}{2lvl_final_hom_c6_C12_keps0.1_k100p9_P4}{P4, $\rho = k$, $(\frac{hk}{6})^{2p} \rho = 1$, $(\frac{H_ck}{12})^{2p} \rho = k^{0.1}$}
% \caption{
% Experiments for free space with $p=4$ in the set up of Point \ref{i:B1} above (i.e., a scenario not covered by Theorem \ref{t:main}).
% }
% \end{figure}

% \begin{figure}[h!]
% \generatepic{00}{}{2lvl_final_hom_ccst2_Ccst4_k100p9_P4}{P4, $hk = 2, H_ck = 4$}
% \caption{
% Experiments for free space with $p=2$ in the set up of Point \ref{i:B2} above (i.e., a scenario not covered by Theorem \ref{t:main}).
% }
% \end{figure}

\begin{figure}[h!!]
\generatepic{00}{}{2lvl_radialfix_hom_ccst2.5_Ccst5_k100p9_P4}{Free Space, P4 Elements, $\rho = k$, $hk = 2.5, H_ck = 5$}
\caption{
Experiments for free space with $p=4$ in the set up of Point \ref{i:B2} above (i.e., a scenario not covered by Theorem \ref{t:main}).
}
\label{f:P4B2}
\end{figure}

\begin{figure}[h!!]
\generatepic{00}{}{2lvl_radialfix_hom_c5_Ccst5_k100p9_P4}{Free Space, P4 Elements, $\rho = k$, $(\frac{hk}{5})^{2p} \rho = 1$, $H_ck = 5$}
\caption{
Experiments for free space with $p=4$ in the set up of Point \ref{i:B3} above (i.e., a scenario not covered by Theorem \ref{t:main}).
}
\label{f:P4B3}
\end{figure}

\subsection{Experiments in the exterior of a trapping obstacle}

Here we present experiments in the following domain, pictured in Figure~\ref{f:TrappingDomain},
\begin{equation}
\Omega=\Omega_{\tr}\setminus \Omega_-,\qquad \Omega_-:=\cup_{x_0\in\{\pm 1\},y_0\in\{\pm 1\}} 0.3(x_0,y_0)+\begin{pmatrix} 1&1\\-1&1\end{pmatrix} \tfrac{1}{\sqrt{2}}(-0.15,0.15)^2.
\label{e:trapping}
\end{equation}

\begin{figure}
    \centering
\begin{tikzpicture}
\clip (-3.5,-3.5) rectangle (3.5,3.5);
\draw node at (0,0){\includegraphics[width=0.75\linewidth]{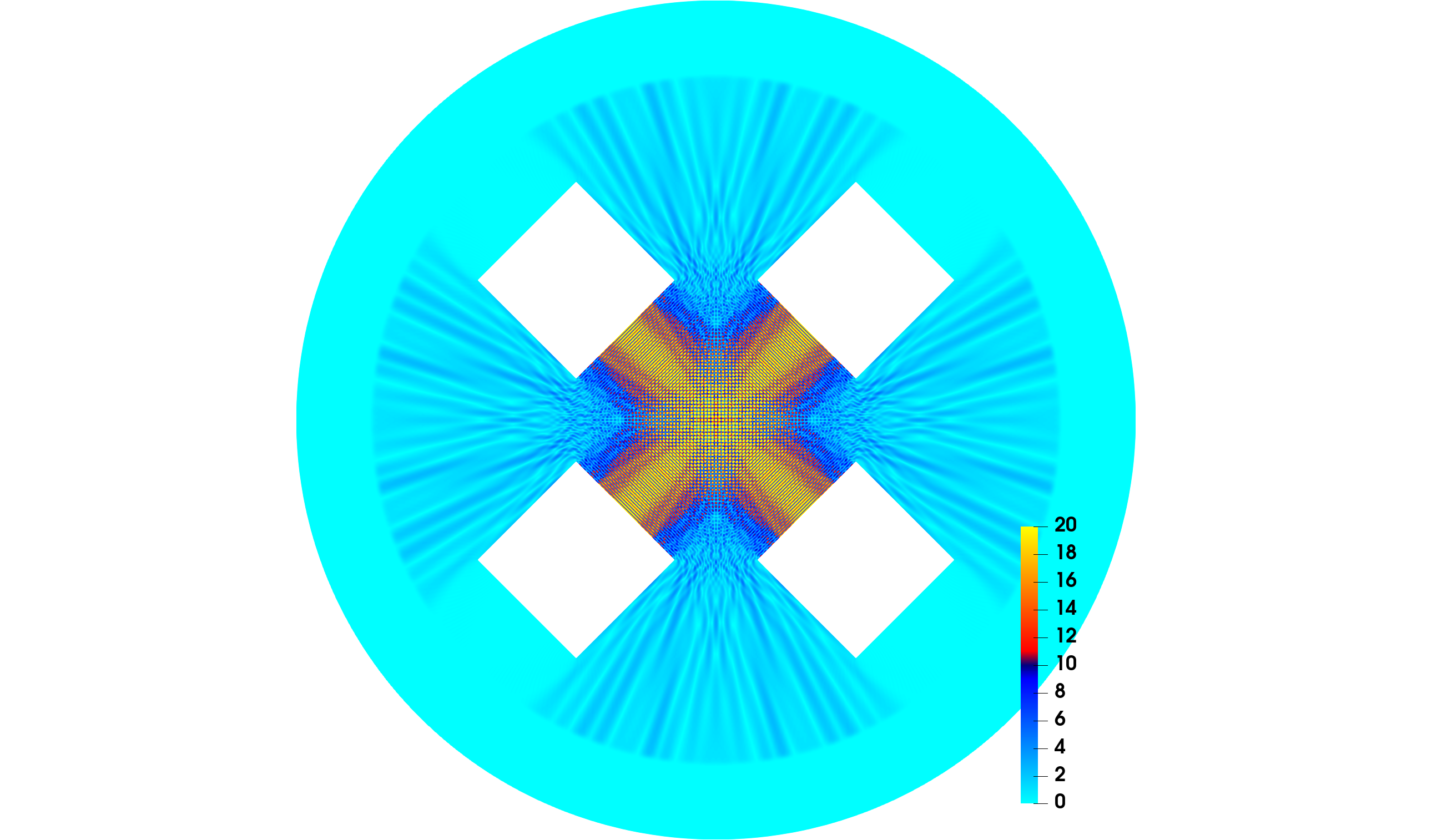}};
\end{tikzpicture}
    \caption{The absolute value of the
    total field in the exterior of the 
    domain $\Omega_-$ defined in \eqref{e:trapping}, 
    at wavenumber $k=595.641$,
    and where the data $f$ is defined by \eqref{e:rhsExperiments}.}
    \label{f:TrappingDomain}
\end{figure}

In this case, one can show  that $\rho(k)\gtrsim k^2$ when 
\beq\label{e:specialks}
k=\frac{2\pi}{0.6\sqrt{2}-0.3}\sqrt{n^2+m^2} \,\text{ with $m=0,n\in\mathbb{Z}$ or $n/m\in \{0,\pm 1,\pm 2\}$}.
\eeq
In the experiments we take $\rho(k)=k^2$. The data $f$ is given by \eqref{e:rhsExperiments} (with the source point in the centre of the four squares).

We present results under Point \ref{i:A} for both $p=2$ and $p=4$. For the results under Point \ref{i:B}, for brevity we present results under Point \ref{i:B1} for $p=2$ only, and results under Points \ref{i:B2} and \ref{i:B3} for $p=4$ only. 

\paragraph{P2 elements.} 
%\ 
%\FloatBarrier 
Figures \ref{f:P2Asquares} and \ref{f:P2B1squares} %\ref{f:P2B2squares}, and \ref{f:P2B3squares} 
show numerical results for $p=2$ in the scenarios in Points \ref{i:A} and  \ref{i:B1}.
%\ref{i:B2}, and \ref{i:B3}, respectively. 
In Figure \ref{f:P2Asquares} we see that both the fixed-point iteration and GMRES converge in a $k$-uniform number of iterations (in accordance with Theorem \ref{t:main}) but in the other three experiments we see the number of iterations for both methods increasing as $k$ increases.

\begin{remark}In each of the Figures~\ref{f:P2Asquares},
\ref{f:P2B1squares}, \ref{f:P4Asquares}, \ref{f:P4B2squares}, and
\ref{f:P4B3squares}, the fixed-point iteration behaves significantly worse at $k=194.729$ than at any other frequency tested. This can be explained by the fact that, in addition to being in $\frac{2\pi}{0.6\sqrt{2}-0.3}\mathbb{N} $, $194.729$ is very close to $\frac{2\pi}{0.6\sqrt{2}-0.3}\sqrt{2}\mathbb{N}$ and hence the solution operator is more badly behaved at this particular frequency than the others (see \eqref{e:specialks}).
\end{remark}

\begin{figure}
\generatepic{0Y}{}{2lvl_radialfix_4sqr_c8_C16_keps0_k100p9_P2}{Trapping Obstacle, P2 Elements, $\rho = k^2$, $(\frac{hk}{8})^{2p} \rho = 1$, $(\frac{H_ck}{16})^{2p} \rho = 1$, $k \in \frac{2 \pi}{0.6 \sqrt{2} - 0.3}\mathbb{N}$
}
\caption{Experiments in the exterior of a trapping obstacle in the set up of Point \ref{i:A} above (i.e., covered by Theorem \ref{t:main}).}
\label{f:P2Asquares}
\end{figure}

\begin{figure}
\generatepic{0Y}{}{2lvl_radialfix_4sqr_c8_C4_keps1_k100p9_P2}{Trapping Obstacle, P2 Elements, $\rho = k^2$, $(\frac{hk}{8})^{2p} \rho = 1$, $(\frac{H_ck}{4})^{2p} \rho = k$, $k \in$ $\frac{2 \pi}{0.6 \sqrt{2} - 0.3}\mathbb{N}$
}
\caption{
Experiments in the exterior of a trapping obstacle
with $p=2$ in the set up of Point \ref{i:B1} above (i.e., a scenario not covered by Theorem \ref{t:main}).
}
\label{f:P2B1squares}
\end{figure}

\paragraph{P4 elements.} Figures \ref{f:P4Asquares}, \ref{f:P4B2squares}, \ref{f:P4B3squares} 
show numerical results for $p=4$ in the scenarios in Points \ref{i:A}, \ref{i:B2}, and \ref{i:B3}, respectively. 
In Figure \ref{f:P2A} we see that both the fixed-point iteration and GMRES converge in a $k$-uniform number of iterations (in accordance with Theorem \ref{t:main}), but in the other two experiments we see the number of iterations for both iterations increasing as $k$ increases.

\begin{figure}
\generatepic{0Y}{}{2lvl_radialfix_4sqr_c6_C12_keps0_k100p9_P4}{Trapping Obstacle, P4 Elements, $\rho = k^2$, $(\frac{hk}{6})^{2p} \rho = 1$, $(\frac{H_ck}{12})^{2p} \rho = 1$, $k \in \frac{2 \pi}{0.6 \sqrt{2} - 0.3}\mathbb{N}$
}
\caption{Experiments in the exterior of a trapping obstacle with $p=4$ in the set up of Point \ref{i:A} above (i.e., covered by Theorem \ref{t:main}).}
\label{f:P4Asquares}
\end{figure}

\begin{figure}
\generatepic{0Y}{}{2lvl_radialfix_4sqr_ccst2_Ccst4_k100p9_P4}{Trapping Obstacle, P4 Elements, $\rho = k^2$, $hk = 2, H_ck = 4$, $k \in \frac{2 \pi}{0.6 \sqrt{2} - 0.3}\mathbb{N}$
}
\caption{Experiments in the exterior of a trapping obstacle with $p=4$ in the set up of Point \ref{i:B2} above (i.e., covered by Theorem \ref{t:main}).}
\label{f:P4B2squares}
\end{figure}

\begin{figure}
\generatepic{0Y}{}{2lvl_radialfix_4sqr_c6_Ccst4_k100p9_P4}{Trapping Obstacle, P4 Elements, $\rho = k^2$, $(\frac{hk}{6})^{2p} \rho = 1$, $H_ck = 4$, $k \in \frac{2 \pi}{0.6 \sqrt{2} - 0.3}\mathbb{N}$
}
\caption{Experiments in the exterior of a trapping obstacle with $p=4$ in the set up of Point \ref{i:B3} above (i.e., covered by Theorem \ref{t:main}).}
\label{f:P4B3squares}
    \end{figure}

\FloatBarrier

\section{Functional-analytic set up}
\label{s:preliminaries}

For convenience, and (up to multiplication by a $re^{-i\theta_0}$ for some $r>0$) without loss of generality we will assume that $\theta_0=0$, hence by~\eqref{e:strongElliptic} $\cos \theta_2\neq 0$,  and there is a constant $c>0$ such that for all $k>k_0$, 
$$
\Re a(u,u)\geq c\|u\|_{H_k^1(\Omega)}^2-\tfrac{3}{2}\|u\|_{L^2(\Omega)}^2 \quad \tfa u \in \cZ
$$
(where the sesquilinear form $a$ and space $\cZ$ are defined by \eqref{e:quadratic} and \eqref{e:cZ}, respectively).
We denote by $\cP:\cZ\to \cZ^*$ the self-adjoint operator generated by the quadratic form
$$
(\Re a)(u,v):=\tfrac{1}{2}(a(u,v)+\overline{a(v,u)}).
$$
To state more conveniently the mapping properties of $P$, $P^*$, and $\cP$ we define several new scales of Hilbert spaces. For $n=0$, set $\Hspace{0}=\Zspace{0}=L^2(\Omega)$ and for $n\geq 1$, we define the Hilbert spaces
$$
\Hspace{n}:=H_k^n(\Omega\setminus \Gamma_p)\cap H_k^1(\Omega),
$$
$$
\Zspace{n}:=\begin{cases}\cZ,&n=1\\ \Hspace{n}\cap \Bigg\{u\in \cZ\,:\, \begin{gathered}(\nu \cdot \tilde{A}k^{-1}\nabla u+\tfrac{i}{2}\tilde{b}u)|_{\Gamma_{\rm hard}}=0,\\
\nu \cdot (\tilde{A}_+k^{-1}\nabla u_+-\tilde{A}_-k^{-1}\nabla u_-+\tfrac{i}{2}(\tilde{b}_+-\tilde{b}_-)u)|_{\Gamma_p}=0
\end{gathered}\Bigg\},&n\geq 2.\end{cases} 
$$
For $n\leq -1$, we define $\Zspace{n}=(\Zspace{-n})^*$ and $\Hspace{n}=(\Hspace{-n})^*$. Then we set \beq\label{e:capcupdefs}
\ZcapHspace{n} := \Zspace{n} \cap \Hspace{n}, \qquad \ZcupHspace{n} := \Zspace{n} + \Hspace{n}.
\eeq
By the inclusion $\Zspace{n} \subset \Hspace{n}$ for $n\geq 0$ and duality,  %(and thus $\ZcapHspace{n}= \Zspace{n}$ and $\ZcupHspace{n}= \Hspace{n}$), 
%while the opposite inclusion is true for $n \leq 0$.
%Therefore
\beqs
\ZcapHspace{n} = \begin{cases}
\Zspace{n} & \textup{if } n \geq 0,\\
\Hspace{n} & \textup{if } n \leq 0,
\end{cases} \quad \tand \quad \ZcupHspace{n} = \begin{cases}
\Hspace{n} & \textup{if } n \geq 0,\\
\Zspace{n} & \textup{if } n \leq 0.
\end{cases}
\eeqs
Thus, since $\Zspace{-n}= (\Zspace{n})^*$, 
\beqs%\label{e:dualNewSpaces}
\ZcupHspace{-n} = (\ZcapHspace{n})^*, \quad\tfa n\in \mathbb{Z}.
\eeqs
One may think of $\cW^n$ (respectively $\cY^n)$ as the space ``with'' (respectively ``without'') boundary conditions, and the application of $P$ ``removes'' the boundary conditions. 

%When $v\in \ZcapHspace{n}$ for $n\geq 0$, we tefore use interchangeably $\|v\|_{H^{n}_k}$
% and 
% $\|v\|_{\ZcapHspace{n}}$, although if the boundary conditions are important, we use the latter.

For use later, we fix $\Psi_0,\Psi \in C_c^\infty(\mathbb{R};[0,\infty))$ such that 
\begin{equation}\label{e:Phi}
\begin{gathered}
\Re P+\Psi_0(\Re P)\geq c>0,\qquad \Re P_\ell+\Psi_0(\cP_\ell)\geq c>0
\\    
\supp (I-\Psi)\cap \supp \Psi_0=\emptyset. 
\end{gathered}
\end{equation}
We then set $S:=\Psi_0(\Re P)$ and define
\beq\label{e:Xnorm}
\| u\|^2_X:= 
\| \Psi(\mathcal{P})u\|^2_{\ZcapHspace{p+1}}+(\Hc k)^{-2
p}\|(I-\Psi(\mathcal{P}))u\|^2_{H_k^1}.
\eeq
and similarly
\beq\label{e:XnormLocal}
\| u\|^2_{X(\Omega_\ell)}:= 
\| \Psi(\mathcal{P})u\|^2_{H^{p+1}_k(\Omega_\ell\setminus \Gamma_p)}+(\Hc k)^{-2
p}\|(I-\Psi(\mathcal{P}))u\|^2_{H_k^1(\Omega_\ell)}.
\eeq

\section{Preasymptotic analysis of coarse problem}

\begin{lemma}\label{l:coarse}
If $(\Hc k)^{2p}\rho \leq c$ then $\Qcoarse:\fine \to \coarse$ defined by \eqref{e:galerkinCoarse} exists, is unique, and satisfies 
\beq\label{e:firstComp}
\| (I-\Qcoarse) v_h\|_X \leq C (1 + (\Hc k)^{2p}\rho) (H_ck)^{-p}\|(I-\Pi_{H_c})v_h\|_{\cZ} .
%\leq C \|v_h\|_X.
\eeq
Furthermore,
\beq\label{e:FEMLowbound}
\| (I-\Qcoarse) v_h\|_{\ZcapHspace{1-p}}\leq C (\Hc k)^{2p}\rho (H_ck)^{-p}\|(I-\Pi_{H_c})v_h\|_{\cZ}.
\eeq
% Furthermore, if $\chi_{\rm PML}\in C^\infty(\overline{\Omega})$ with $\supp \chi_{\rm PML}$ contained in the PML region of $\Omega$, then 
% \beq\label{e:FEMPMLbound}
% \|\chi_{\rm PML} (I-\Qcoarse) v_h\|_{\ZcapHspace{1-p}}\leq C (\Hc k)^{2p}
% \| (I-\Qcoarse) v_h\|_X.
% \eeq
\end{lemma}

\bpf
By the definition of $\|\cdot\|_X$ \eqref{e:Xnorm}, 
\begin{align}\nonumber
\|(I-\Qcoarse)v_h\|_{X}^2
& = \| \Psi( \cP) (I-\Qcoarse)v_h\|_{\ZcapHspace{p+1}}^2 + (\Hc k)^{-2p} \| (I- \Psi(\cP) ) (I-\Qcoarse)v_h\|_{\cZ}^2.%\\ \nonumber
%&\leq \| \Psi( \cP) (I-\Qcoarse)v_h\|_{\ZcapHspace{p+1} }
%+ \|(\Psi(\cP) - \Psi(\cP))(I-\Qcoarse)v_h\|_{\ZcapHspace{p+1}}
%\\
%&\qquad+ (\Hc k)^{-p} \| (I- \Psi(\cP) ) (I-\Qcoarse)v_h\|_{H^1_k}.
\label{e:hFEMX}
\end{align}
% The idea for the split $I= \Psi( \cP) + I-\Psi( \cP))$ is that the (high) frequencies in $I-\Psi( \cP)$ that lead to the ``bad'' mapping property \eqref{e:PsiMap1} have a ``good'' FEM estimate, while the low frequencies $\Psi( \cP)$ have a ``good'' mapping property and a ``bad'' FEM estimate.
To prove \eqref{e:firstComp}, we prove the two bounds 
\beq\label{e:hFEMX1}
\| \Psi( \cP) (I-\Qcoarse)v_h\|_{\ZcapHspace{p+1} }\leq C
(\Hc k)^{p}\rho \| (I-\Pi_{H_c})v_h\|_{\cZ},
\eeq
% \begin{align}
% \|(\Psi(\cP) - \Psi(\cP))(I-\Qcoarse)v_h\|_{\ZcapHspace{p+1}}
% \leq C (1+ (\Hc k)^{2p}\rho)  \|v_h\|_X,
% \label{e:hFEMX2}
% \end{align}
and
\begin{align}
 \| (I- \Psi(\cP) ) (I-\Qcoarse)v_h\|_{\cZ}\leq C \big(1+ (\Hc k)^{2p}\rho\big)\|(I-\Pi_{H_c})v_h\|_{\cZ}.   \label{e:hFEMX3}
\end{align}
We start with \eqref{e:hFEMX1}, i.e., 
the low-frequencies of FEM error. Here we use the bounds proved in \cite{GS3} (recapped in \cite[\S6]{GSAN}). 
By \cite[Equation 6.24]{GSAN}, 
\beq\label{e:psi1}
\|\Psi(\cP)\|_{\ZcupHspace{1-p} \to \ZcapHspace{p+1}}+\|\Psi_0(\cP)\|_{\ZcupHspace{1-p} \to \ZcapHspace{p+1}}\leq C
\eeq
and, by \cite[Equations 6.33 and 6.43]{GSAN}, if $(\Hc k)^{2p}\rho\leq c$ then
\begin{align}
&   \|(I-\Qcoarse)v_h\|_{\ZcapHspace{1-m}}
\label{e:hFEMLow}
\leq C \big((\Hc k)^m + (\Hc k)^p \rho \big) \| (I-\Pi_\Hc) v_h\|_{\cZ }.
\end{align}
This last inequality with $m=p$ is \eqref{e:FEMLowbound} (since $\rho \geq c$), and then 
combined with the bound \eqref{e:psi1} it gives \eqref{e:hFEMX1}.
%and~\eqref{e:FEMLowbound}.

We now prove~\eqref{e:hFEMX3} by modifying arguments in \cite{AGS2}. Define $\Pi_S:\cZ\to \coarse$ by 
\beq\label{e:PiS}
\langle (P+S)w_h,(I-\Pi_S)u\rangle =0,\qquad \tfa w_h\in\coarse.
\eeq
Then, by the definition of $(P^*)^{-1}$, Galerkin orthogonality, the definition $P_S:=P+S$, and the smoothing property of $S$ \eqref{e:psi1},
\begin{align}\nonumber
&\big\langle (I- \Psi(\cP))(I- \Qcoarse) v_h, v\big\rangle\\ \nonumber
&= \big\langle P(I- \Qcoarse) v_h, (P^*)^{-1}
(I- \Psi(\cP))v\big\rangle
\\ \nonumber
&= \big\langle P((I- \Qcoarse) v_h), (I-\Pi_S)(P^*)^{-1}
(I- \Psi(\cP))v\big\rangle\\ \nonumber
&= \big\langle P_S(I-\Pi_{H_c})v_h, (I-\Pi_S)(P^*)^{-1}
(I- \Psi(\cP))v\big\rangle
- \big\langle S ((I- \Qcoarse) v_h), (I-\Pi_S)(P^*)^{-1}(I- \Psi(\cP))v\big\rangle
\\ \nonumber
&
\leq C \| (I-\Pi_{H_c})v_h\|_{\cZ} 
\| (P^*)^{-1}(I- \Psi(\cP))\|_{\cZ^*\to \cZ}
\|v\|_{\cZ^*}\\
&\hspace{2cm} 
+ C \|(I- \Qcoarse) v_h\|_{\ZcupHspace{1-p}}\|(I-\Pi_S)(P^*)^{-1} (I- \Psi(\cP))v\|_{\ZcupHspace{1-p}}.
\label{e:sun1}
    \end{align}
By \eqref{e:Phi}, $P+S$ is $k$-uniformly coercive (see \cite[Theorem 6.21]{GSAN}) and then a standard Aubin-Nitsche-type duality argument (see \cite[Lemma 8.2/Corollary 8.3]{AGS2}) implies that
\begin{align}\nonumber
\|(I-\Pi_S)(P^*)^{-1} (I- \Psi(\cP))v\|_{\ZcupHspace{1-p}}
&\leq (\Hc k)^p 
\|(I-\Pi_{H_c})(P^*)^{-1} (I- \Psi(\cP))v\|_{\cZ}\\
&\leq 
(\Hc k)^p \| (P^*)^{-1}(I- \Psi(\cP))\|_{\cZ^*\to \cZ}\|v\|_{\cZ^*}.
\label{e:sun2}
    \end{align}
We claim that 
\beq\label{e:resolveHigh}
\| (P^*)^{-1}(I- \Psi(\cP))\|_{\cZ^*\to \cZ}\leq C;
\eeq
combining this with \eqref{e:sun1} and \eqref{e:sun2} gives
\begin{align*}
\|(I- \Psi(\cP))(I-\Qcoarse)v_h\|_{\cZ}\leq 
C\| (I-\Pi_{H_c})v_h\|_{\cZ}
+ C(\Hc k)^p \|(I- \Qcoarse) v_h\|_{\ZcupHspace{1-p}},
\end{align*}
so that by \eqref{e:hFEMLow} with $m=p$, the bound \eqref{e:hFEMX3} follows.

It therefore remains to prove the bound \eqref{e:resolveHigh}.
Since $P_S^*:=P^*+S$,
\beqs
(P^*+S)(P^*)^{-1}= I+ S(P^*)^{-1},
\eeqs
and thus
\begin{equation}\label{e:EFF}
    (P^*)^{-1} = (P^*_S)^{-1} + (P^*_S)^{-1} S (P^*)^{-1}
\end{equation}
so that 
\begin{equation}\label{e:sun3}
    (P^*)^{-1}(I- \Psi(\cP)) = (P^*_S)^{-1}(I- \Psi(\cP)) + (P^*_S)^{-1} S (P^*)^{-1}(I- \Psi(\cP)).
\end{equation}
By \eqref{e:psi1}, $\|I-\Psi(\cP)\|_{\cZ\to \cZ}\leq C$
so that, by self-adjointness, $\|I-\Psi(\cP)\|_{\cZ^* \to \cZ^*}\leq C$. Since $\|(P^*_S)^{-1}\|_{\cZ^*\to \cZ}\leq C$, the $\cZ^*\to \cZ$ norm of the first term in \eqref{e:sun3} is $\leq C$.

To deal with the second term in \eqref{e:sun3}, 
we let $\varphi,\widetilde\varphi \in C^\infty(\overline{\Omega})$ be such that $\supp \varphi \cap \Gamma_{\tr}=\emptyset$
$\supp(1-\varphi)\cap \supp \widetilde \varphi =\emptyset$ (i.e. $\widetilde\varphi\prec \varphi$) and $1-\widetilde\varphi$ is supported in the PML region (note that this last condition requires that the support of $\varphi$ extend into the PML region).
We write
\begin{align}(P^*_S)^{-1} S (P^*)^{-1}(I- \Psi(\cP))=(P^*_S)^{-1} S (P^*)^{-1}(I- \Psi(\cP))\varphi+(P^*_S)^{-1} S (P^*)^{-1}(I- \Psi(\cP))(1-\varphi).\label{e:sun4}
\end{align}
By \cite[Lemma 8.6]{AGS2} and the fact that 
$\supp (I-\Psi)\cap \supp \Psi_0=\emptyset$, 
\begin{align}\label{e:highFreqGood}
\|(P^*)^{-1}(I- \Psi(\cP))\varphi\|_{\cZ^*\to \cZ}\leq C,
\end{align}
so that the $\cZ^*\to \cZ$ norm of the first term on the right-hand side of \eqref{e:sun4} is bounded (Here, we use~\eqref{e:imagParts} to see that, near $(\Gamma_p\cup \partial\Omega_-)\setminus \overline{U}$, $\Im (e^{i\theta_2}P)=0$ so that $P=(1-i\tan\theta_2)\cP$ and hence the proof of~\cite[Lemma 8.6]{AGS2}, involving \cite[Lemma 6.9]{AGS2}, applies.) 
Furthermore,
\begin{align*}
    &(P^*_S)^{-1} S (P^*)^{-1}(I- \Psi(\cP))(1-\varphi)\\
    &=(P^*_S)^{-1} S (P^*)^{-1}(1-\widetilde\varphi)(I- \Psi(\cP))(1-\varphi)
    + (P^*_S)^{-1} S (P^*)^{-1}\widetilde\varphi(I- \Psi(\cP))(1-\varphi)
\\
    &= (P^*_S)^{-1} S (P^*)^{-1}(1-\widetilde\varphi)(I- \Psi(\cP))(1-\varphi)+ O(k^{-\infty})_{(H^N_k)^*\to H^N_k}
\end{align*}
by spatial pseudolocality of $\Psi(\cP)$ \cite[Theorem 6.2]{AGS2}.
Finally, by \cite[Lemma 4.4]{GLS2} (see also Lemma~\ref{l:absorb} below with $\Omega_\ell$ replaced by $\Omega$), $\|(P^*)^{-1}(1-\widetilde\varphi)\|_{\cZ^*\to \cZ}\leq C$ since 
$1-\widetilde\varphi$ is supported in the absorbing layer.
We have therefore proved \eqref{e:resolveHigh}, and the proof is complete.
\epf

\section{FEM analysis of the subdomain problems}

For $1\leq s<\frac{3}{2}$, let
    $$
    \cZ_\ell^s:=
H_{\partial\Omega}^1(\Omega_\ell)\cap H_k^s(\Omega_\ell),\qquad \|u\|_{\cZ^s_\ell}:=\|u\|_{H_k^s(\Omega_\ell)},
    $$
    with $\cZ_\ell^0=L^2$ and $\cZ_\ell^s$ for $s\in (0,1)$ defined by interpolation. We then set $\cZ_\ell^{s}:=(\cZ_\ell^{-s})^* $ for $-\frac{3}{2}<s<0$.

\begin{lemma}[Solution operator bounds in the absorbing region]
\label{l:absorb}
    Let $c>0$ and $\{D_{\alpha}\}_{\alpha\in\mathbb{N}^d}$ and $-\frac{1}{2}<s<\frac{1}{2}$. Then 
    there exists $C>0$ such that 
    for all $\varphi\in \overline{C^\infty}(\Omega_\ell)$ with $\supp \varphi\subset U_\ell$ and 
    $$
    \delta(\varphi,\{D_\alpha\},2)\geq c\delta_\ell,
    $$
     we have
    \beq\label{e:lotsofs}
    \|P_{\ell}^{-1}\varphi \|_{\cZ_\ell^{s-1}\to \cZ_\ell^{s+1}}+\|\varphi P_{\ell}^{-1} \|_{\cZ_\ell^{s-1}\to \cZ_\ell^{s+1}}\leq C.
    \eeq
    Moreover for all $M>0$, there is $C>0$ such that for all $\varphi_2\in \overline{C^\infty}(\Omega_\ell)$ with 
    $$
    d(\supp \varphi,\supp \varphi_2)\geq c\delta_\ell,\qquad  
    \delta(\varphi_2,\{D_\alpha\},2)\geq c\delta_\ell,
    $$
    and 
      \beq\label{e:lotsofs2}
    \|\varphi_2P_{\ell}^{-1}\varphi \|_{\cZ_\ell^{s-1}\to \cZ_\ell^{s+1}}+\|\varphi P_{\ell}^{-1}\varphi_2 \|_{\cZ_\ell^{s-1}\to \cZ_\ell^{s+1}}\leq C\langle k\delta_\ell\rangle ^{-M}.
    \eeq
       \end{lemma}
\begin{proof}
    To prove Lemma~\ref{l:absorb}, we set $y:=\delta_\ell^{-1}x$, rescaling the subdomain problem on $\Omega_\ell$ to one with scale $\sim 1$ and $\varphi$ to a cutoff with scale $\sim 1$. With $\tilde{k}=k\delta_\ell$, The operator $P_\ell$ then becomes 
    $$
    \tilde{P}_\ell=- \tilde{k}^{-2}\operatorname{div}_y (A_\ell (y\delta_\ell) \nabla_y u )-i\tilde{k}^{-1} b_\ell(y\delta_\ell)\cdot \nabla_y -n_\ell(y\delta_\ell)
    $$
Since there is $C>0$ such that 
$$
C^{-1}\|u\|_{H_{k}^s(\Omega_\ell)}\leq \delta_{\ell}^{d/2}\|u\|_{H_{\tilde{k}}^s(\delta_\ell^{-1}\Omega_\ell)}\leq C\|u\|_{H_{k}^s(\Omega_\ell)},
$$
it is enough to prove estimates in $\tilde{\cZ}_\ell^s$ spaces; i.e. rescaled $\cZ^s_\ell$ space with the $H_{\tilde{k}}^s$ norm. In the proof of this lemma, we will now abuse notation by writing $\Omega_\ell$ for $\delta_{\ell}^{-1}\Omega_\ell$ and $U_\ell$ for $\delta_{\ell}^{-1}U_\ell$.

    By Lemma~\ref{l:absorbMe}, for all  $u\in H_{\partial\Omega }^1(\Omega_\ell)$ with $\supp u\subset \delta_\ell^{-1} U_\ell$, 
    \begin{equation} \label{e:absorbMe}
\|u\|_{\tilde{\cZ}_\ell^1}\leq C\|\tilde{P}_\ell u\|_{\tilde{\cZ}_\ell^{-1}}.   
    \end{equation}   
   Then, for $\tilde{\varphi}\in \overline{C^\infty}(\Omega_\ell)$ with $\supp \tilde{\varphi}\subset \overline{U_\ell}$ and $\supp(1-\tilde{\varphi})\cap \supp \nabla \varphi=\emptyset$, and $\supp \tilde{\varphi}\cap \partial\Omega_\ell\setminus \partial\Omega=\emptyset$, 
    $$
    \tilde{P}_\ell \varphi u= [\tilde{P}_\ell,\varphi]\tilde{\varphi}u+\varphi \tilde{P}_\ell u
    $$
    and hence, using~\eqref{e:absorbMe}
    $$
    \|\varphi u\|_{ \tilde{\cZ}_\ell^1}\leq C\tilde{k}^{-1}\|\tilde{\varphi}u\|_{\tilde{\cZ}_\ell^1}+C\|\varphi \tilde{P}_\ell u\|_{\tilde{\cZ}_\ell^{-1}}
    $$
    Iterating this argument, and letting $\varphi_+\in \overline{C^\infty}(U_\ell)$ with $\supp (1-\varphi_+)\cap \supp \varphi=\emptyset$ and $\supp \varphi_+\cap \supp \varphi_1=\emptyset$ we obtain
    $$
    \|\varphi u\|_{\tilde{\cZ}_\ell^1}\leq C\tilde{k}^{-N}\|\varphi_+u\|_{\tilde{\cZ}_\ell^1}+C\|\varphi_+\tilde{P}_{\ell}u\|_{\tilde{\cZ}_\ell^{-1}},
    $$
    and hence
    \beq\label{e:plant1}
    \|\varphi \tilde{P}_{\ell}^{-1}\|_{\tilde{\cZ}_\ell^{-1}\to \tilde{\cZ}_\ell^1}\leq C(1+\tilde{k}^{-N}\rho_\ell)\leq C,
    \eeq
    and 
       \beq\label{e:plant1a}
    \|\varphi \tilde{P}_{\ell}^{-1}\varphi_1\|_{\tilde{\cZ}_\ell^{-1}\to \tilde{\cZ}_\ell^1}\leq C\tilde{k}^{-N}\rho_\ell.
    \eeq
    Now, by elliptic regularity on a Lipschitz domain (see, e.g., \cite[Theorem 3]{savare1998regularity}, \cite{JeKe:95}), for $0\leq s<\frac{1}{2}$, 
    \begin{equation} \label{e:nonAdjoint}
    \|\varphi \tilde{P}_{\ell}^{-1}\|_{\tilde{\cZ}_\ell^{s-1}\to \tilde{\cZ}_\ell^{s+1}}\leq C(1+\|\varphi \tilde{P}_{\ell}^{-1}\|_{\tilde{\cZ}_\ell^{-1}\to\tilde{\cZ}_{\ell}^1})\leq C.
    \end{equation}
    and
        \begin{equation} \label{e:nonAdjointb}
    \|\varphi \tilde{P}_{\ell}^{-1}\varphi_1\|_{\tilde{\cZ}_\ell^{s-1}\to \tilde{\cZ}_\ell^{s+1}}\leq C(\|\varphi \tilde{P}_{\ell}^{-1}\varphi_1\|_{\tilde{\cZ}_\ell^{-1}\to\tilde{\cZ}_{\ell}^1})\leq C\tilde{k}^{-N}.
    \end{equation}
 Now, we show that 
 \begin{equation}
 \label{e:adjoint}
 \|\tilde{P}_\ell^{-1}\varphi\|_{\tilde{\cZ}_\ell^{s-1}\to \tilde{\cZ}_\ell^{s+1}}\leq C;
 \end{equation}
i.e., the result \eqref{e:lotsofs} for $0\leq s<1/2$.
  Indeed, observe that for $\varphi$ as above, and  $\tilde{\varphi}\in \overline{C^\infty}(\Omega_\ell)$ with $\supp \tilde{\varphi}\subset \overline{U_\ell}$ and $\supp(1-\tilde{\varphi})\cap \supp \nabla \varphi=\emptyset$, and $\supp \tilde{\varphi}\cap \partial\Omega_\ell\setminus \partial\Omega=\emptyset$, 
 $$
 \tilde{P}_{\ell}^{-1}\varphi = \varphi \tilde{P}_{\ell}^{-1}+ [P_{\ell}^{-1},\varphi]=\varphi \tilde{P}_{\ell}^{-1}- \tilde{P}_{\ell}^{-1}\tilde{\varphi}[\tilde{P}_{\ell},\varphi]\tilde{\varphi}\tilde{P}_{\ell}^{-1}
 $$
 Hence, by \eqref{e:plant1},
$$\|\tilde{P}_{\ell}^{-1}\varphi\|_{\tilde{\cZ}_\ell^{-1}\to \tilde{\cZ}_{\ell}^1}\leq C+C\|\tilde{P}_{\ell}^{-1}\tilde{\varphi}\|_{\tilde{\cZ}_{\ell}^{-1}\to \tilde{\cZ}_\ell^{1}}\tilde{k}^{-1}.
 $$
By iteration and the polynomial bound on $\rho_\ell$,
 $$
 \|\tilde{P}_{\ell}^{-1}\varphi\|_{\tilde{\cZ}_\ell^{-1}\to \tilde{\cZ}_{\ell}^1}\leq C, 
 $$
 and hence~\eqref{e:adjoint} holds by elliptic regularity. The fact that
 $$
 \|\varphi_1\tilde{P}_\ell^{-1}\varphi\|_{\tilde{\cZ}_\ell^{-1}\to \tilde{\cZ}_{\ell}^1}\leq C\tilde{k}^{-N}
 $$
 follows in the same way.

 The same arguments applied to $\tilde{P}_{\ell}^*$ imply that for $0\leq s<\frac{1}{2}$
 \begin{equation}
 \label{e:adjoint2}
 \begin{gathered}
\|\varphi(\tilde{P}^*_\ell)^{-1}\|_{\tilde{\cZ}_\ell^{s-1}\to 
 \tilde{\cZ}_\ell^{s+1}}+\|(\tilde{P}^*_\ell)^{-1}\varphi\|_{\tilde{\cZ}_\ell^{s-1}\to 
 \tilde{\cZ}_\ell^{s+1}}\leq C,\\
 \|\varphi(\tilde{P}^*_\ell)^{-1}\varphi_1\|_{\tilde{\cZ}_\ell^{s-1}\to 
 \tilde{\cZ}_\ell^{s+1}}+ \|\varphi_1(\tilde{P}^*_\ell)^{-1}\varphi\|_{\tilde{\cZ}_\ell^{s-1}\to 
 \tilde{\cZ}_\ell^{s+1}}\leq C\tilde{k}^{-N}.
 \end{gathered}
 \end{equation}
 and hence the result \eqref{e:lotsofs} for $-1/2<s<0$
 follows from~\eqref{e:adjoint2} and  duality.
\end{proof}

\begin{lemma}[Solution operator bounds at high frequency]\label{l:HFsubdomain}
With $\Phi$ as in \eqref{e:Phi}, 
there exists $C>0$ such that 
    \beq\label{e:lotsofHigh}
    \| (I-\Psi(\cP_\ell))P_{\ell}^{-1} \|_{\cZ_\ell^*\to \cZ_\ell}\leq C(\rho_\ell (k\delta_\ell)^{-1}+1).
    \eeq
\end{lemma}
\begin{proof}
    By \eqref{e:Phi},
    \begin{align*}
    \Re \big\langle P_\ell(I-\Psi(\cP_\ell))u,(I-\Psi(\cP_\ell))u\big\rangle 
    &= \big\langle \cP_\ell (I-\Psi(\cP_\ell)),(I-\Psi(\cP_\ell))u\big\rangle 
    \\
    &=\big\langle \big(\cP_\ell + \Psi_0(\Psi)\big) (I-\Psi(\cP_\ell)),(I-\Psi(\cP_\ell))u\big\rangle
    \geq c\|(I-\Psi(\cP_\ell))u\|_{\cZ}^2,
            \end{align*}
    so that 
    $$
    \|(I-\Psi(\cP_\ell))u\|_{\cZ_\ell}\leq C\|P_\ell(I-\Psi(\cP_\ell))u\|_{\cZ_\ell^*}.
    $$
Thus, since 
$$
    P_\ell (I-\Psi(\cP_\ell))P_\ell^{-1}=-[P_\ell,\Psi(\cP_\ell)]P_\ell^{-1}+(I-\Psi(\cP_\ell)),
    $$
we have that 
    \begin{align}\nonumber
    \|(I-\Psi(\cP_\ell))P_\ell^{-1}\|_{\cZ^*_\ell\to \cZ_\ell}&\leq \|[P_\ell,\Psi(\cP_\ell)]P_\ell^{-1}\|_{\cZ_\ell^*\to \cZ_\ell^*}+\|(I-\Psi(\cP_\ell))\|_{\cZ_\ell^*\to \cZ^*_\ell}\\
    &\leq \|[P_\ell,\Psi(\cP_\ell)]P_\ell^{-1}\|_{\cZ_\ell^*\to \cZ_\ell^*}+C.
    \label{e:lotsofZells}
    \end{align}
        Now, 
        by \eqref{e:subdomainCoeff2} there exist $\varphi_{\rm in}+\varphi_{\rm abs}+\varphi_{\rm out}=1$ on $\Omega_\ell$ such that 
\begin{gather}\label{e:phiinout}
    \supp \varphi_{\rm abs}\subset U_\ell,\quad \supp \varphi_{\rm in}\cap (\partial\Omega_{\ell}\cup \Gamma_p)=\emptyset,\quad \supp \varphi_{\rm out}\cap( 
    \supp \Im e^{i\theta_2}A_\ell
    \cup 
    \supp \Im e^{i\theta_2} b_\ell)=\emptyset,\\
    \nu \cdot A_\ell\nabla \varphi_{\rm out}|_{\partial\Omega_\ell}=0,\qquad \nu \cdot (A_+\nabla (\varphi_{\rm out})_+-A_-\nabla (\varphi_{\rm out})_-)|_{\Gamma_p}=0
    \label{e:phioutD3}
    \end{gather}
    (i.e., $\varphi_{\rm out}$ can touch the boundary/interfaces only when $\Im e^{i\theta_2}A_\ell$ and $\Im e^{i\theta_2} b_\ell$ are zero).
    Thus
    $$
    [P_\ell,\Psi(\cP_\ell)]P_\ell^{-1}=[P_\ell,\Psi(\cP_\ell)]\varphi_{\rm abs}P_\ell^{-1}+[P_\ell,\Psi(\cP_\ell)]\varphi_{\rm in}P_\ell^{-1}+[P_\ell,\Psi(\cP_\ell)]\varphi_{\rm out}P_\ell^{-1}
    $$
    By Lemma~\ref{l:absorb} and the facts that $\Psi(\cP_\ell):\cZ_\ell\to \cZ_\ell$, $\Psi(\cP_\ell):\cZ_\ell^*\to \cZ_\ell^*$ and $P_\ell:\cZ_\ell\to \cZ_\ell^*$ are all uniformly bounded,
    $$
    \|[P_\ell,\Psi(\cP_\ell)]\varphi_{\rm abs}P_\ell^{-1}\|_{\cZ_\ell^*\to \cZ_\ell^*}\leq C\|[P_\ell,\Psi(\cP_\ell)]\|_{\cZ_\ell\to \cZ_\ell^*}\leq C.
    $$

        Next, write 
        \begin{equation}
        \label{e:commutators}
        \begin{gathered}
[P_\ell,\Psi(\cP_\ell)]\varphi_{\rm in}= [P_\ell \varphi_{\rm in},\Psi(\cP_\ell)]+P_\ell[\Psi(\cP_\ell),\varphi_{\rm in}],
\\
[P_\ell,\Psi(\cP_\ell)]\varphi_{\rm out}= [P_\ell \varphi_{\rm out},\Psi(\cP_\ell)]+P_\ell[\Psi(\cP_\ell),\varphi_{\rm out}].
\end{gathered}
\end{equation}
    By the Helffer--Sj\"ostrand formula (see, e.g., \cite[Theorem 14.8]{Zw:12}), 
    $$
    \Psi(\cP_\ell)=\frac{1}{\pi}\int_{\mathbb{C}} \partial_{\bar{z}}\tilde{\Psi}(z)(\cP_\ell-z)^{-1}dm_{\mathbb{C}}(z),
    $$
    where $\tilde{\Psi}(z)$ is an almost-analytic extension of $\Psi$ (see, e.g., \cite[Theorem 3.6]{Zw:12}), 
$$
[A,\Psi(\cP_\ell)] 
= \frac{1}{\pi}\int_{\mathbb{C}} \partial_{\bar{z}}\tilde{\Psi}(z)[A,(\cP_\ell-z)^{-1}]dm_{\mathbb{C}}(z)
= \frac{1}{\pi}\int_{\mathbb{C}} \partial_{\bar{z}}\tilde{\Psi}(z)(\cP_\ell-z)^{-1}[\cP_\ell,A](\cP_\ell-z)^{-1}dm_{\mathbb{C}}(z).
$$
Therefore, to show that 
    $$
    \|[A,\Psi(\cP_\ell)]\|_{\cZ_\ell\to \cZ_\ell}\leq C(k\delta_\ell)^{-1},\qquad \|[B,\Psi(\cP_\ell)]\|_{\cZ_\ell\to \cZ_\ell^*}\leq C(k\delta_\ell)^{-1},
    $$
    it is enough to show that
    \begin{equation} \label{e:commutatorAB}
    \|[A,\cP_\ell]\|_{\mathcal{D}_{\ell}^3\to \cZ_\ell^*}\leq C(k\delta_\ell)^{-1},\qquad  \|[B,\cP_\ell]\|_{\mathcal{D}_\ell^3\to \mathcal{D}_{\ell}^{-3}}\leq C(k\delta_\ell)^{-1},
    \end{equation}
    where, letting $M>0$ large enough so that $(\cP_\ell+M)>c>0$, 
    $$
    \cD_\ell^s:=(\cP_\ell+M)^{-s/2}L^2.
    $$
    (Note that $\cD_\ell^1=\cZ$ and $\cD_\ell^{-1}=\cZ^*$
    and recall that $(\cP-z)^{-1}: \cD_\ell^s \to \cD_\ell^{s+2}$ for $z\in \mathbb{C}\setminus \mathbb{R}$; see, e.g., \cite[Proposition 5.21]{AGS2}.)
    We now prove~\eqref{e:commutatorAB} with $A=\varphi_{\rm in/\rm out}$ and $B=P_\ell\varphi_{\rm in/\rm out}$; the result then follows from combining \eqref{e:commutatorAB} with \eqref{e:commutators} and \eqref{e:lotsofZells}.

    For $A=\varphi_{\rm in/\rm out}$, observe that $A:\cZ\to \cZ$ and hence for $u, v\in\cZ$, 
    $$
    |\langle [A,\cP_\ell]u,v\rangle|= |(\Re a_\ell)(Au,v)-(\Re a_\ell)(u,Av)|\leq C\tilde{k}^{-1}\|u\|_{\cZ}\|v\|_{\cZ},
    $$
    hence $\|[A,\cP_\ell]\|_{\cZ_\ell\to \cZ_\ell^*}\leq C\tilde{k}^{-1}$ as required.

    The computation for $B$ is more complicated. For $B=P_\ell \varphi_{\rm in }$, the operators are supported away from all boundaries; hence the claim follows from the fact that $P_\ell$ and $\cP_\ell$ are semiclassical differential operators with coefficients having derivatives of size $\delta_\ell^{-1}$. 

    For $B=P_\ell \varphi_{\rm out}$, observe that since $\varphi_{\rm out}:\cZ_\ell\to \cZ_\ell$ and $\varphi_{\rm out}:\cD_\ell^3\to \cD_\ell^3$
    (by the conditions \eqref{e:phioutD3}), the support condition \eqref{e:phiinout} on $\varphi_{\rm out}$ implies that $P_\ell \varphi_{\rm out}:\cD_\ell^3\to \cZ$. 
    
 Furthermore, the support condition on $\varphi_{\rm out}$ combined with the assumption \eqref{e:subdomainCoeff2} implies that $P= (1-i\tan \theta_2)\cP_\ell  +V_\ell$, where $V_\ell:=- i \sec \theta_2\Im n_\ell$, on $\supp \varphi_{\rm out}$. (Recall that $\cos \theta_2\neq 0$ since 
 \eqref{e:strongElliptic} holds and 
 we have assumed, without loss of generality, that $\theta_0=0$.)   
    Finally, $V_\ell:\cZ_\ell \to \cZ_\ell$ and hence 
    \begin{align*}
    \langle [P_\ell\varphi_{\rm out},\cP_\ell]u,v\rangle &= \langle P_\ell \varphi_{\rm out}\cP_\ell u,v\rangle -\langle \cP_\ell P_\ell \varphi_{\rm out}u,v\rangle\\
    &=a_\ell (\varphi_{\rm out}\cP_\ell u,v) - (\Re a_\ell)(P_\ell \varphi_{\rm out}u,v)\\
    &=(1-i\tan \theta_2)\Re a_\ell (\varphi_{\rm out}\cP_\ell u,v)-(\Re a_{\ell})(((1-i\tan\theta_2)\cP_\ell+V_\ell) \varphi_{\rm out}u,v)\\
    &\qquad+\langle V_\ell\varphi_{\rm out}\cP_\ell u,v\rangle\\
    &=(1-i\tan \theta_2)\Re a_\ell ([\varphi_{\rm out},\cP_\ell] u,v)-(\Re a_{\ell})(V_\ell \varphi_{\rm out}u,v)+(\Re a_\ell)( u,\varphi_{\rm out}\overline{V_\ell}v)\\
    &=(1-i\tan \theta_2)\langle \cP_\ell [\varphi_{\rm out},\cP_\ell]u,v\rangle -(\Re a_{\ell})(V_\ell \varphi_{\rm out}u,v)+(\Re a_\ell)( u,\varphi_{\rm out}\overline{V_\ell}v).
    \end{align*} 
 Now, 
 $$
 \|\cP_\ell [\varphi_{\rm out},\cP_\ell]\|_{\mathcal{D}_{\ell}^3\to \mathcal{D}_\ell^{-3}}\leq C\|[\varphi_{\rm out},\cP_\ell]\|_{\cZ_\ell\to \cZ_\ell^*}\leq C(k\delta_\ell)^{-1},
 $$
 and 
 $$
    |(\Re a_{\ell})(V_\ell \varphi_{\rm out}u,v)-(\Re a_\ell)( u,\varphi_{\rm out}\overline{V_\ell}v)|\leq C(k\delta_\ell)^{-1}\|u\|_{\cZ}\|v\|_{\cZ}.
    $$
    Hence, 
    $$
    \|[\cP_\ell\varphi_{\rm out},\cP_\ell]\|_{\mathcal{D}_\ell^3\to \mathcal{D}_\ell^{-3}}\leq C(k\delta_\ell)^{-1};
    $$
    the proof of \eqref{e:commutatorAB}, and hence also the result, is therefore complete.    
    \end{proof}

\begin{lemma}[Quasi-optimality of the FEM solution for the subdomain problems]
\label{l:subdomainFEM}
Let $0<s<\frac{1}{2}$. Then there is $c>0$ such that if 
$$
\rho_\ell (h_\ell k)^{p}+(h_\ell k)^{s}\leq c,
$$
then for all $u\in H_{\partial\Omega}^1(\Omega_\ell)$ there is a unique $u_h\in \subdomain$ satisfying
    \beq\label{e:Gogsubdomain}
    a_\ell (u_h,w_h)=a_\ell(u,w_h)\quad \tfa w_h\in \subdomain,
    \eeq
    and
    \beq\label{e:FEMLowReg}
    \|u-u_h\|_{(H_k^{p-1}(\Omega_\ell\setminus \Gamma_p)\cap H^1_{\partial}(\Omega_\ell))^*}\leq C\big((h_\ell k)^p\rho_\ell +(h_\ell k)^{s}\big)\inf_{w_h\in \subdomain}\|u-w_h\|_{H_k^1(\Omega_\ell)}.
    \eeq
\end{lemma}
\begin{proof}
    The proof follows the Schatz duality argument \cite{Sc:74, ScWa:96} (i.e., the Aubin--Nitsche lemma combined with a G\aa rding inequality). 
    Since $\subdomain$ is finite dimensional it is sufficient to prove the bound \eqref{e:FEMLowReg} under the assumption that $u_h$ exists. 
    Let $\varphi$ be as in Lemma~\ref{l:absorb} with $\supp (1-\varphi)\cap \overline{\partial\Omega_\ell\setminus\partial\Omega}=\emptyset$. Observe first that, for all $w_h = w_h^1 + w_h^2 \in \subdomain$,
    $$
    \begin{aligned}
        |\langle u-u_h,v\rangle|&=|\langle P_\ell(u-u_h),(P_\ell^*)^{-1}v-w_h\rangle|\\
        &=\big|\langle P_\ell(u-u_h),\varphi (P_\ell^*)^{-1}v-w^1_h\rangle+\langle P_\ell (u-u_h),(1-\varphi)(P_{\ell}^*)^{-1}v-w_h^2\rangle\big|\\
        &\leq C\|u-u_h\|_{H_k^1}\big((h_\ell k)^{s}\|v\|_{L^2}+(h_\ell k)^p\rho_\ell\|v\|_{H_k^{p-1}(\Omega_\ell \setminus \Gamma_p)}\big)
    \end{aligned}
    $$
    by Lemma \ref{l:absorb} and the polynomial approximation result \eqref{e:assIell}.
        Hence,
    $$
    \|u-u_h\|_{(H_k^{p-1}(\Omega_\ell\setminus \Gamma_p))^*}\leq C\big((h_\ell k)^{s}+(h_\ell k)^p\rho_\ell\big)\|u-u_h\|_{H_k^1}.
    $$
    Now, by the G\aa rding inequality satisfied by $a_\ell$, Galerkin orthogonality \eqref{e:Gogsubdomain}, and interpolation, for all   
      $\e>0$ there is $C_\e>0$ such that for all $w_h\in \subdomain$ 
    \begin{align*}
    \|u-u_h\|_{H_k^1}^2&\leq C|\langle P_\ell(u-u_h),u-u_h\rangle|+C\|u-u_h\|_{L^2}^2\\
    &\leq C_\e \|u-w_h\|^2_{H_k^1}+\e\|u-u_h\|^2_{H_k^1}+\e\|u-u_h\|_{H_k^1}^2+C_\e\|u-u_h\|_{(H_k^{p-1}(\Omega_\ell\setminus \Gamma_p))^*}^2\\
    &\leq C_\e \|u-w_h\|^2_{H_k^1}+\Big(2\e+C_\e\big((h_\ell k)^{s}+(h_\ell k)^p\rho_\ell\big)^2\Big)\|u-u_h\|_{H_k^1}^2.
    \end{align*}
    The result then follows by taking $\e$ small and $(h_\ell k)^s+(h_\ell k)^p\rho_\ell$ small.
\end{proof}

\section{Superapproximation and subdomain overlaps}

%\subsection{Superapproximation}

% \begin{assumption}
% There exists  $\mu_\ell, C>0$ such that
% \beq\label{e:superapproxAss}
% \| (I- \cIh)\chiell v_h\|_{H^1_k(\Omega_\ell)} \leq  \mu_\ell  \|v_h\|_{H^1_k(\Omega_\ell)}
% \eeq
% for all $v_h\in \fine$,
% and
% \beq\label{e:superapproxAss2}
% \| (I- \cIh)u\|_{H^1_k(\Omega_\ell)}^2 \leq C(h_\ell k)^{2p} \sum_{T\in \Omega_\ell}\|u\|_{H^{p+1}_k(T)}^2
% \eeq
%  for all $u\in H^{p+1}(\Omega_\ell)$.
% \end{assumption}

Recall that, by the quasi-interpolation results of \cite[Appendix D]{AGS2} (which cover curved meshes; see also, e.g.,  \cite[Theorem 22.6(i)]{ErGu:21} for analogous results for simplicial meshes), for $\ell=1,\ldots,N$, 
there exists a bounded linear operator
\beq\label{e:QI1}
\widetilde\cI_{\ell} :H^1_{\partial \Omega}(\Omega_\ell)\to \subdomain,\qquad \chiell\widetilde{\cI}_{\ell}:H^1_k(\Omega_\ell)\to \subdomain
\eeq
such that
\beq\label{e:QI2}
\widetilde\cI_{\ell} v_{\ell,h} =
v_{\ell,h} \tfa v_{\ell,h}\in \subdomain,
\eeq
\beq\label{e:QI3}
\chiell \widetilde\cI_{\ell} (1-\chiellg)=0,
\eeq
and 
\beq\label{e:QI4}
\sum_{T\cap \supp \chiell\neq \emptyset}\| (I- \widetilde\cI_{\ell})u\|_{H^{p+1}_k(T)}^2 \leq C \|u\|_{H^{p+1}_k(\Omega_\ell)}^2.
\eeq

The following lemma, which appears new and may be of independent interest, shows that if a function is $H_k^{p+1}$ close to a finite element function, then one still has a type of superapproximation.
\begin{lemma}
\label{l:splitSuper}
There exists $C>0$ such that for all $v_h\in \cV_h$ and $w\in H_k^1(\Omega_\ell)$ and $\chiellg w\in H_k^{p+1}(\Omega_\ell)$, 
\begin{equation} \label{e:superSplit}
\| (I- \cI_\ell)\chiell v_h\|_{H^1_k(\Omega_\ell)} \leq C\Big((h_\ell k)^p\|\chiellg w\|_{H_k^{p+1}}+\mu_\ell\|\chiellg(v_h-w)\|_{H_k^1(\Omega_\ell)}\Big)
\end{equation}
(where $\cI_\ell$ and $\mu_\ell$ are as in \eqref{e:assIell}). 
In particular, there is $C>0$ such that for all $v_h\in \cV_h$, 
    \beq\label{e:superapprox}
\| (I- \cI_\ell)\chiell v_h\|_{H^1_k(\Omega_\ell)} \leq C\Big( (h_\ell k)^p  + (\Hc k)^p \mu_\ell \Big) \|v_h\|_{X(\Omega_\ell)}.
% \| (I- \cI_\ell)\chiell v_h\|_X \leq \bigg( \mu_\ell + \Big(\frac{h_\ell}{\Hc}\Big)^p\Big( \frac{1}{k\delta_\ell}\Big)^{p+1} \bigg) \|v_h\|_{X(\Omega_\ell)}.
\eeq
\end{lemma}
%Scribble Page 69.

\bpf
% We prove that 
% \begin{align}\label{e:SAtoProve}
% \| (I- \cI_\ell)\chiell v_h\|_{H^1_k} \leq \big( (h_\ell k)^p (k\delta_\ell)^{-p-1} + (\Hc k)^p \mu_\ell \big) \|v_h\|_{X(\Omega_\ell)},
% \end{align}
% and then \eqref{e:superapprox} follows from \eqref{e:normEquiv1}.
By \eqref{e:QI2} and \eqref{e:QI3},
\begin{align*}
    (I- \cI_\ell) \chiell v_h 
    &= (I- \cI_\ell)\chiell w 
     +(I- \cI_\ell) \chiell (v_h-w)
     \\
&= (I- \cI_\ell)\chiell w 
     +(I- \cI_\ell) \chiell \big(v_h-\widetilde\cI_{\ell} w + \widetilde\cI_{\ell} w-  w\big)
\\     
     & = (I-\cI_\ell)\chiell w
+ (I- \cI_\ell) \chiell \widetilde\cI_{\ell}(v_h- w)
- (I- \cI_\ell) \chiell (I-\widetilde\cI_{\ell}) w
\\
& = (I-\cI_\ell)\chiell w
+ (I- \cI_\ell) \chiell \widetilde\cI_{\ell}\chiellg(v_h- w)
- (I- \cI_\ell) \chiell (I-\widetilde\cI_{\ell}) \chiellg w. 
\end{align*}
Now, 
% \begin{equation}\label{e:multchiell}
% \| \chiell w \|_{H^{p+1}(\Omega_\ell)}
% \leq C 
% \|w\|_{H^{p+1}_k(\Omega_\ell)}.
% \end{equation}
% and
by % this and
\eqref{e:assIell},
\begin{align*}
    \big\|(I-\cI_\ell)\chiell w
    \big\|_{H^1_k(\Omega_\ell)} 
    \leq C (h_\ell k)^p \| \chiell w\|_{H^{p+1}_k(\Omega_\ell)}
&\leq C (h_\ell k)^p \| \chiellg w\|_{H^{p+1}_k(\Omega_\ell)}.
\end{align*}
By \eqref{e:assIell}, \eqref{e:QI1}, and \eqref{e:XnormLocal},
\begin{align*}
   \big\| (I- \cI_\ell) \chiell \widetilde\cI_{\ell}\chiellg(v_h-w)\big\|_{H^1_k(\Omega_\ell)} 
   &\leq \mu_\ell \big\|\widetilde\cI_{\ell}\chiellg(v_h- w)\big\|_{H^1_k(\Omega_\ell)}\leq C\mu_\ell \big\|\chiellg (v_h-w)\big\|_{H^1_k
   (\Omega_\ell)
   }.
\end{align*}
By \eqref{e:assIell} and \eqref{e:QI4},
\begin{align*}
    \big\|(I- \cI_\ell) \chiell (I-\widetilde\cI_{\ell}) w\big\|_{H^1_k}^2
    &\leq 
    C (h_\ell k)^{2p} 
    \sum_{T\cap \supp \chiell\neq \emptyset}\big\|(I-\widetilde\cI_{\ell}) w\big\|_{H^{p+1}_k(T)}^2
    \leq 
    C (h_\ell k)^{2p} 
    \|\chiellg w\|_{H^{p+1}_k(\Omega_\ell)}^2.
    \end{align*}
This completes the proof of~\eqref{e:superSplit}.

To prove \eqref{e:superapprox}, we take $w=\Psi(\mathcal{P})v_h$ in~\eqref{e:superSplit} and use the definition of $\|\cdot\|_{X(\Omega_\ell)}$ (see~\eqref{e:XnormLocal}) to obtain
\begin{align*}
\|(I-\cI_\ell)\chi_\ell v_h\|_{H_k^1(\Omega_\ell)}&\leq C\Big( (h_\ell k)^p\|\Psi(\cP)v_h\|_{H_k^{p+1}(\Omega_\ell)}+\mu_\ell(H_ck)^p(H_ck)^{-p} \|(I-\Psi(\cP))v_h\|_{H_k^1(\Omega_\ell)}\\
&\leq C\big((h_\ell k)^p+\mu_\ell(H_ck)^p\big)\|v_h\|_{X(\Omega_\ell)}.
\end{align*}

\epf

\subsection{Cutoff functions associated to the overlapping subdomains}\label{s:cutoff}

Given $\chi_{\ell}$, let 
$\chiellg\in \overline{C^\infty}(\Omega_\ell\setminus\Gamma_p)$ be such that 
$$
\nu\cdot A\nabla \chiellg=0\text{ on }\Gamma_{\rm hard}\quad \tand \quad \nu \cdot (A_+\nabla (\chiellg)_+-A_-\nabla (\chiellg)_-)=0\text{ on }\Gamma_p,
$$
(compare to 
\eqref{e:domain2domain}),
\beq\label{e:chiBound}
\delta_\ell\leq \delta(\chiellg,\{D_\alpha\},p+1\big),
\eeq
\begin{align*}%\label{e:cutoffLarger1}
\chiellg P_\ell = \chiellg P,\quad
    d(\chiell, \supp(1-\chiellg))\geq c\delta_\ell, \quad\tand\quad d(\supp\chiellg,\Omega_\ell\setminus W_\ell)\geq c\delta_\ell;
  %  \quad\tand\quad \chiell \cI_\ell (1-\chiellg)=0
\end{align*}
we abbreviate the penultimate condition to $\chiell \prec \chiellg$.

We then define three sets of even larger cut-offs, supported in $\Omega_\ell$, but such that the ``transition'' from zero to one happens in the absorbing region. That is, let $\chiellgg, \chiellggg, \chiellgggg\in \overline{C^\infty}(\Omega_\ell\setminus\Gamma_p;[0,1])$ be such that  $\supp \chiellgggg\subset \Omega_\ell\cup\partial\Omega$, 
\beq
\begin{aligned}
&\qquad \chiellg \prec \chiellgg \prec\chiellggg\prec \chiellgggg,\quad
%\chiellgg(1-\chiellggg)=0 
  d( \supp(1-\chiellggg),\Omega_\ell\setminus U_\ell)\geq c\delta_\ell,\quad\tand\quad d( \chiellgg,\Omega_\ell\setminus W_\ell)\geq c\delta_\ell,
    \end{aligned}
\label{e:lotsofggg}
\end{equation}
and satisfying
$$
\delta_\ell\leq \delta(\chi_\ell^{(j)},\{D_\alpha\},p+1\big), \quad j=2,3,4.
$$
We emphasise that only $\chiell$ appears in the definition of the preconditioner \eqref{e:Q}; the other cutoffs are only used in the analysis.

\subsection{Approximate orthogonality of spaces on subdomains}

\begin{lemma}\label{l:overlap}
 For all $u\in H^1(\Omega)$, 
\begin{align}\label{e:overlap1}
   \sum_{\ell=1}^N \| u\|_{X(\Omega_\ell)}^2 \leq \tilde{\Lambda}(\{\Omega_\ell\}) \| u\|_{X(\Omega)}^2. 
\end{align}
\end{lemma}

To prove Lemma \ref{l:overlap}, we recall the following result.
\begin{lemma}\mythmname{Consequences of the definition of $\tilde{\Lambda}(\{\Omega_\ell\})$ \eqref{e:overlapDef} \cite[Lemma 4.2]{GSV1}, \cite[Lemma 3.3]{GS4}}\label{l:overlapH}
Let $r\in \{0,1,\dots\}$ and $v\in H^{r}(\Omega)$. Then
\beq\label{e:overlap2a}
\sum_{\ell=1}^N \N{v}^2_{H^r_k(\Omega_\ell\setminus\Gamma_p)} \leq \tilde{\Lambda}(\{\Omega_\ell\})  \N{v}^2_{H^r_k(\Omega\setminus \Gamma_p)}.
% \quad\tand\quad \sum_{\ell=1}^N \N{w}^2_{H^1_k(\Omega_\ell)} \leq \Lambda \N{w}^2_{H^1_k(\Omega)}.
\eeq
Furthermore, given $v_\ell\in \Hilbert_\ell$,
\beq\label{e:overlap_new}
\bigg\|
\sum_{\ell=1}^N v_\ell 
\bigg\|^2_{H^r_k(\Omega\setminus\Gamma_p)}
\leq 2\tilde{\Lambda}(\{\Omega_\ell\})  \sum_{\ell=1}^N \N{v_\ell}^2_{H^r_k(\Omega_\ell\setminus \Gamma_p)}.
\eeq
\end{lemma}

\bpf[Proof of Lemma \ref{l:overlap}]
By \eqref{e:overlap2a} with $w= (I-\Psi)u$,
\begin{align*}
\sum_{\ell=1}^N (\Hc k)^{-2p} \| (I-\Psi(\mathcal{P})) u\|^2_{H^1_k(\Omega_\ell)}
    \leq \tilde{\Lambda} (\Hc k)^{-2p}\|(I-\Psi(\mathcal{P})) u\|^2_{H^1_k(\Omega)}
\end{align*}
Similarly, 
\begin{align*}
\sum_{\ell=1}^N  \| \Psi(\mathcal{P}) u\|^2_{H^{p+1}_k(\Omega_\ell\setminus \Gamma_p))}
    \leq \tilde{\Lambda}\|\Psi(\mathcal{P}) u\|^2_{H^{p+1}_k(\Omega\setminus \Gamma_p))},
\end{align*}
and then \eqref{e:overlap1} follows from the combination of the last two displayed inequalities.

\epf

We also need an analogue of Lemma~\ref{l:overlapH} in negative Sobolev spaces.
\begin{lemma}\label{l:overlapNegative}
Let $1\leq m\leq p+1$, $\psi_\ell\in \overline{C^\infty}(\Omega\setminus \Gamma_p)\cap \overline{C^0(\Omega)}$ with $\supp \psi_\ell \subset \Omega_\ell$,
\beq\label{e:JeffHatesInterfaces}
\delta(\psi_\ell,\{D_\alpha\},p+1\big)
\geq \delta_\ell,
\eeq
%and
% $$
% \nu \cdot (A\nabla \psi_\ell)|_{\Gamma_{\rm hard}}=0,\qquad \nu \cdot (A_+(\nabla \psi_\ell)_+-A_-(\nabla \psi_{\ell})_-)|_{\Gamma_p}=0.
% $$
 Then for all $w \in \ZcapHspace{-m}$,
$$
\sum_{\ell=1}^N \| \psi_\ell w \|^2_{\ZcapHspace{-m}}\leq C \tilde{\Lambda} %(k\delta_\ell)^{-m}
\| w\|^2_{\ZcapHspace{-m}}. 
$$
\end{lemma}
\begin{proof}
By \eqref{e:JeffHatesInterfaces}, there is $C>0$ such that for all $\ell$,
\begin{equation} 
\label{e:boundNegative}
\begin{gathered}\|\psi_\ell w\|_{\ZcapHspace{-m}}\leq C\|w\|_{(\dot{H}_k^m(\Omega_\ell\setminus\Gamma_p)\cap H^{1}_{k}(\Omega_\ell))^*},\\
%H^2_{\partial\Omega}(\Omega_\ell):=\Big\{ u\in H^1_{\partial\Omega}(\Omega_\ell)\cap H^2(\Omega_\ell \setminus \Gamma_p)\,:\, \nu \cdot A\nabla u=0\text{ on }\Gamma_{\rm hard},\,\nu \cdot (A_+\nabla u_+-A_-\nabla u_-)|_{\Gamma_p}=0\Big\},\\
\text{ where } \quad \dot H_{k}^m(\Omega_\ell\setminus \Gamma_p):=\big\{u\in H_k^m(\Omega\setminus \Gamma_p)\,:\,\supp u\subset\overline{\Omega_\ell} \big\}
\end{gathered}
\end{equation}
    Let $\varphi_{\ell}\in \overline{C^\infty}(\Omega_\ell\setminus\Gamma_p)\cap H^{1}_k(\Omega_\ell)$ with  $\supp \varphi_\ell \cap \partial\Omega_{\ell}\setminus\partial\Omega=\emptyset$, $\|\varphi_\ell \|_{\dot H_{k}^m(\Omega_\ell\setminus \Gamma_p)}=1$, define $$\theta_\ell:=\|w\|_{(\dot H_{k}^m(\Omega_\ell\setminus\Gamma_p)\cap H_{\partial\Omega}^{1}(\Omega_\ell))^*},$$
    and set 
    $$
    \varphi:=\sum_\ell \theta_\ell \varphi_\ell.
    $$
    Then, by Lemma~\ref{l:overlapH}
    $$
    \|\varphi\|_{\ZcupHspace{m}}^2\leq \tilde{\Lambda} \sum_{\ell=1}^N\theta_\ell^2,
    $$
    and hence
    $$
    \|w\|_{\ZcapHspace{-m}}^2\geq \frac{|\langle w,\varphi\rangle |^2}{\|\varphi\|^2_{\ZcupHspace{m}}}\geq \tilde{\Lambda}^{-1}\frac{
\Big|\sum_{\ell}\theta_\ell \langle w,\varphi_\ell\rangle\Big|^2}{\sum_{\ell=1}^N\theta_\ell^2}= \tilde{\Lambda}^{-1}\frac{
\Big|\sum_{\ell}\|w\|_{(\dot H_{k}^{m}(\Omega_\ell\setminus\Gamma_p)\cap H^{1}_k(\Omega_\ell))^*} \langle w,\varphi_\ell\rangle\Big|^2}{\sum_{\ell=1}^N\|w\|_{(\dot H_{k}^{m}(\Omega_\ell\setminus \Gamma_p)\cap H^{1}_k(\Omega_\ell))^*}^2}.
    $$
    Taking the supremum over $\{\varphi_\ell\}_{\ell=1}^N$ then implies
    $$
    \sum_{\ell}\|w\|_{(\dot H_{k}^m(\Omega_\ell\setminus \Gamma_p)\cap H^{1}_k(\Omega_\ell))^*}^2\leq \tilde{\Lambda} \|w\|^2_{\ZcapHspace{-m}},
    $$
    and hence the lemma follows from~\eqref{e:boundNegative}
\end{proof}

\section{Proof of Theorem~\ref{t:main}}

We prove the result for left preconditioning. Since all the assumptions are invariant under taking the adjoint, the result for right preconditioning follows (see, e.g., \cite[Equation 1.8]{GS4}).

%(since they involve either the real part of the coefficients or the modulus of the imaginary part)

\subsection{Definition of the preconditioner}

Recall from \eqref{e:galerkinCoarse} that $\Qcoarse: \fine \to \coarse$ is defined by 
\begin{align}
    a(\Qcoarse u_h, v_{\Hc}) =a(u_h, v_{\Hc}) \quad\tfa v_{\Hc}\in \coarse,
\end{align}
and recall 
from \eqref{e:QellIntro} that
$\Qsub:\fine \to \subdomain$ is defined by
\begin{align}\label{e:subdomain}
    a_\ell(\Qsub u_h, v_{\ell,h}) =a(u_h, E_{\ell} v_{\ell,h}) \quad\tfa v_{\ell,h}\in \subdomain
\end{align}
where $E_{\ell}: H_{\partial\Omega}^1(\Omega_\ell)\to \cZ$ denotes extension by zero. 
The preconditioner $B_L^{-1}$ defined by 
\eqref{e:BLinverse} then 
satisfies $B_{L}^{-1}A= Q\iota$, where $Q:\cZ\to \fine$ is given by
\begin{align}\label{e:Q}
Q := \Qcoarse+\sum_{\ell=1}^N 
%\cR_\ell^*
\cI_\ell \big(\chiell \Qsub(I-\Qcoarse)\big).
 \end{align}

\subsection{Subdomain problems at the continuous level}

Recall that $P:\cZ\to \cZ^*$ is the operator defined by $a$ and $P_\ell:\cV_\ell \to \cV_\ell^*$ that defined by $a_\ell$. 
The subdomain problems defined by \eqref{e:subdomain} can then be written as
\begin{align}\label{e:subdomain2}
    \langle P_\ell \Qsub u_h, v_{\ell,h} \rangle = \langle P u_h , E_{\ell} v_{\ell,h}\rangle \quad \tfa v_{\ell,h} \in \subdomain.
\end{align}
Recall from \eqref{e:Qell} that the ``continuous-level counterparts" of $\Qsub: \cV_h\to \subdomain$, $Q^\ell:\fine\to \cV_\ell$,
are defined by 
$Q^\ell:= P_\ell^{-1} E_{\ell}^* P$.
% as follows. 
% Let $Q^\ell: \cV_h \to \cV_\ell$ be defined by 
% \begin{align}\label{e:subdomainCts}
%     \langle P_\ell Q^\ell u_h, v_{\ell} \rangle = \langle P u_h , E_{\ell} v_{\ell}\rangle \quad \tfa v_{\ell} \in H^1_{\partial\Omega}(\Omega_\ell)=:\cV_\ell;
% \end{align}
% %where $P_\ell : \cV_\ell \to (\cV_\ell)^*$; 
% i.e., 
% \begin{align}%\label{e:Qell}
%     P_\ell Q^\ell u_h = E_{\ell}^* P u_h \in (\cV_\ell)^* \quad \Rightarrow\quad\label{e:Qell}
% Q^\ell u_h = P_\ell^{-1} E_{\ell}^* P u_h\in \cV_\ell.
% \end{align}

\subsection{Start of the proof of Theorem \ref{t:main}}

We first observe that, by Lemma~\ref{l:coarse} together with the definition of $Q$~\eqref{e:Q}, 
$$
Q\Pi_{H_c}=\Pi_{H_c},
$$
where $\Pi_{H_c}$ denotes that $\cZ$-orthogonal projector onto $\coarse$. 
Hence, we may decompose $\fine=\Pi_{H_c}\fine\oplus (I-\Pi_{H_c})\fine$ and, written with respect to this decomposition,
$$
(Q-I)\iota=\begin{pmatrix} 0&\Pi_{H_c}(Q-I)\iota(I-\Pi_{H_c})\\ 0&(I-\Pi_{H_c})(Q-I)\iota(I-\Pi_{H_c})\end{pmatrix}
$$
Hence, to prove Theorem~\ref{t:main} it is enough to show that 
\begin{equation}
\label{e:theGoal}
\begin{gathered}
\begin{aligned}
&\|(I-\Pi_{H_c})(Q-I)\iota\|_{\cZ\to\cZ}
\\
&\leq C\Lambda\Bigg[\max_\ell \big\{\rho_\ell (k\delta_\ell)^{-1}\big\}\rho (H_ck)^{2p}+\max_\ell \rho_\ell (H_ck)^{2p}
 +\max_\ell \langle k\delta_\ell\rangle^{-M}(H_ck)^p\\
 &\qquad\qquad+ (h /H_c)^{p}  +  \max_\ell\mu_\ell  + \max_{\ell}((h_\ell k)^{p}\rho_{\ell}+(h_\ell k)^s)\Big(\rho_\ell(h_\ell k)^{p}+\rho_\ell(k\delta_\ell)^{-1}+1\Big)\Bigg]=: \e
,\end{aligned}
 \\
 \|(Q-I)\iota\|_{\cZ\to \cZ}\leq \e+C\Lambda\Big(\max_\ell \big(\rho_\ell (k\delta_\ell)^{-1}\big)\rho (H_ck)^{p}+\max_\ell \rho_\ell (H_ck)^{p}+\max_\ell \langle k \delta_\ell\rangle^{-M}\Big).\end{gathered}
\end{equation}
Indeed, once we have~\eqref{e:theGoal}, then 
$$
\|[(Q-I)\iota]^Ku\|_{\cZ}\leq C^K\e^{K-1}\|(Q-I)\iota\|_{\cZ\to \cZ}\|u\|_{\cZ}.
$$

Therefore, it remains to prove~\eqref{e:theGoal}.
We start by decomposing $Q-I$.  Using the facts that $\{\chiell\}_{\ell=1}^N$ is a partition of unity, $\cR_\ell^* \chiell \cR_\ell = \chiell$, and
 $\chiellg\equiv 1$ on $\supp \chiell$ (so that $\chiell =\chiell \chiellg$),  for all $u_h\in \Hilbert_h$,
\begin{align*}%\nonumber
u_h= \sum_{\ell=1}^N\cR_\ell^*\chi_{\ell}\cR_\ell u_h &= \sum_{\ell=1}^N\cR_\ell^*\Big(\cI_\ell(\chiell \cR_\ell u_h)+
(I-\cI_\ell)\chi_{\ell}\cR_\ell u_h\Big).
% &=\sum_{\ell=1}^N\cR_\ell^*\Big(\cI_\ell(\chiell\chiellg \cR_\ell u_h)+
% (I-\cI_\ell)\chiell\cR_\ell u_h\Big). 
%\label{e:magic1}
\end{align*}
By  \eqref{e:Q} and the last displayed equation (with $u_h=(I-\Qcoarse)v_h$), for $v_h\in \Hilbert_h$,
\beq
\begin{aligned}
(Q- I)v_h &= -(I-\Qcoarse)v_h+ \sum_{\ell=1}^N\cR_\ell^* \cI_\ell\big(\chiell Q^\ell_h (I-\Qcoarse)v_h\big)\\ 
&= \sum_{\ell=1}^N\cR_\ell^*\Big[- \cI_\ell (\chiell \cR_\ell(I-\Qcoarse)v_h)+ \cI_\ell\big(\chiell Q^\ell_h (I-\Qcoarse)v_h\big)\\
&\hspace{6.3cm} -(I-\cI_\ell)\chiell\cR_\ell(I-\Qcoarse)v_h\Big]\\
&= \sum_{\ell=1}^N\cR_\ell^*\Big[ \cI_\ell\big(\chiell \big(Q^\ell_h- \cR_\ell\big)(I-\Qcoarse)v_h\big)-(I-\cI_\ell)\chiell\cR_\ell(I-\Qcoarse)v_h\Big]\\
&= \sum_{\ell=1}^N\cR_\ell^*\Big[
\big(\chiell \big(Q^\ell- \cR_\ell\big)(I-\Qcoarse)v_h\big)
-\big(\chiell \big(Q^\ell- Q^\ell_h\big)(I-\Qcoarse)v_h\big)\\
&\hspace{2cm}
-(I-\cI_\ell)\big(\chiell \big(Q^\ell_h- \cR_\ell\big)(I-\Qcoarse)v_h\big)-(I-\cI_\ell)\chiell\cR_\ell(I-\Qcoarse)v_h\Big].
\label{e:start1}
\end{aligned}
\eeq

\subsection{Study of the local problems at the continuous level}

In this subsection, we prove two results operator $Q^\ell (I-\Qcoarse)$. The first -- Lemma \ref{l:ctsPart} -- splits 
$Q^\ell (I-\Qcoarse)$ into ``low'' and ``high" frequency components. 
The second -- Lemma \ref{l:ctsDifference} -- studies the difference between $Q^\ell (I-\Qcoarse)$ and $I-\Qcoarse$.

% we study the difference %(when restricted by $\chi_\ell$) 
% between the identity operator and the continuous solution operator $Q_\ell$  when acting on the error in the coarse space solve (i.e. $(I-\Qcoarse)v_h$ for $v_h\in \fine$).

\begin{lemma}\label{l:ctsPart} 
     Given $M>0$ there exists $C>0$ such that for all $\ell$, $\mathscr{L}_\ell    :\fine \to H^{p+1}(\Omega_{\ell})\cap H^1_{\partial\Omega}(\Omega_{\ell})$ and $\mathscr{H}_\ell:\fine \to  H^1_{\partial\Omega}(\Omega_\ell)$ such that for all $v_h\in \fine $,
    $$
     Q^\ell(I-\Qcoarse)v_h=\mathscr{L}_\ell v_h+\mathscr{H}_\ell v_h,
    $$
    with 
    \begin{equation}
    \begin{aligned}\label{e:ctsPartBound}
     \|\mathscr{L}_\ell v_h\|_{\ZcapHspace{p+1}(\Omega_{\ell})}&\leq
      C \rho_{\ell}
(H_ck)^{p}\sum_{j=1}^N \langle kd(\Omega_j,\Omega_{\ell})\rangle^{-M}\|(I-\Qcoarse)v_h\|_{X(\Omega_j)}\\
\|\mathscr{H}_\ell v_h\|_{H_k^1(\Omega_{\ell})}& \leq
      C (\rho_{\ell}(k\delta_\ell)^{-1}+1)
(H_ck)^{p}\sum_{j=1}^N \langle kd(\Omega_j,\Omega_{\ell})\rangle^{-M}\|(I-\Qcoarse)v_h\|_{X(\Omega_j)}\\
%\sum_{j=1}^N e^{-c\frac{d(\Omega_j,\Omega_\ell)k}{2H_c k}}\|(I-\Qcoarse)v_h\|_{X(\Omega_j)}
%\label{e:oldBbound}
       \end{aligned}
       \end{equation}
%and $(1-\chiellg)\mathscr{A}_\ell = (1-\chiellg)\mathscr{B}_\ell=0.$
       \end{lemma}

\begin{lemma}\label{l:ctsDifference} 
        Given $M>0$ there exists $C>0$ such that for all $\ell$, $\mathscr{A}_\ell    :\fine \to\cZ$  and $\mathscr{B}_\ell:\fine \to\cZ$ $\chiellg\mathscr{A}_\ell:\fine\to \ZcapHspace{p+1}$, $\chiellg\mathscr{B}_\ell:\fine\to \ZcapHspace{p+1}$ and for all $v_h\in \fine $,
    $$
    (Q^\ell-I)(I-\Qcoarse)v_h=\mathscr{A}_\ell v_h+\mathscr{B}_\ell v_h,
    $$
    with 
    \begin{align}\nonumber
&\|\chiellg\mathscr{A}_\ell v_h\|_{\ZcapHspace{p+1}}
\\ \label{e:oldAbound}
& \leq C \rho_\ell \Big( (k\delta_\ell )^{-1}\|\chiellgg(I-\Qcoarse)v_h\|_{\ZcapHspace{1-p}}
+ (\Hc k)^{2p}\sum_{j=1}^N \langle kd(\Omega_j,\Omega_{\ell})\rangle^{-M}\|(I-\Qcoarse)v_h\|_{X(\Omega_j)}
\Big),
\\ 
     &\|\chiellg\mathscr{B}_\ell v_h\|_{\ZcapHspace{p+1}}\leq
      C \langle k\delta_\ell\rangle^{-M}
(H_ck)^{p}\sum_{j=1}^N \langle kd(\Omega_j,\Omega_{\ell})\rangle^{-M}\|(I-\Qcoarse)v_h\|_{X(\Omega_j)}.
\label{e:oldBbound}
       \end{align}
       \end{lemma}

A key ingredient of the proofs of both Lemmas \ref{l:ctsPart} and \ref{l:ctsDifference} is the following result.

\begin{lemma}\label{l:sneaky}
Given $M>0$ there exists $C>0$ such that 
\begin{equation}\label{e:sneaky1}
\|P(I-\Qcoarse)v_h\|_{(H^1_{\partial\Omega}(\Omega_\ell))^*}
    \leq C(H_ck)^{p}\sum_{j} \langle d(\Omega_j,\Omega_{\ell})k\rangle^{-M}\|(I-\Qcoarse)v_h\|_{X(\Omega_j)},
\end{equation}
and 
\begin{equation}\label{e:sneaky2}
\|\chiellgggg P(I-\Qcoarse)v_h\|_{\ZcupHspace{-p-1}}
    \leq C(H_ck)^{2p}\sum_{j} \langle d(\Omega_j,\Omega_{\ell})k\rangle^{-M}\|(I-\Qcoarse)v_h\|_{X(\Omega_j)},
\end{equation}

\end{lemma}

\bpf%[Proof of Lemma \ref{l:sneaky}]
%Let $\Pi_S$ be defined by \eqref{e:PiS}.
 In this proof, we use the fact that $\|u\|_{\cZ}$ is equivalent to $\langle (\cP+2)u,u\rangle^{1/2}$.
% % i.e. with the norm given by 
% % $$
% % \|u\|_{\cZ}^2=\langle (\cP+2)u,u\rangle.
% % $$
 Let $\Pi_{\cZ}^c: \cZ \to \coarse$ denote the $\cZ$ orthogonal projection using this equivalent norm. 
% ; i.e., 
% %Recall that the $L^2$ orthogonal projection, $\Pi_{\cZ}^c: L^2 \to \coarse$ is defined by 
% \beq\label{e:L2projDef}
% \langle (I-\Pi_{\cZ}^c)v, w_{\Hc}\rangle=0 \quad\tfa w_{\Hc} \in \coarse.
% \eeq
% Recall that $\Pi_{\cZ}^c : H^1_0(\Omega)\to \coarse$ by, e.g., \cite[Proposition 22.21]{ErGu:21}.
Thus, by Galerkin orthogonality, for all $v_h\in \fine$,
$$
\big\langle (\Pi_{\cZ}^c)^* P(I-\Qcoarse)v_h , w\big\rangle = 0 \quad\tfa w \in \cZ,
$$
so that 
\beq\label{e:L2projClaim1}
P(I-\Qcoarse)v_h=(I- (\Pi_{\cZ}^c)^*)P(I-\Qcoarse)v_h \quad\tfa v_h\in \fine.
\eeq
To prove \eqref{e:sneaky1}, it is therefore sufficient to prove that 
\beq\label{e:Pi0approx1}
\|(I- (\Pi_{\cZ}^c)^*)Pu\|_{(H^1_{\partial}(\Omega_\ell))^*}\leq
C(H_ck)^{p}\sum_{j} \langle d(\Omega_j,\Omega_{\ell})k\rangle^{-N}\|u\|_{X(\Omega_j)},
\eeq
and to prove \eqref{e:sneaky2} it is sufficient to prove
\beq\label{e:Pi0approx1a}
\|\chiellgggg(I- (\Pi_{\cZ}^c)^*)Pu\|_{\ZcupHspace{-p-1}}\leq
C(H_ck)^{2p}\sum_{j} \langle d(\Omega_j,\Omega_{\ell})k\rangle^{-N}\|u\|_{X(\Omega_j)}.
\eeq
Let $\widetilde{\chi}_\ell\in C^\infty(\Omega)$ with $\supp(1-\widetilde{\chi}_{\ell})\cap \Omega_\ell =\emptyset$, 
$$
\sup\{ d(x,\Omega_\ell)\,:\, x\in \supp\widetilde\chi_\ell\}<\tfrac{1}{2k},
$$
and 
$$
\|\partial^\alpha \widetilde{\chi}_{\ell}\|_{L^\infty}\leq C_{\alpha}k^{|\alpha|}\quad\tfa \alpha.
$$
Since $\supp(1-\widetilde{\chi}_{\ell})\cap \Omega_\ell =\emptyset$,
\beq\label{e:basic1}
\|(I-(\Pi_{\cZ}^c)^*)Pu\|_{(H_{\partial\Omega}^1(\Omega_\ell))^*}\leq\|\widetilde{\chi}_\ell(I-(\Pi_{\cZ}^c)^*)Pu\|_{\cZ^*}
\eeq
and 
\beq\label{e:basic2}
\|\chiellgggg(I-(\Pi_{\cZ}^c)^*)Pu\|_{\ZcupHspace{-1-p}}\leq\|\widetilde{\chi}_\ell(I-(\Pi_{\cZ}^c)^*)Pu\|_{\ZcupHspace{-1-p}} .
\eeq
Furthermore,
\beq\label{e:almostLocal1}
1_{\Omega_\ell}(I-(\Pi_{\cZ}^c)^*)Pu=\sum_{j}1_{\Omega_\ell}\widetilde{\chi}_{\ell}(I-(\Pi_{\cZ}^c)^*)\chi_jP \chi_j^{(1)}  (\Psi u+(I-\Psi)u).
\eeq
Set
$$
\mathcal{F}_{\ell}:=\big\{j\,:\, d(\Omega_{\ell},\Omega_j)>k^{-1}\big\}, \qquad \mathcal{N}_{\ell}:=\big\{j\,:\, d(\Omega_{\ell},\Omega_j)\leq k^{-1}\big\}
$$
Then, by \eqref{e:almostLocal1} and the fact that $P: \ZcapHspace{p+1}\to \ZcupHspace{p-1}$ and $P: \cZ\to \cZ^*$ with bounded norms,
\begin{equation}
    \begin{aligned}[b]
&\|\widetilde{\chi}_\ell(I-(\Pi_{\cZ}^c)^*)Pu\|_{\ZcupHspace{-s}}\\
&\leq \sum_{j\in \mathcal{F}_\ell}\|\widetilde{\chi}_{\ell}(I-(\Pi_{\cZ}^c)^*)\chi_j\|_{\cZ^*\to \ZcupHspace{-s}} \|\chi_j^{(1)}u\|_{\cZ}\\
&\quad +\sum_{j\in \mathcal{N}_\ell}\Big(\|\tilde{\chi}_{\ell
}(I-(\Pi_{\cZ}^c)^*)\chi_j\|_{\ZcupHspace{p-1}\to \ZcupHspace{-s }}\|\chi_j^{(1)}\Psi u\|_{\ZcapHspace{p+1}}+\|\tilde{\chi}_{\ell}(I-(\Pi_{\cZ}^c)^*)\chi_j\|_{\cZ^*\to \ZcupHspace{-s}}\|\chi_j^{(1)}(I-\Psi)u\|_{\cZ}\Big). 
\label{e:JeffLikesThisPaper1}
\end{aligned}
\end{equation}
Notice that since
\begin{equation*}
    \|u\|_{H^1_k}\leq
    \| \Psi(\mathcal{P})u\|_{H_k^{1}}+\|(I-\Psi(\mathcal{P})u\|_{H_k^1}   
    \leq\|u\|_X,
\end{equation*}
we have
\beq\label{e:multchiell2}
\|\chi_j^{(1)}u\|_{\cZ}\leq 
%\big(1 + (k\delta_\ell)^{-1}\big)
C\| u\|_{X(\Omega_j)}.
\eeq
Since $\Pi_{\cZ}^c: \cZ\to \coarse$ is bounded, 
\beq\label{e:Pi0approx2}
\|I- (\Pi_{\cZ}^c)^*\|_{\cZ^*\to \ZcupHspace{-s}}
=
\|I- \Pi_{\cZ}^c\|_{\ZcapHspace{s}\to \cZ}\leq C(H_ck)^{s-1}.
\eeq
We now claim that \beq\label{e:Pi0approx3}
\|I- (\Pi_{\cZ}^c)^*\|_{\ZcupHspace{p-1}\to \ZcupHspace{-s}}\leq C(\Hc k)^{p+s-1},
\eeq
Indeed, $(\Pi_{\cZ}^c)^*=(\cP +2)\Pi_{\cZ}^c (\cP+2)^{-1}$ and hence
\begin{align*}
\|I-(\Pi_{\cZ}^c)^*\|_{\ZcupHspace{p-1}\to \cZ^*}=\|I-\Pi_{\cZ}^c\|_{\ZcapHspace{p+1}\to \cZ}\leq C(H_ck)^{p},
\end{align*}
Also, by, e.g., \cite[Lemma D.15]{AGS2},
for $1\leq s\leq p+1$,
\begin{align*}
\|I-(\Pi_{\cZ}^c)^*\|_{\cZ^*\to \ZcupHspace{-s}}=\|I-\Pi_{\cZ}^c\|_{\ZcapHspace{s}\to \cZ}\leq C(H_ck)^{s-1};
\end{align*}
the combination of these last two inequalities then implies \eqref{e:Pi0approx3}.
% since
% \begin{align*}
% \|I- \Pi_{\cZ}^c\|_{\ZcapHspace{s}\to \cZ}\leq (\Hc k)^s 
%     \end{align*}
    
Furthermore for any $N$ there is $C>0$ such that \beq\label{e:expDecay1}
\|\widetilde{\chi}_\ell(I-(\Pi_{\cZ}^c)^*)\chi_{j}
\|_{\cZ^*\to \ZcupHspace{-s}}
\leq C(d(\Omega_j,\Omega_\ell )k)^{-N}(H_ck)^{p+s-1}\quad \tfor j\in \mathcal{F}_\ell;
\eeq
    %The estimate~\eqref{e:expDecay1}
    this follows from~\cite[Theorem 7.2]{AGS2} or more precisely its proof after letting $d(\supp\chi,\supp \psi)\geq dk^{-1}$ and assuming that $h(\mathcal{T})k\leq c$ and accounting for the interface $\Gamma_p$. 
    %We have sketched in more detail the necessary changes to the proof in Appendix~\ref{a:pseudoLocal}
    
Combining 
\eqref{e:basic1} with 
\eqref{e:JeffLikesThisPaper1},
\eqref{e:multchiell2}, \eqref{e:Pi0approx2}, \eqref{e:Pi0approx3}, and  
\eqref{e:expDecay1}, we obtain that
\begin{align*}
\|(I-(\Pi_{\cZ}^c)^*)Pu\|_{(H^1_{\partial\Omega}(\Omega_\ell))^*}&\leq C\Big(\sum_{j\in\mathcal{F}_{\ell}}
(H_ck)^{p}
(d(\Omega_j,\Omega_{\ell})k)^{-N}\|u\|_{X(\Omega_j)}+\sum_{j\in\mathcal{N}_{\ell}}(H_ck)^{p}\|u\|_{X(\Omega_j)}
\Big)
\\
&\leq C(H_ck)^{p}\sum_{j} \langle d(\Omega_j,\Omega_{\ell})k\rangle^{-N}\|u\|_{X(\Omega_j)},
\end{align*}
i.e., \eqref{e:Pi0approx1}. The proof of the bound \eqref{e:Pi0approx1a} is analogous, starting from \eqref{e:basic2} instead of \eqref{e:basic1}.
\epf

\

       \bpf[Proof of Lemma \ref{l:ctsPart}]
       Observe that 
       \begin{align*}
       Q^\ell (I-Q^0_{H_c})v_h&=    P_\ell^{-1} E_{\ell}^* P  (I-\Qcoarse)v_h\\
       &=   \underbrace{\chiellgggg \Psi(\mathcal{P}_{\ell})P_\ell^{-1} E_{\ell}^* P  (I-\Qcoarse)}_{=:\mathscr{L}_\ell}v_h\\
       &\qquad +\underbrace{\chiellgggg(I-\Psi(\mathcal{P}_\ell))P_\ell^{-1} E_{\ell}^* P  (I-\Qcoarse)+ (1-\chiellgggg)P_\ell^{-1}E_\ell^*P(I-\Qcoarse)}_{=:\mathscr{H}_\ell}v_h.
       \end{align*}
Now, since $\supp \chiellgggg \Subset \Omega_\ell$,
$$\|\chiellgggg\Psi(\mathcal{P}_\ell)\|_{H^1_{\partial\Omega}(\Omega_\ell)\to \ZcapHspace{p+1}\cap H^1_{\partial\Omega}(\Omega_\ell)}\leq C.$$
By Lemma \ref{l:absorb},
       $$\|(1-\chiellgggg)P_\ell^{-1}\|_{(H^1_{\partial\Omega}(\Omega_\ell))^*\to H_{\partial\Omega}^1(\Omega_{\ell})}\leq C,$$ 
       and by Lemma \ref{l:HFsubdomain}
       $$
       \|\chiellgggg(I-\Psi(\cP_\ell))\cP_{\ell}^{-1}\|_{(H^1_{\partial\Omega}(\Omega_\ell))^*\to H_{\partial\Omega}^1(\Omega_{\ell})}\leq C\rho_\ell (k\delta_\ell)^{-1}.
       $$
Thus
       \begin{align*}
       \|\mathscr{L}_{\ell}v_h\|_{\ZcapHspace{p+1}}&\leq C\rho_{\ell} \|P(I-Q^0_{H_c})v_h\|_{(H_{\partial\Omega}^1(\Omega_{\ell}))^*},\\
       \|\mathscr{H}_{\ell}v_h\|_{H_k^{1}}&\leq  C(1+\rho_\ell(k\delta_\ell)^{-1})\|P(I-Q^0_{H_c})v_h\|_{(H_{\partial\Omega}^1(\Omega_{\ell}))^*}.
       \end{align*}
       The bounds \eqref{e:ctsPartBound} then follow from \eqref{e:sneaky1}. 
           \epf

              \

The rest of this subsection proves Lemma \ref{l:ctsDifference}.
By the definition of $Q^\ell$ \eqref{e:Qell}, the fact that $\chiell=\chiell\chiellg$, and the fact that $E_{\ell}$ denotes extension by zero,
\begin{align*}
&(Q^\ell-I) (I-\Qcoarse)v_h \\
&
= \big( P_\ell^{-1} E_{\ell}^* P - I \big) (I-\Qcoarse)v_h\\
&
=  \big( P_\ell^{-1} E_{\ell}^* P - \chiellgg \big) (I-\Qcoarse)v_h+(1-\chiellgg)(I-\Qcoarse)v_h\\
& =  \big( P_\ell^{-1} E_{\ell}^* P - P_\ell^{-1}P_\ell \chiellgg \big) (I-\Qcoarse)v_h
+(1-\chiellgg)(I-\Qcoarse)v_h\\
& 
=  \big( P_\ell^{-1} \chiellgggg P - P_\ell^{-1}P_\ell \chiellgg \big) (I-\Qcoarse)v_h
+  P_\ell^{-1} 
 (1-\chiellggg)E_{\ell}^*
 (1-\chiellgggg) 
P(I-\Qcoarse)v_h\\
&\qquad +(1-\chiellgg)(I-\Qcoarse)v_h
\\
& 
= \underbrace{ P_\ell^{-1} [\chiellgg, P] (I-\Qcoarse)}_{=:\mathscr{A}^1_\ell}v_h
+ \underbrace{ P_\ell^{-1} (\chiellgggg- \chiellgg)  P(I-\Qcoarse)}_{=:\mathscr{A}^2_\ell}v_h+\underbrace{(1-\chiellgg)(I-\Qcoarse)}_{=:\mathscr{A}^3_\ell}v_h
\\
&\qquad\qquad
+ \underbrace{ P_\ell^{-1} 
 (1-\chiellggg)E_{\ell} ^*
 (1-\chiellgggg) 
P(I-\Qcoarse)}_{=: \mathscr{B}_\ell}v_h.
\end{align*}
We then set $\mathscr{A}_\ell:= \mathscr{A}^1_\ell + \mathscr{A}^2_\ell+\mathscr{A}^3_\ell$. After noting $\nu \cdot A \nabla \chiellg=0$ on $\Gamma_{\rm{hard}}$ and $\nu\cdot (A_+(\nabla \chiellg)_+-A_-(\nabla \chiellg)_-)=0$ on $\Gamma_p$ implies  $\chiellg\mathscr{A}_{\ell}v_h,\,\chiellg\mathscr{B}_{\ell}\in \cD$, we have
by Lemma \ref{l:sneaky}, that
to prove Lemma \ref{l:ctsDifference} it is sufficient to prove that 
\begin{equation}\label{e:calA}
\sum_{j=1}^3\|\chiellg\mathscr{A}^j_\ell v_h\|_{\ZcapHspace{p+1} }
%+\|\mathscr{C}\|_{\cD^{p+1}_k } 
\leq C \rho_\ell  
\Big((k\delta_\ell)^{-1}\|\chiellggg(I-\Qcoarse)v_h\|_{\ZcapHspace{1-p}}
+\|\chiellgggg P(I-\Qcoarse)v_h\|_{\ZcupHspace{-p-1}}
\Big).
%(k\delta_\ell)^{-2p}\big(1+(k\delta_\ell)^{-2p}\big)
%(\Hc k)^{2p} \rho \|v_h\|_X
\end{equation}
% %\begin{equation}\label{e:calB}
% %\end{equation}
and
\begin{equation}\label{e:calC}
\|\chiellg\mathscr B_\ell\|_{\ZcapHspace{p+1}}
\leq 
 C \langle k \delta_\ell\rangle^{-\infty}
 %(\Hc k)^p 
\|%(1-\chiellggg) 
P(I-\Qcoarse)v_h\|_{{(H^1_{\partial\Omega}(\Omega_\ell))^*}}
%
%\| v_h\|_X.
  \end{equation}
  Note first that $\chiellg \mathscr{A}_{\ell}^3=0$, so it is enough to consider $j=1,2$. 
Our proof of \eqref{e:calA} uses the following commutator bound.

\begin{lemma}\label{l:commutator}
Suppose that $\chiell$
is as in Theorem \ref{t:main} and $\chiellg$ is as in \S\ref{s:cutoff}.
Given $-p\leq s\leq p$ there exists $C>0$ such that for all $\ell$
\begin{align*}
%\| [P,\chiell]\|_{(H^{s}_k)^* \to (\cZ^{s+1}_k)^*}    
%=
    \| [P,\chiell]\|_{\ZcapHspace{s+1} \to \ZcupHspace{s}}    
    %\| [P,\chiell]\|_{H^j_k \to H^{j-1}_k}
\leq C (k\delta_\ell)^{-1}\big(1+(k\delta_\ell)^{-1}\big).
\end{align*}
\end{lemma}

\bpf
By the assumed smoothness of $\chi_\ell$, $[P,\chi_\ell]u$ agrees with a differential operator
when $u\in \cZ^{s+1}_k$ for $0\leq s\leq p$. Then, by direct calculation and the bounds \eqref{e:chiBound} on $\chi_\ell$,
\begin{align*}
  | [P,\chiell]u |
    & \leq 
Ck^{-1}\Big(\|\nabla\chiell\|_{L^\infty} | k^{-1}\nabla u|
+\big(\|\nabla\chiell\|_{L^\infty}
+k^{-1}\|\Delta\chiell\|_{L^\infty}\big)
|u|
\Big)\\
& \leq C \Big( (k\delta_\ell)^{-1}|k^{-1}\nabla u|
+ \big( (k\delta_\ell)^{-1} + (k\delta_\ell)^{-2}\big)|u|\Big)
\end{align*}
which proves the result for $s=0$. Taking further derivatives proves the result for $1\leq s\leq p$. The result for $-p\leq s\leq 0$ follows by duality, using that $[P,\chiell]^*= [\chiell, P^*]$
\epf

\begin{lemma}\label{l:ERcomm}
Suppose that $\chiell$
is as in Theorem \ref{t:main} and $\chiellg$ is as in \S\ref{s:cutoff}. Suppose further that $\nu \cdot A\nabla \chiell=0$ on $\Gamma_{\rm hard}$ and $\nu \cdot (A_+\nabla(\chi_\ell)_+-A_-\nabla (\chi_\ell)_-)=0$ on $\Gamma_p$. 
Given $c_0>0$ there is $C>0$ such that if $k\delta_\ell \geq c_0$ 
then 
    \begin{align*}
\| \chiell  u \|_{\ZcapHspace{p+1}} \leq C\Big(\| \chiell P u \|_{\ZcupHspace{p-1}} + 
%(k\delta_\ell)^{-N}
%\big(1+(k\delta_\ell)^{-N}\big)
\| \chiellg u\|_{\ZcapHspace{1-p}}
\Big).       
    \end{align*}
\end{lemma}

\begin{proof}
Since the coefficients of $P$ are smooth, elliptic regularity implies that there exists $C>0$ such that 
$(P+C)^{-1}=O(1)_{\ZcupHspace{s-2} \to \ZcapHspace{s}}$ (see, e.g., \cite[Theorems 4.18 and 4.20]{Mc:00}).
    Since $\nu \cdot A\nabla\chiell=0$ on $\Gamma_{\rm hard}$ and $\nu \cdot (A_+\nabla(\chi_\ell)_+-A_-\nabla (\chi_\ell)_-)=0$ on $\Gamma_p$ and $u \in \cD$, we have $\chiell u\in \cD$ and hence
    $$
    (P+C) \chiell u = \chiell (P+C) u + [P,\chiell]u,
    $$
$$
\chiell u = (P+C)^{-1} \chiell (P+C) u + (P+C)^{-1} [P,\chiell]u.
$$
By the mapping property of $(P+C)^{-1}$,
the fact that $[P,\chiell] = [P,\chiell]\chiellg$, 
and Lemma \ref{l:commutator}, 
\begin{align*}
    \| \chiell u\|_{\ZcapHspace{m}} 
    &\leq C\Big(
\| \chiell (P+C)u \|_{\ZcupHspace{m-2}} + \| [P,\chiell]u \|_{\ZcupHspace{m-2}}
    \Big)\\
    & \leq C\Big(
\| \chiell (P+C)u \|_{\ZcupHspace{m-2}} + (k\delta_\ell)^{-1}\big(1+(k\delta_\ell)^{-1}\big)\| \chiellg u \|_{\ZcapHspace{m-1}}\Big)\\
& \leq C\Big(
\| \chiell Pu \|_{\ZcupHspace{m-2}} + \big(1+(k\delta_\ell)^{-1}\big)\big(1+(k\delta_\ell)^{-1}\big)\| \chiellg u \|_{\ZcapHspace{m-1}}
    \Big).
\end{align*}
% Similarly, for $m\geq 2$, 
% \begin{align*}
%     \| \chiell u\|_{(H^{m-2}_k)^*} 
% %     &\leq C\Big(
% % \| \chiell (P+C)u \|_{(\mathcal{Z}^{m}_k)^*} + \| [P,\chiell]u \|_{(\mathcal{Z}^{m}_k)^*}
% %     \Big)\\
% %     & \leq C\Big(
% % \| \chiell (P+C)u \|_{(\mathcal{Z}^{m}_k)^*} + (k\delta_\ell)^{-1}\big(1+(k\delta_\ell)^{-1}\big)\| \chiellg u \|_{(H^{m-1}_k)^*}\Big)\\
%  & \leq C\Big(
% \| \chiell Pu \|_{(\mathcal{Z}^m_k)^*} + \big(1+(k\delta_\ell)^{-1}\big)\big(1+(k\delta_\ell)^{-1}\big)\| \chiellg u \|_{(H^{m-1}_k)^*}
%     \Big).
% \end{align*}
The result then follows by iterating, using a sequence of cutoff functions ``between" $\chiell$ and $\chiellg$.  
%Need $\chi_\ell u$ in the domain ($u$ in domain, need normal derivative of $\chi_\ell$ equal zero on obstacle).
% Need $[P,\chiell]: H^{-s} \to (Z^{s+1})^*$. Since domain of adjoint is domain of operator, equivalent to $[P,\chiell]: Z^{s+1}\to H^s$ --ok.
\end{proof}

\

Lemma \ref{l:ERcomm} applied with $u= P_\ell^{-1} [\chiellgg, P](I-\Qcoarse)v_h$ implies that 
\begin{align*}\nonumber
\|\chiellg\mathscr{A}^{1} v_h\|_{\ZcapHspace{p+1}}&=
\| \chiell P_\ell^{-1} [\chiellgg, P](I-\Qcoarse)v_h\|_{\ZcapHspace{p+1}} \\
&
\leq C\Big(
\| \chiellg P P_\ell^{-1} [\chiellgg, P](I-\Qcoarse)v_h\|_{\ZcupHspace{p-1}}+\| \chiellg P_\ell^{-1} [\chiellgg, P](I-\Qcoarse)v_h\|_{\ZcapHspace{1-p}}\Big).
\end{align*}
Now $\chiellg P = \chiellg P_\ell$, $\chiellg [\chiellg,P]=0$, 
and $[\chiellgg,P]= [\chiellgg,P]\chiellggg$, 
so that
\begin{align*}
\|\chiellg\mathscr{A}^1 v_h\|_{\ZcapHspace{p+1}}
%\| \chiell P_\ell^{-1} [\chiellg, P](I-\Qcoarse)v_h\|_{H^N_k}
&\leq C
\| \chiellgg P_\ell^{-1} [\chiellgg, P](I-\Qcoarse)v_h\|_{\ZcapHspace{1-p}}\\
&\leq C
\| \chiellgg P_\ell^{-1}\|_{\ZcupHspace{-1-p}\to \ZcapHspace{1-p}}\| [\chiellgg, P]\|_{\ZcapHspace{1-p}\to \ZcupHspace{-1-p}}\|\chiellggg(I-\Qcoarse)v_h\|_{\ZcapHspace{1-p}}\\
&\leq C\rho_{\ell} (k\delta_\ell)^{-1}
\|\chiellggg(I-\Qcoarse)v_h\|_{\ZcapHspace{1-p}}.
\end{align*}
Similarly, by Lemma \ref{l:ERcomm} 
(applied with $u= P_\ell^{-1} (\chiellgggg- \chiellgg)  P(I-\Qcoarse)v_h$) and the facts that $\chi_\ell P = \chi_\ell P_\ell$ and that $\chi_\ell(\chiellggg-\chiellg)=0$,
\begin{align*}
 \|\chiellg\mathscr A^{2}v_h\|_{\ZcapHspace{p+1}}
&=\|\chiellg P_\ell^{-1} (\chiellgggg- \chiellgg)  P(I-\Qcoarse)v_h\|_{\ZcapHspace{p+1}} \\
&\leq 
C
\|\chiellgg P_\ell^{-1} (\chiellgggg- \chiellgg)  P(I-\Qcoarse)v_h\|_{\ZcapHspace{1-p}}\\
&=C
\|\chiellgg P_\ell^{-1} (1- \chiellgg) \chiellgggg P(I-\Qcoarse)v_h\|_{\ZcapHspace{1-p}}.
\end{align*}
Since $\chiellg P_\ell^{-1}(\chiellgggg- \chiellgg) : \ZcupHspace{-p-1}\to \ZcapHspace{p+1}$, 
 \begin{align*}
 \|\chiellg \mathscr A^2 v_h\|_{\ZcapHspace{p+1}}
&\leq 
C
\rho_\ell \|\chiellgggg P(I-\Qcoarse)v_h\|_{\ZcupHspace{-1-p}},
%\leq 
%C
%\rho_\ell (\Hc k)^{2p} \rho \|v_h\|_{X},
\end{align*}
which completes the proof of the bound \eqref{e:calA}.

We now prove \eqref{e:calC}. By definition
\begin{align*}
\|\chiellg\mathscr{B}_\ell\|_{\ZcapHspace{p+1}}
&=\|
\chiellg P_\ell^{-1} 
 (1-\chiellggg)E_\ell^*
 (1-\chiellgggg) 
P(I-\Qcoarse)v_h\|_{\ZcapHspace{p+1}}
\\
&\leq \|
\chiellg P_\ell^{-1} 
 (1-\chiellggg)E_\ell^*
\|_{(H^1_{\partial\Omega}(\Omega_\ell))^*\to \ZcapHspace{p+1}}
\|(1-\chiellgggg)P(I-\Qcoarse)v_h\|_{(H^1_{\partial\Omega}(\Omega_\ell))^*}
\\
&
=\|
 %(1-\chiellggg)
 E_{\ell}
 (1-\chiellggg)
 (P_\ell^*)^{-1} \chiellg
 \|_{\ZcupHspace{-p-1}\to H^1_{\partial\Omega}(\Omega_\ell)}
\|(1-\chiellgggg)P(I-\Qcoarse)v_h\|_{(H^1_{\partial\Omega}(\Omega_\ell))^*}.
\end{align*}
Since $(P_\ell^*)^{-1} \chiellg v\in H^1_{\partial\Omega}(\Omega_\ell)$, so is
$(1-\chiellggg)(P_\ell^*)^{-1} \chiellg v$. Thus 
$$
E_{\ell} 
 (1-\chiellggg)
 (P_\ell^*)^{-1} \chiellg  \in H^1_{\partial\Omega}(\Omega_\ell)
 $$
 with
 $$\|E_{\ell} 
 (1-\chiellggg)
 (P_\ell^*)^{-1}\chiellg \|_{\ZcupHspace{}{-p-1}\to H^1_k(\Omega_\ell)}
 =
 \|
 (1-\chiellggg)
 (P_\ell^*)^{-1} \chiellg \|_{\ZcupHspace{-p-1}\to H^1_k(\Omega_\ell)}.
 $$
Thus, 
\begin{align*}
\|\chiellg\mathscr{B}_\ell\|_{\ZcapHspace{p+1}}
&\leq
\|(1-\chiellggg)
 (P_\ell^*)^{-1} \chiellg \|_{\ZcupHspace{-p-1}\to H^1_{\partial\Omega}(\Omega_\ell)}
 \|P(I-\Qcoarse)v_h\|_{(H^1_{\partial\Omega}(\Omega_\ell))^*}
\\
&=
\|
\chiellg P_\ell^{-1} 
 (1-\chiellggg)\|_{(H^1_{\partial\Omega}(\Omega_\ell))^*\to \ZcapHspace{p+1}}
\|P(I-\Qcoarse)v_h\|_{(H^1_{\partial\Omega}(\Omega_\ell))^*}.
\end{align*}
Since $\supp(1-\chiellggg)$ is contained in the absorbing layer of $\Omega_\ell$ (by \eqref{e:lotsofggg}), by Lemma~\ref{l:absorb},
%Now prove bound  on $\mathcal C$ in %\eqref{e:calA}.
$$
\|\chiell P_\ell^{-1} 
 (1-\chiellggg)\|_{(H^1_{\partial\Omega}(\Omega_\ell))^*\to \ZcapHspace{p+1}} \leq C\| \chiellgg P_{\ell}^{-1} (1-\chiellggg)\|_{(H_{\partial\Omega}^1)^*\to \ZcapHspace{1-p}}\leq C \langle k \delta_\ell\rangle^{-\infty},
$$
and the bound \eqref{e:calC} then follows. We have therefore proved \eqref{e:calA} and \eqref{e:calC} and the proof of Lemma \ref{l:ctsDifference} is complete.

\subsection{End of the proof of Theorem \ref{t:main}}
Recall from \eqref{e:start1} that 
\begin{align*}
(Q-I)v_h&= \sum_{\ell=1}^N\cR_\ell^*\Big[
\big(\chiell \big(Q^\ell- \cR_\ell\big)(I-\Qcoarse)v_h\big)
-\big(\chiell \big(Q^\ell- Q^\ell_h\big)(I-\Qcoarse)v_h\big)\\
&\hspace{2cm}
-(I-\cI_\ell)\big(\chiell \big(Q^\ell_h- \cR_\ell\big)(I-\Qcoarse)v_h\big)-(I-\cI_\ell)\chiell\cR_\ell(I-\Qcoarse)v_h\Big].
\end{align*}
By Lemma \ref{l:ctsDifference},
\begin{align*}
&(Q-I)v_h= \sum_{\ell=1}^N\cR_\ell^*\Big[
\chiell\mathscr{A}_\ell + \chiell\mathscr{B}_\ell + \chiell\mathscr{C}_\ell + \mathscr{D}_\ell + \mathscr{E}_\ell
\Big]v_h
\end{align*}
where 
\beqs
\mathscr{C}_\ell:=- \big(Q^\ell- Q^\ell_h\big)(I-\Qcoarse),
\eeqs
\begin{align*}
    \mathscr{D}_\ell:= 
-(I-\cI_\ell)\chiell \big(Q^\ell_h- \cR_\ell\big)(I-\Qcoarse)
=-(I-\cI_\ell)\chiell\big(\mathscr{A}_\ell + \mathscr{B}_\ell+\mathscr{C}_{\ell}\big),
\end{align*}
and
\begin{align*}
\mathscr{E}_\ell:=
-(I-\cI_\ell)\chiell\cR_\ell(I-\Qcoarse).
\end{align*}

To prove~\eqref{e:theGoal}, we apply the triangle inequality and Lemma~\ref{l:overlapH} to obtain
\begin{align}
\|(I-\Pi_{H_c})(Q-I)v_h\|_{\cZ}^2 \leq C\Big( 
\sum_{\mathscr{F} \in \{\mathscr{A},\mathscr{B}\}}
(H_ck)^{2p}\sum_{\ell=1}^N \big\|\mathscr{F}_\ell v_h\big\|_{\ZcapHspace{p+1}}^2
+ \sum_{\mathscr{F}\in \{\mathscr{C},\mathscr{D},\mathscr{E}\}}
\sum_{\ell=1}^N\big\|\mathscr{F}_\ell  v_h\big\|_{\cZ}^2\Big),\label{e:IMinusPi}
\end{align}
and
\begin{align}
\label{e:noIMinusPi}
\|(Q-I)v_h\|_{\cZ}^2 \leq C\Big( 
\sum_{\mathscr{F} \in \{\mathscr{A},\mathscr{B}\}}
\sum_{\ell=1}^N \big\|\mathscr{F}_\ell v_h\big\|_{\ZcapHspace{p+1}}^2
+\sum_{\mathscr{F}\in \{\mathscr{C},\mathscr{D},\mathscr{E}\}}
\sum_{\ell=1}^N\big\|\mathscr{F}_\ell  v_h\big\|_{\cZ}^2\Big),
\end{align}
where, importantly, $C$ does not depend on $N$.

We first bound the terms involving $\mathscr{D}_\ell$ in terms of $\mathscr{A}_\ell$ and $\mathscr{B}_\ell$ and $\mathscr{C}_\ell$:~by Lemma~\ref{l:splitSuper} with $w=\mathscr{A}_\ell v_h+\mathscr{B}_\ell v_h$, , 
\beq\label{e:mathscrDbound}
\begin{aligned}[b]
\sum_{\ell=1}^N\|
\mathscr{D}_\ell v_h \|_{\cZ}^2
&\leq C (hk)^{2p} \sum_{\ell=1}^N \|
\chiellg(\mathscr{A}_\ell + \mathscr{B}_\ell)v_h\|_{\ZcapHspace{p+1}}^2+C\sum_{\ell=1}^N\mu^2_\ell \|\chiellg\mathscr{C}_{\ell}\|_{\cZ}^2\\
&\leq 
C  \bigg((hk)^{2p}\sum_{\ell=1}^N \|
\mathscr{A}_\ell v_h\|_{\ZcapHspace{p+1}}^2 
+(hk)^{2p}\sum_{\ell=1}^N
\|\mathscr{B}_\ell v_h\|_{\ZcapHspace{p+1}}^2+\sum_{\ell=1}^N\mu^2_\ell \|\mathscr{C}_{\ell}\|_{\cZ}^2\bigg).
\end{aligned}
\eeq
For $\chiellg\mathscr{A}_\ell$, by \eqref{e:oldAbound}, 
the Cauchy--Schwarz inequality, 
Lemmas \ref{l:overlap} and  \ref{l:overlapNegative}, and \eqref{e:FEMLowbound}, if $\rho(\Hc k)^{2p}\leq c$ then, since $\Lambda(\{\Omega_\ell\},k)\leq \Upsilon$,
\beq
\begin{aligned}[b]
&\sum_{\ell=1}^N\|\chiellg\mathscr{A}_\ell v_h\|_{\ZcapHspace{p+1}}^2\\
&\leq 
\sum_{\ell=1}^N\Bigg(
C\rho_\ell^2 (k\delta_{\ell})^{-2}
\|\chiellggg(I-\Qcoarse)v_h\|_{\ZcapHspace{1-p}}^2 +C\rho_{\ell}^2(H_ck)^{4p}\Big(\sum_{j=1}^N\langle kd(\Omega_j,\Omega_\ell)\rangle^{-M}\|(I-Q^0_{H_c})v_h\|_{X(\Omega_j)}\Big)^2\Bigg)
\\
&\leq 
C\sum_{\ell=1}^N\Bigg(
\rho_\ell^2 (k\delta_{\ell})^{-2}
\|\chiellggg(I-\Qcoarse)v_h\|_{\ZcapHspace{1-p}}^2 \\
&\qquad +C\rho_{\ell}^2(H_ck)^{4p}\Big(\sum_{j=1}^N\langle kd(\Omega_j,\Omega_\ell)\rangle^{-M}\Big)\Big(\sum_{j=1}^N\langle kd(\Omega_j,\Omega_\ell)\rangle^{-M}\|(I-Q^0_{H_c})v_h\|^2_{X(\Omega_j)}\Big)\Bigg)\\
&\leq 
\sum_{\ell=1}^N\Bigg(
C\rho_\ell^2 (k\delta_{\ell})^{-2}
\|\chiellggg(I-\Qcoarse)v_h\|_{\ZcapHspace{1-p}}^2  +C\rho_{\ell}^2(H_ck)^{4p}\sum_{j=1}^N\langle kd(\Omega_j,\Omega_\ell)\rangle^{-M}\|(I-Q^0_{H_c})v_h\|^2_{X(\Omega_j)}\Bigg)\\
&\leq 
C\Big(\max_\ell \big(\rho_\ell^2 (k\delta_\ell)^{-2}\big)\|(I-Q^0_{H_c})v_h\|^2_{\ZcapHspace{1-p}}+\max_\ell \rho_\ell^2 (H_ck)^{4p}\|(I-Q^0_{H_c})v_h\|^2_{X}\Big)\\
&\leq 
C\Big(\max_\ell \big(\rho_\ell^2 (k\delta_\ell)^{-2}\big)(H_ck)^{2p}\rho^2\|(I-\Pi_{H_c})v_h\|^2_{\cZ}+\max_\ell \rho_\ell^2 (H_ck)^{2p}\|(I-\Pi_{H_c})v_h\|^2_{\cZ}\Big)
%&\leq 
%C\Big(\max_\ell \big(\rho_\ell^2 (k\delta_\ell)^{-2}\big)(\rho(H_ck)^{2p})^2+\max_\ell \rho_\ell^2 (H_ck)^{4p}\Big)\|v_h\|^2_{X}.
\label{e:mathscrAbound}
\end{aligned}
\eeq
For $\chiellg\mathscr{B}_\ell$, by \eqref{e:oldBbound}, 
\eqref{e:overlap1}, 
and \eqref{e:firstComp}, if $\rho (\Hc k)^{2p}\leq c$ then, again using that $\Lambda(\{\Omega_\ell\},k)\leq \Upsilon$,
\beq
\begin{aligned}[b]
    \sum_{\ell=1}^N\|\chiellg\mathscr{B}_\ell 
    v_h
    \|_{\ZcapHspace{p+1}}^2
&\leq
  C \max_\ell \langle k \delta_\ell\rangle^{-M}
(H_ck)^{2p}
\sum_{\ell=1}^N
\sum_{j=1}^N \langle kd(\Omega_j,\Omega_{\ell})\rangle^{-M}\|(I-\Qcoarse)v_h\|^2_{X(\Omega_j)}\\
&\leq
  C \max_\ell \langle k \delta_\ell\rangle^{-M}
(H_ck)^{2p}
\sum_{j=1}^N \|(I-\Qcoarse)v_h\|^2_{X(\Omega_j)}\\
&\leq
  C  \max_\ell \langle k \delta_\ell\rangle^{-M}
(H_ck)^{2p}
 \|(I-Q_{H_c}^0)v_h\|^2_{X}\leq 
  C\max_\ell \langle k \delta_\ell\rangle^{-M}
 \|(I-\Pi_{H_c})v_h\|^2_{\cZ}.
 \end{aligned}
\label{e:mathscrBbound}
\eeq
For $\mathscr{E}_\ell$, by \eqref{e:superapprox}, 
\eqref{e:overlap1}, and \eqref{e:firstComp},  if 
$(\Hc k)^{2p}\rho\leq c$ then
\begin{equation}
\begin{aligned}[b]
 \sum_{\ell=1}^N 
\big\| (I-\cI_\ell)\chiell(I-\Qcoarse)v_h\big\|_{H^1_k(\Omega_\ell)}^2
&\leq C\sum_{\ell=1}^N  
\Big( (h_\ell k)^p  + (\Hc k)^p \mu_\ell \Big)^2
%\bigg(\mu_\ell
%+ \Big(\frac{h_\ell}{\Hc}\Big)^p\Big( %\frac{1}{k\delta_\ell}\Big)^{p+1} %\bigg)^2 
\| (I- \Qcoarse)v_h\|_{X(\Omega_\ell)}^2\\
&
\leq 
C \max_\ell 
\Big( (h_\ell k)^p + (\Hc k)^p \mu_\ell \Big)^2
%\bigg(
% \mu_\ell
% + \Big(\frac{h_\ell}{\Hc}\Big)^p\Big( \frac{1}{k\delta_\ell}\Big)^{p+1} \bigg)^2
\|(I-\Qcoarse)v_h\|_X^2
\\
&
\leq 
C \max_\ell
\Big( (h_\ell k)^p  + (\Hc k)^p \mu_\ell \Big)^2
(H_ck)^{-2p}\|(I-\Pi_{H_c})v_h\|_{\cZ}^2.
\end{aligned}
\label{e:mathscrEBound}
\end{equation}
Finally, we estimate $\mathscr{C}_{\ell}$. Since $d(\supp \chiellg,\partial\Omega_\ell\setminus \partial\Omega)\geq c\delta_\ell$, we may apply the ``local quasioptimality modulo low frequencies'' result of~\cite[Theorem 4.2]{AGS1} (see also \cite[Theorem 1.2 and Example 3.1]{AGS1}), to obtain
\begin{align*}
&\|\chi_\ell( Q^\ell -Q_{h}^\ell)(I-Q^0_{H_c})v_h\|_{\cZ}\\
&\leq C\inf_{w_{h,\ell}\in\subdomain}\|\chiellg Q^{\ell}(I-Q^0_{H_c})v_h-w_{h,\ell}\|_{\cZ}+ C\|\chiellg( Q^\ell -Q_{h}^\ell)(I-Q^0_{H_c})v_h\|_{\ZcupHspace{-p}}\\
&\leq C\inf_{w_{h,\ell}\in\subdomain}\|\chiellg (Q^{\ell}-I)(I-Q^0_{H_c})v_h-w_{h,\ell}\|_{\cZ}+C\inf_{w_{h,\ell}\in\subdomain}\|\chiellg (I-Q^0_{H_c})v_h-w_{h,\ell}\|_{\cZ}\\
&\qquad \quad+ C\|\chiellg( Q^\ell -Q_{h}^\ell)(I-Q^0_{H_c})v_h\|_{\ZcupHspace{-p}}.
% &\leq C(h_\ell k)^p\Big(\|\mathscr{A}_{\ell}^{(1)}v_h\|_{\ZcapHspace{p+1}}+\|\mathscr{B}_{\ell}^{(1)}v_h\|_{\ZcapHspace{p+1}}\Big)+
%  \Big( (h_\ell k)^p (k\delta_\ell)^{-p-1} + (\Hc k)^p \mu_\ell \Big) \| (I-Q^0_{H_c})v_h\|_{X(\Omega_\ell)}\\
% %C\inf_{w_{h,\ell}\in\subdomain}\|\chiellg (I-Q^0_{H_c})v_h-w_{h,\ell}\|_{\cZ}\\
% &\qquad \quad+ C\|\chi_\ell( Q^\ell -Q_{h}^\ell)(I-Q^0_{H_c})v_h\|_{\ZcupHspace{-p}}
\end{align*}
Now, by Lemmas~\ref{l:ctsPart} and~\ref{l:ctsDifference}, \eqref{e:superapprox}, \eqref{e:chiBound}, and Lemma~\ref{l:subdomainFEM}, for any $0<s<\frac{1}{2}$,
\begin{align*}
&\|\chi_\ell( Q^\ell -Q_{h}^\ell)(I-Q^0_{H_c})v_h\|_{\cZ}\\
&\leq C(h_\ell k)^p\Big(\|\chiellg\mathscr{A}_{\ell}v_h\|_{\ZcapHspace{p+1}}+\|\chiellg\mathscr{B}_{\ell}v_h\|_{\ZcapHspace{p+1}}\Big)+
 \Big( (h_\ell k)^p  + (\Hc k)^p \mu_\ell \Big) \| (I-Q^0_{H_c})v_h\|_{X(\Omega_\ell)}\\
%C\inf_{w_{h,\ell}\in\subdomain}\|\chiellg (I-Q^0_{H_c})v_h-w_{h,\ell}\|_{\cZ}\\
&\qquad \quad+ C\big((h_\ell k)^p\rho_{\ell}+(h_\ell k)^{s}\big)\inf_{w_{h,\ell}\in \subdomain}\| Q^\ell(I-Q^0_{H_c})v_h-w_{h,\ell}\|_{H_k^1(\Omega_{\ell})}\\
&\leq C(h_\ell k)^p\Big(\|\chiellg\mathscr{A}_{\ell}v_h\|_{\ZcapHspace{p+1}}+\|\chiellg\mathscr{B}_{\ell}v_h\|_{\ZcapHspace{p+1}}\Big)+
 \Big( (h_\ell k)^p  + (\Hc k)^p \mu_\ell \Big) \| (I-Q^0_{H_c})v_h\|_{X(\Omega_\ell)}\\
%C\inf_{w_{h,\ell}\in\subdomain}\|\chiellg (I-Q^0_{H_c})v_h-w_{h,\ell}\|_{\cZ}\\
&\qquad \quad+ C\big((h_\ell k)^p\rho_{\ell}+(h_\ell k)^{s}\big)\Big((h_\ell k)^p\|\mathcal{L}_\ell v_h\|_{\ZcapHspace{p+1}}+\|\mathcal{H}_\ell v_h\|_{H_k^1}\Big)\\
&\leq C(h_\ell k)^p\Big(\|\chiellg\mathscr{A}_{\ell}v_h\|_{\ZcapHspace{p+1}}+\|\chiellg\mathscr{B}_{\ell}v_h\|_{\ZcapHspace{p+1}}\Big)+
 \Big( (h_\ell k)^p  + (\Hc k)^p \mu_\ell \Big) \| (I-Q^0_{H_c})v_h\|_{X(\Omega_\ell)}\\
%C\inf_{w_{h,\ell}\in\subdomain}\|\chiellg (I-Q^0_{H_c})v_h-w_{h,\ell}\|_{\cZ}\\
&\qquad \quad+ C\big((h_\ell k)^p\rho_{\ell}+(h_\ell k)^{s}\big)\Big(\rho_\ell(h_\ell k)^p+(k\delta_\ell)^{-1}\rho_\ell+1\Big)(H_ck)^p\sum_{j=1}^N\langle kd(\Omega_j,\Omega_\ell)\rangle^{-M}\|(I-Q^0_{H_c})v_h\|_{X(\Omega_j)}.
\end{align*}
Hence, arguing as in~\eqref{e:mathscrAbound} and~\eqref{e:mathscrBbound},
\begin{equation} 
\label{e:mathscrCBound}
\begin{aligned}[b]
&\sum_{\ell=1}^N\|\chi_\ell( Q^\ell -Q_{h}^\ell)(I-Q^0_{H_c})v_h\|^2_{\cZ}\\
 &\leq C \max_\ell (h_\ell /H_c)^{2p}\Big(\max_\ell \big(\rho_\ell^2 (k\delta_\ell)^{-2}\big)(\rho(H_ck)^{2p})^2+\max_\ell \rho_\ell^2 (H_ck)^{4p}\Big)\|(I-\Pi_{H_c})v_h\|^2_{\cZ}\\
 &\qquad +  
  C\max_\ell \langle k \delta_\ell\rangle^{-M}
\max_\ell (h_\ell k)^{2p}
 \|(I-\Pi_{H_c})v_h\|^2_{\cZ}+C\max_\ell\Big( (h_\ell /H_c)^{2p}  +  \mu^2_\ell \Big) \| (I-\Pi_{H_c})v_h\|^2_{\cZ}\\
&\qquad \quad+ C \max_{\ell}\Big((h_\ell k)^{2p}\rho^2_{\ell}+(h_\ell k)^{2s}\Big)\Big(\rho_\ell^2(h_\ell k)^{2p}+\rho_\ell^2(k\delta_\ell)^{-2}+1\Big)\|(I-\Pi_{H_c})v_h\|^2_{\cZ}.
\end{aligned}
\end{equation}
Putting together~\eqref{e:IMinusPi}, \eqref{e:mathscrDbound}, \eqref{e:mathscrAbound} \eqref{e:mathscrBbound}, \eqref{e:mathscrEBound}, and~\eqref{e:mathscrCBound}, we obtain
\begin{equation}
    \label{e:betterBound}
\begin{aligned}
    &\|(I-\Pi_{H_c})(Q-I)v_h\|_{\cZ}\\
    &\leq C\Bigg[\max_\ell \big(\rho_\ell (k\delta_\ell)^{-1}\big)\rho (H_ck)^{2p}+\max_\ell \rho_\ell (H_ck)^{2p}\\
    &\qquad \quad
 +\max_\ell\langle k\delta_\ell\rangle^{-M}(H_ck)^p+ (h /H_c)^{p}  +  \max_\ell \mu_\ell  \\
%C\inf_{w_{h,\ell}\in\subdomain}\|\chiellg (I-Q^0_{H_c})v_h-w_{h,\ell}\|_{\cZ}\\
&\qquad \quad+  \max_{\ell}\big((h_\ell k)^p\rho_{\ell}+(h_\ell k)^{s}\big)\Big(\rho_\ell(h_\ell k)^{p}+\rho_\ell(k\delta_\ell)^{-1}+1\Big)\Bigg]\|(I-\Pi_{H_c})v_h\|_{\cZ}.
\end{aligned}
\end{equation}
Also, putting together~\eqref{e:noIMinusPi}, \eqref{e:mathscrDbound}, \eqref{e:mathscrAbound} \eqref{e:mathscrBbound}, \eqref{e:mathscrEBound}, and~\eqref{e:mathscrCBound}, we obtain 
\begin{equation} 
\label{e:basicBound}
\begin{aligned}[b]
    \|(Q-I)v_h\|_{\cZ}&\leq C\Bigg(\max_\ell \big(\rho_\ell (k\delta_\ell)^{-1}\big)\rho (H_ck)^{p}+\max_\ell \rho_\ell (H_ck)^{p}\\
    &\qquad+\max_\ell \langle k \delta_\ell\rangle^{-M}+ 
 (h /H_c)^{p}  + \max_\ell\mu_\ell\\
&\qquad \quad+  \max_{\ell}\big((h_\ell k)^p\rho_{\ell}+(h_\ell k)^{s}\big)\Big(\rho_\ell(h_\ell k)^{p}+\rho_\ell(k\delta_\ell)^{-1}+1\Big)\Bigg)\|(I-\Pi_{H_c})v_h\|_{\cZ}.
\end{aligned}
\end{equation}
The estimates~\eqref{e:betterBound} and~\eqref{e:basicBound} then imply~\eqref{e:theGoal}, and the proof is complete.
\appendix

\section{Coercive estimates in the absorbing layer}\label{app:coercive}

\begin{lemma}
\label{l:absorbMe}
  Let $c_1>0$. Then there  is $c_0>0$ such that for all $k>0$, all $U\subset \mathbb{R}^d$ bounded with Lipschitz boundary and $\Gamma\subset U$ a Lipschitz hypersurface, 
  all
   $A\in L^\infty(U;\mathbb{M}(d\times d))$ symmetric matrices, $b\in L^\infty(U\setminus \Gamma);\mathbb{C}^d)$ and $n\in L^\infty(U;\mathbb{C})$, and $\theta\in \R$
   with 
    \begin{equation}
        \label{e:symbol-coercivity}
    q(x,\xi):= \Re\Big(e^{i\theta} \big(A^{ij}(x)\xi_i\xi_j + b^i(x)\xi_i-n(x)\big)\Big)>c_1(1+|\xi|^2)\quad \tfa   x\in U\setminus \Gamma,\,\xi\in\mathbb{R}^d,
    \end{equation}
    the following is true. 
For all $u\in H^1(U)$ 
\begin{gather*}
Q(u)\geq c_0\|u\|_{H_k^1(U)}^2\\
\text{where}\\
Q(u):=\Re e^{i\theta}\Big(\langle Ak^{-1}\nabla u,k^{-1}\nabla u\rangle_{L^2(U)}-\tfrac{i}{2}\langle b\cdot k^{-1}\nabla u,u\rangle_{L^2(U)}+\tfrac{i}{2}\langle b u,k^{-1}\nabla u\rangle_{L^2(U)}-\langle n u,u\rangle_{L^2(U)}\Big).
\end{gather*}
\end{lemma}
\bpf
Set
$$
G(x)
:=
\Re
\big(e^{i\theta}A(x)\big),\quad 
\beta(x)
:=
\Re
\bigl( e^{i\theta}b(x)\bigr),
\quad\tand\quad
\gamma(x)
:=
-\Re
\bigl(e^{i\theta}n(x)\bigr)
$$
so that 
\begin{equation}
q(x,\xi)
=
\xi^T G(x)\xi
+
\beta(x)\cdot\xi
+
\gamma(x).
\label{e:q-real-form}
\end{equation}
Since $A$ is symmetric, \(G\) is a real symmetric
matrix. Moreover, by \eqref{e:symbol-coercivity}, there is $c_0>0$ such that $c_1>2c_0$ and
$\xi^T G(x)\xi
\geq
2c_0|\xi|^2.$
Set
$
G_1(x)
:=
G(x)-c_0 I
$ so that
\begin{equation}
G_1(x)
\geq c_0 I
>
0.
\label{e:G1-positive}
\end{equation}
Define
$$
v(x)
:=
\frac12G_1(x)^{-1}\beta(x)
\quad\tand\quad
\delta(x)
:=
\gamma(x)-c_0
-
\frac14
\beta(x)^T G_1(x)^{-1}\beta(x).
$$
Completing the square gives
\begin{equation}
q(x,\xi)
=
c_0\bigl(\xi^T\xi+1\bigr)
+
(\xi+v(x))^T
G_1(x)
(\xi+v(x))
+
\delta(x) \quad\tfa x\in U, \,\xi \in \R^d.
\label{e:completed-square}
\end{equation}
On the other hand, \eqref{e:symbol-coercivity} and the fact that $c_1>2c_0$ implies that
$$
q(x,\xi)
-
c_0\bigl(|\xi|^2+1\bigr)
\geq
c_0\bigl(|\xi|^2+1\bigr)
\quad \tfa x\in U,\,\xi \in \R^d.
$$
Evaluating this inequality at \(\xi=-v(x)\) (which is allowed since $v$ is real) and using
\eqref{e:completed-square} gives
\begin{equation}
\delta(x)
\geq
c_0
\bigl(|v(x)|^2+1\bigr)
\geq
c_0 \quad\tfa x\in U.
\label{e:delta-positive}
\end{equation}
Now let
$$
C(x):=G_1(x)^{1/2},
\quad\tand\quad
d(x):=C(x)v(x).
$$
With this notation \eqref{e:completed-square} becomes
\begin{equation}
q(x,\xi)
=
c_0\bigl(\xi^T\xi+1\bigr)
+
|C(x)\xi+d(x)|^2
+
\delta(x)
\quad\tfa \xi \in \R^d.
\label{e:symbol-factorisation}
\end{equation}
Set
$z:=-ik^{-1}\nabla u$ 
and write
$$
e^{i\theta}A
=
G+iH,
$$
where \(G\) and \(H\) are real symmetric matrices. Since
\(z^*Hz\) is real,
\begin{equation}
\Re\Big(
e^{i\theta}
A^{ij}
k^{-1}\partial_i u\,
\overline{k^{-1}\partial_j u}
\Big)
=
G^{ij}z_i\overline{z_j}.
\label{e:second-order-part}
\end{equation}

Next, we consider the symmetrized first-order term. Write
$$
e^{i\theta}b=\beta+i\alpha,
\qquad
\alpha:=\Im(e^{i\theta}b),
$$
where \(\alpha\) and \(\beta\) are real-valued. Observe then that
\begin{equation} \label{e:first-order-part}
\begin{aligned}
&\Re\left(
-\tfrac{i}{2}e^{i\theta}
  \langle b\cdot k^{-1}\nabla u,u\rangle_{L^2(U)}
+\tfrac{i}{2}e^{i\theta}
  \langle bu,k^{-1}\nabla u\rangle_{L^2(U;\mathbb C^d)}
\right)
\\
&\qquad
=
\tfrac12\Re\left(
\langle(\beta+i\alpha)\cdot z,u\rangle_{L^2(U)}
+
\langle(\beta+i\alpha)u,z\rangle_{L^2(U}
\right)\\
&\qquad
=
\tfrac12\Re\left(
\overline{\langle (\beta-i\alpha)u,z\rangle_{L^2(U)}}
+
\langle(\beta+i\alpha)u,z\rangle_{L^2(U)}
\right)\\
&\qquad
= \Re \langle \beta u,z\rangle=\Re \langle \beta\cdot z,u\rangle.
\end{aligned}
\end{equation}

Because
$
C^T C=G_1$ and $
C^T d
=
G_1v
=
\frac12\beta$,
\begin{align}
\|Cz+du\|_{L^2(U)}^2
&=
\int_{U}
G_1^{ij}z_i\overline{z_j}\,dx
+
\Re
\int_{U}
\beta\cdot z\,
\overline{u}\,dx
+
\int_{U}
|d|^2|u|^2\,dx.
\label{e:square-identity}
\end{align}
Combining
\eqref{e:symbol-factorisation}--\eqref{e:square-identity}
gives
\begin{align}
Q(u)
&=
c_0\left(\|k^{-1}\nabla u\|_{L^2(U)}^2
+
\|u\|_{L^2(U)}^2
\right)+
\|-iCk^{-1}\nabla u+du\|_{L^2(U)}^2
%\notag
%\\
%&\qquad
+\int_{U}
\delta(x)|u(x)|^2\,dx.
\label{e:form-factorisation}
\end{align}
Finally, using \eqref{e:delta-positive} in
\eqref{e:form-factorisation}, we obtain that
\begin{align*}
Q(u)
&\geq c_0
\|k^{-1}\nabla u\|_{L^2(U)}^2
+
c_0
\|u\|_{L^2(U)}^2.
\end{align*}
and the result follows.
\epf

\section*{Acknowledgements}

All three authors were supported by ERC Synergy Grant ``PSINumScat"  (101167139).
JG was 
supported by EPSRC grants EP/V001760/1 and EP/V051636/1 and Leverhulme Research Project Grant RPG-2023-325.

This work was granted access to the HPC resources of GENCI–TGCC, CEA, under the allocation 2026-AD010618004.

ChatGPT version 5.6 was used to proofread the paper and to help with the algebraic manipulations in Appendix \ref{app:coercive}.

\bibliographystyle{plain}
\bibliography{combined.bib}

\end{document}